\documentclass[11pt]{article}

\usepackage{mathtools}
\usepackage{amssymb, amsfonts,amsthm}
\usepackage[a4paper]{geometry}
\usepackage{lmodern}
\usepackage[T1]{fontenc}
\usepackage[utf8]{inputenc}
\usepackage{titlesec}
\usepackage{soul}
\usepackage{authblk}
\usepackage{graphicx}
\usepackage{ stmaryrd } 
\usepackage{tikz}
\usepackage{tdsfrmath} 
\usetikzlibrary{hobby} 
\usetikzlibrary{patterns}
\usepackage{lipsum}
\usepackage{marginnote} 
\usepackage[numbers,sort]{natbib}
\usepackage[english]{babel}
\usepackage{hyperref}
\usepackage{subcaption} 
\usepackage[open,level=1]{bookmark}

\usepackage{fullpage}

\NewDocumentCommand{\nsubset}{}{\not\subset}

\newcommand{\fdag}[1]{\mathcal{F}_{#1}}

\DeclareMathOperator{\Over}{\normalfont{\textbf{Over}}}
\DeclareMathOperator{\Perco}{\normalfont{\textbf{Perco}}}

\DeclareMathOperator{\Centered}{\normalfont{\textbf{Inside}}}
\DeclareMathOperator{\Uncentered}{\normalfont{\textbf{Outside}}}

\DeclareMathOperator{\Stab}{\normalfont{\textbf{Stab}}}

\newcommand\smallO{
\mathchoice
{{\scriptstyle\mathcal{O}}}
{{\scriptstyle\mathcal{O}}}
{{\scriptscriptstyle\mathcal{O}}}
{\scalebox{.7}{$\scriptscriptstyle\mathcal{O}$}}
}

\numberwithin{equation}{section}
\theoremstyle{plain}
\newtheorem{theorem}{Theorem}[section]
\newtheorem{proposition}[theorem]{Proposition}
\newtheorem{lemma}[theorem]{Lemma}

\title{Construction of ergodic \textsc{idla} forests in $\mathbb{Z}^d$}

\author[1]{\textsc{Nicolas Chenavier}}
\author[2]{\textsc{David Coupier}}
\author[1]{\textsc{Keenan Penner}}
\author[3]{\textsc{Arnaud Rousselle}} 
\affil[1]{Université du Littoral Côte d'Opale, UR 2597, LMPA, Laboratoire de Mathématiques Pures et Appliquées Joseph Liouville,
62100 Calais, France.}
\affil[2]{Institut Mines Télécom Nord Europe, Cité Scientifique, 59655 Villeneuve d'Ascq, France.}
\affil[3]{Universit\'e Bourgogne Europe, CNRS, IMB UMR 5584,
F-21000 Dijon, France.}
\date{}

\begin{document}

\maketitle

\begin{abstract}
We prove the existence of infinite-volume IDLA forests in $\Z^d$, with $d\geq 2$, based on a multi-source IDLA protocol. Unlike IDLA aggregates, the laws of the IDLA forests studied here depend on the trajectories of particles, and thus do not satisfy the famous \textit{Abelian property}. Their existence is due to a stabilization result (Theorem~\ref{thm: forest stabilization}, our main result) that we establish using percolation tools. Although the sources are infinitely many, we also prove that each of them play the same role in the building procedure, which results in an ergodicity property for the IDLA forests (Theorem~\ref{thm: forest properties}).
\end{abstract}




\section{Introduction}
\label{sec: intro bis}

The Internal Diffusion Limited Aggregation (\textsc{idla}) model gives a protocol to build random aggregates $(A_n)_{n\geq 0}$ recursively in $\mathbb{Z}^d$. Initially, we assume that $A_0 = \emptyset$. Then, at some step $n\geq 1$, given $A_{n-1}$, the first site visited outside of $A_{n-1}$ by a random walk started at the origin is added to $A_{n-1}$ in order to obtain $A_n$. In this context, such random walks are called \textit{particles}. \textsc{idla} was initially introduced by Meakin and Deutch in \cite{meakin1986formation} to model an industrial chemical technique known as electropolishing. The goal of such a process is to eliminate a small coat of material off of a metallic surface in order to make it smoother. It became pertinent to quantify how smooth the surface of a polished metal could get through such a process.

A first shape theorem was established by Lawler, Bramson and Griffeath in \cite{lawler1992internal} to describe the asymptotic shape of $A_n$, as $n$ tends to infinity, as a Euclidean ball. This result was later made sharper with the works of Jerison, Levine and Sheffield \cite{jerison2012logarithmic, JLS13, JLS14} and Asselah and Gaudilli\`ere \cite{asselah2013logarithmic, asselah2013sublogarithmic, AG14, asselah2019outer}. In their works, the bounds for fluctuations around the limit shape are improved from linear to logarithmic in dimension $d=2$ and sublogarithmic in dimensions $d\geq 3$. See also \cite{GQ} for a continuous time version. From then on, variants of the \textsc{idla} model have been studied on  many other graphs, such as on cylinder graphs in \cite{JLS14b, LS, S19}, on supercritical percolation clusters in \cite{DLYY, Shellef}, on comb lattices in \cite{AR16, HS12} or on non-amenable graphs in \cite{BlB07,H08}. 

While the previously cited \textsc{idla} models are single-source, multi-source \textsc{idla} models have also been considered in \cite{Darrow23, Darrow24, LP10} and by the authors in \cite{chenavier2024idla}. This question was originally investigated by Diaconis and Fulton \cite{diaconis1991growth} in the context of the \textit{smash sum} of two domains in which they discover the famous \textit{Abelian Property} of \textsc{idla} aggregates, meaning that modifying the order in which particles are launched does not change the distribution of the final aggregate. This beautiful property will be a powerful tool for the study of \textsc{idla} models.

In the present paper, we introduce new random graphs on $\mathbb{Z}^d$ with $d\geq 2$, called \textit{\textsc{idla} forests}, whose construction is based on a multi-source \textsc{idla} protocol. We consider an infinite set of sources, namely the hyperplane $\mathcal{H} := \{0\} \times \mathbb{Z}^{d-1}$. Basically, in addition to the site at which the current particle exits the aggregate and stops, we also retain the edge by which the particle reaches that site. This procedure leads to a random forest on the lattice $\mathbb{Z}^d$, i.e.\@\xspace a collection of disjoint random trees rooted at sources of $\mathcal{H}$. Unlike \textsc{idla} aggregates, trajectories of particles really matter for the \textsc{idla} forests, meaning the Abelian property is no longer true for this new model, making it more difficult to show the existence of these forests.

However, we prove in Theorem~\ref{thm: forest stabilization} (our main result) the existence of the infinite-volume \textsc{idla} forests, generated by the infinite set of sources $\mathcal{H}$. See Figure~\ref{fig: adagg simu} for a simulation in dimension $d=2$. Moreover, our construction does not favor any source in the sense that (roughly speaking), at any time, the next source to emit a particle is selected `uniformly' among all the sources of $\mathcal{H}$. This remarkable feature is stated in Theorem~\ref{thm: forest properties} in which we prove that the \textsc{idla} forests are ergodic w.r.t. the translations of $\mathcal{H}$.

Let us notice that the existence of bi-dimensional \textsc{idla} forests has been already explored by the authors in \cite{chenavier2023bi} but their proof strongly used the one-dimensional aspect of the set of sources--which is $\{0\} \times \mathbb{Z}$ when $d = 2$--and then completely collapses in higher dimensions. More than a generalization of \cite{chenavier2023bi}, we think that the method developed here and based on percolation tools is original in the context of \textsc{idla} and is certainly promising to deal with graphs built from \textsc{idla} protocols with infinitely many sources.

\subsection*{Construction of the finite-volume IDLA forests}

Let us start by describing our random inputs. Let us first consider a family of i.i.d. Poisson Point Processes (\textsc{ppp}) on $\R_+$, with intensity $1$ and denoted by $\{\mathcal{N}_z\}_{z\in\mathcal{H}}$. Each \textsc{ppp} $\mathcal{N}_z$ provides a sequence $(\tau_{z,j})_{j \geq 1}$ of successive tops which act as random clocks for the emission of particles from the source $z$. Thus, to the sequence $(\tau_{z,j})_{j \geq 1}$, is associated a sequence of independent simple random walks $(S_{z,j})_{j \geq 1}$ on $\mathbb{Z}^d$ starting at $z$. Note that the $S_{z,j}$'s, for $z \in \mathcal{H}$ and $j \geq 1$, are independent from each other and also independent from the \textsc{ppp} $\mathcal{N}_z$'s. Hence, at time $\tau_{z,j}$, a particle is emitted from the source $z$ (precisely, the $j$-th one coming from $z$), and follows the trajectory given by the random walk $S_{z,j}$ until exiting the current aggregate. To avoid having multiple particles alive at the same time, we assume that particles realize their trajectories instantaneously (w.r.t. the Poisson clocks).

Let $M, n \geq 0$ be integers. Set $\mathcal{H}_M := \{0\} \times \llbracket -M, M \rrbracket^{d-1}$. Let us consider the random set $A^{\dagger}_n[M] \subset \mathbb{Z}^d$ defined as the \textsc{idla} aggregate generated by particles emitted from the sources of $\mathcal{H}_M$ and during the time interval $[0,n]$. Since the number of particles involved in $A^{\dagger}_n[M]$ is a.s.\@\xspace finite ($n(2M+1)^{d-1}$ in mean), this aggregate is a.s.\@\xspace well defined. See the left hand side of Figure~\ref{fig: adagg simu} for a simulation of $A^{\dagger}_{20}[50]$ in dimension $d=3$.

Now, we can build quite naturally from $A^{\dagger}_n[M]$ a finite-volume \textsc{idla} forest $\fdag{n}[M]$, simply by considering the edges of $\mathbb{Z}^d$ from which the particles involved in $A^{\dagger}_n[M]$ exit the current aggregate. Let $\kappa := \sum_{z \in \mathcal{H}_M} \#\mathcal{N}_z([0,n])$ be the total number of such particles. Let us enumerate them according to their starting times, say $0 < \tau_1 < \tau_ 2 < ... < \tau_{\kappa} < n$ (they are a.s.\@\xspace different). For $j \in \llbracket 1, \kappa \rrbracket$, we denote by $A[j]$ the aggregate obtained until time $\tau_j$, including the site added by the particle sent at time $\tau_j$. We then have $A[0] = \emptyset$ (with $\tau_0=0$) and $A[\kappa] = A^{\dagger}_n[M]$. We proceed by induction to build the associated forest $\fdag{n}[M]$. Let us first set $\mathcal{F}_n[M, 0] = (\emptyset, \emptyset)$. Now, for $j \in \llbracket 1, \kappa \rrbracket$, given the random graph $\fdag{n}[M,j-1] = (V_{j-1} , E_{j-1})$, we define $\fdag{n}[M,j] = (V_j,\ E_j)$ as follows. Let $x$ be the new site added to $A[j-1]$ by the $j$-th particle.
\begin{itemize}
\item[$\bullet$] If $x$ is the source from which the $j$-th particle is emitted, then this particle actually is the first one emitted from $x$, and $x$ will be the root of a new tree in the graph. We set $V_j = V_{j-1} \cup \{x\}$ and $E_j = E_{j-1}$.
\item[$\bullet$] Otherwise, let $x'$ be the last site of $A[j-1]$ visited by the $j$-th particle before reaching $x$. Then we set $V_j = V_{j-1} \cup \{x\}$ and $E_j = E_{j-1} \cup \{(x',x)\}$.
\end{itemize}
Finally, we define $\fdag{n}[M] := \mathcal{F}_n[M,\kappa]$. See the right hand side of Figure~\ref{fig: adagg simu} for a simulation of $\fdag{30}[20]$ in dimension $d=2$.
This construction ensures that $\mathcal{F}_n[M]$ is a finite union of trees with roots in $\mathcal{H}_M$ and whose vertex set is equal to $A^{\dagger}_n[M]$.

\begin{figure}[!h]
\centering
\begin{tabular}{cc}
	\includegraphics[width=5cm,height=5cm]{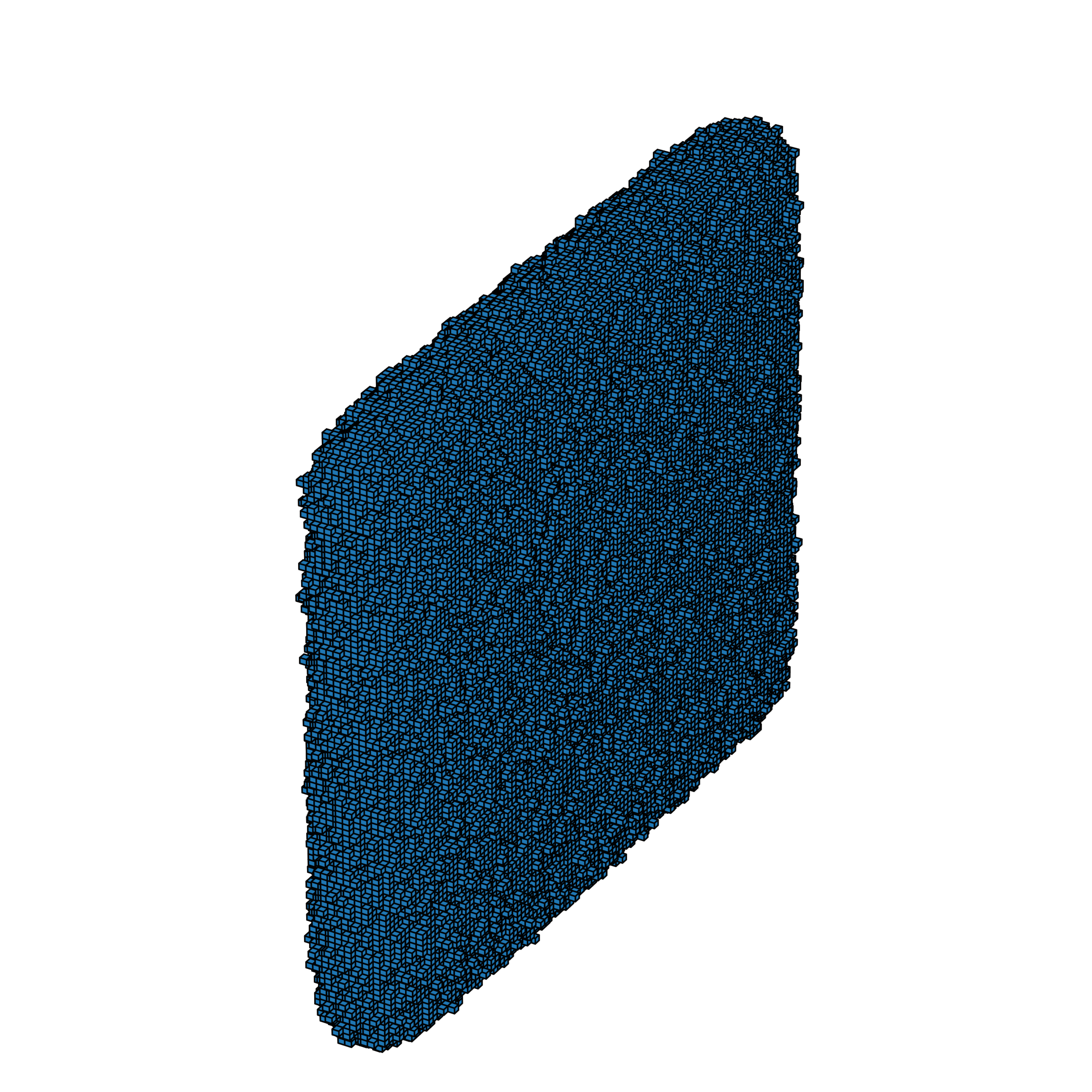} & \includegraphics[width=6.5cm,height=6cm]{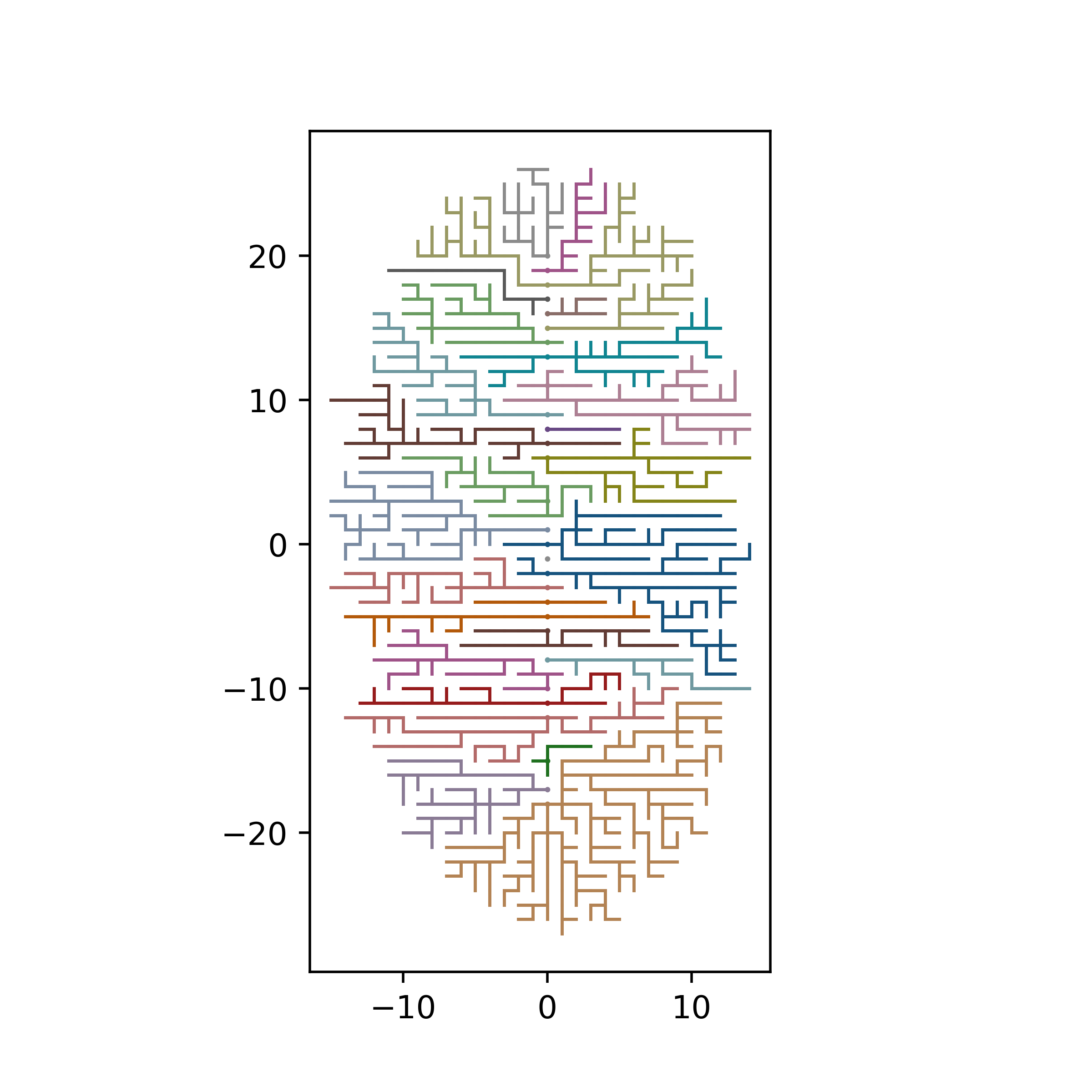}
\end{tabular}
\caption{To the left: A realization in dimension $d=3$ of the multi-source aggregate $A^{\dagger}_{20}[50]$ with particles emitted from $\mathcal{H}_{50}$ and on the time interval $[0,20]$, looking like a `bar of soap'. Each settled particle is represented by a blue cube. To the right: A realization in dimension $d=2$ of the finite-volume \textsc{idla} forest $\fdag{30}[20]$ associated to the aggregate $A^{\dagger}_{30}[20]$. Each tree is represented in a different color. Unfortunately, for visual reasons, we will only represent \textsc{idla} forests in dimension $d=2$.}
\label{fig: adagg simu}
\end{figure}

\subsection*{No monotonicity because of chains of changes}

Thanks to the \textit{natural coupling} defined in Section~\ref{subsec: couplings}, one can construct on the same probability space the aggregates $A^{\dagger}_n[M]$, for all $M\geq 0$, in such a way that a.s.\@\xspace $A^{\dagger}_n[M] \subset A^{\dagger}_n[M+1]$ for any $M$. This monotonicity property allows us to a.s.\@\xspace define a limiting aggregate as
\begin{equation}
\label{Adag(n,infty)}
A^{\dagger}_n[\infty] := \bigcup_{M \geq 0} \uparrow A^{\dagger}_n[M] ~.
\end{equation}
However the same monotonicity property does not hold for the sequence $(\fdag{n}[M])_{M\geq 0}$ of associated \textsc{idla} forests, as depicted in Figure~\ref{fig: forest simu}. Consequently, we cannot define an infinite-volume forest as the increasing union of finite-volume forests, in the same spirit as \eqref{Adag(n,infty)}. Let $M'\geq M \geq 0$. Although their vertex sets satisfy the inclusion $V(\fdag{n}[M]) = A^{\dagger}_n[M] \subset V(\fdag{n}[M']) = A^{\dagger}_n[M']$ (thanks to the natural coupling), this is no longer true for their edge sets. Indeed, some vertices present in both forests $\fdag{n}[M]$ and $\fdag{n}[M']$ are reached using different particles and possibly through different edges. This contributes to an edge in $\fdag{n}[M]$ which is not present in $\fdag{n}[M']$ (and conversely). These discrepancies between $\fdag{n}[M]$ and $\fdag{n}[M']$ can occur through a tricky phenomenon called \textit{chains of changes} that we detail now. In a first time, the reader may skip this part and go directly to the result section below.
\begin{figure}[htbp]
\centering
\begin{minipage}{0.45\textwidth}
	\centering
	\includegraphics[width=0.5\textwidth]{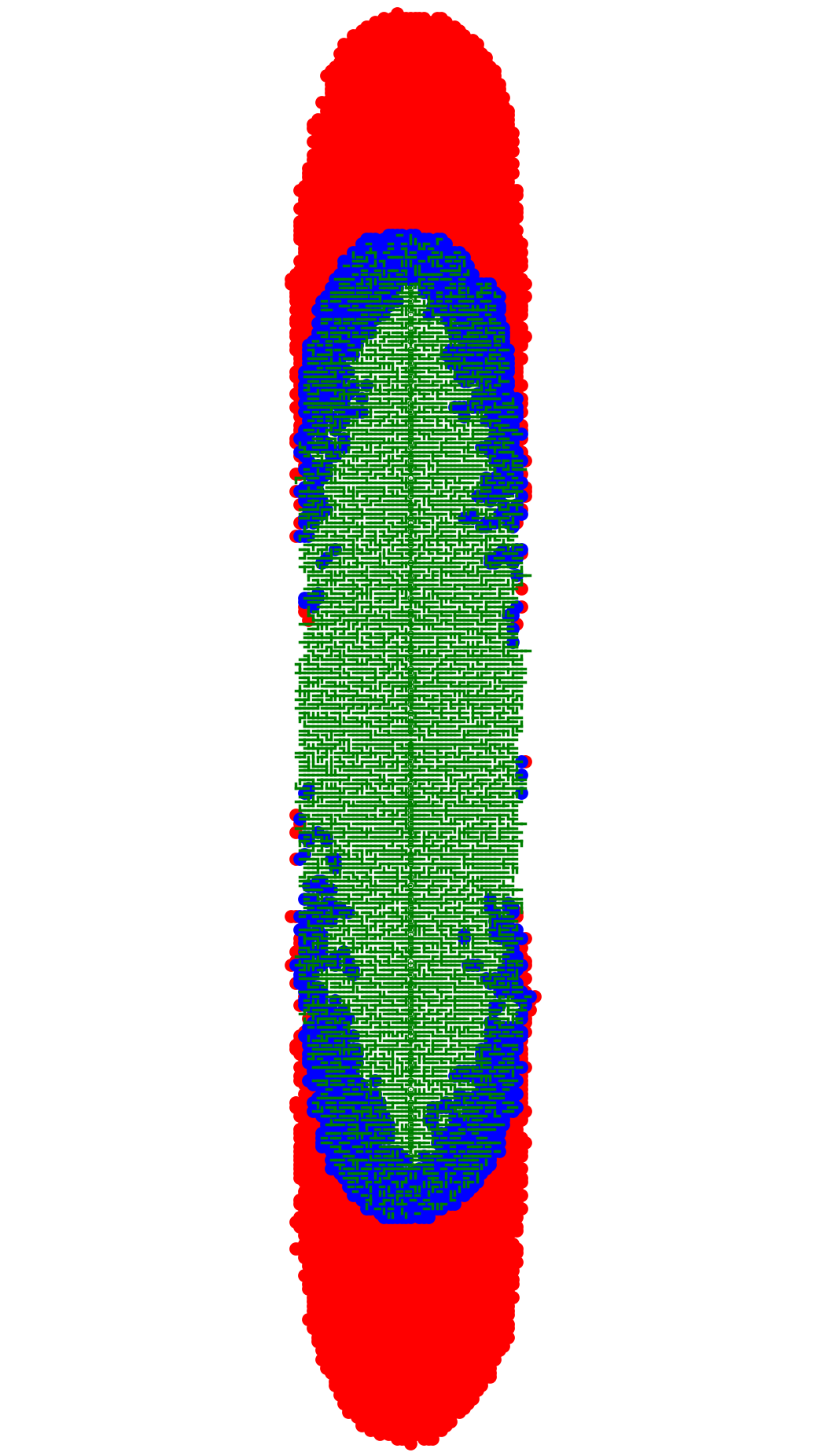}
\end{minipage}%
\hspace{-20mm} 
\begin{minipage}{0.45\textwidth}
	\centering
	\includegraphics[width=0.6\textwidth]{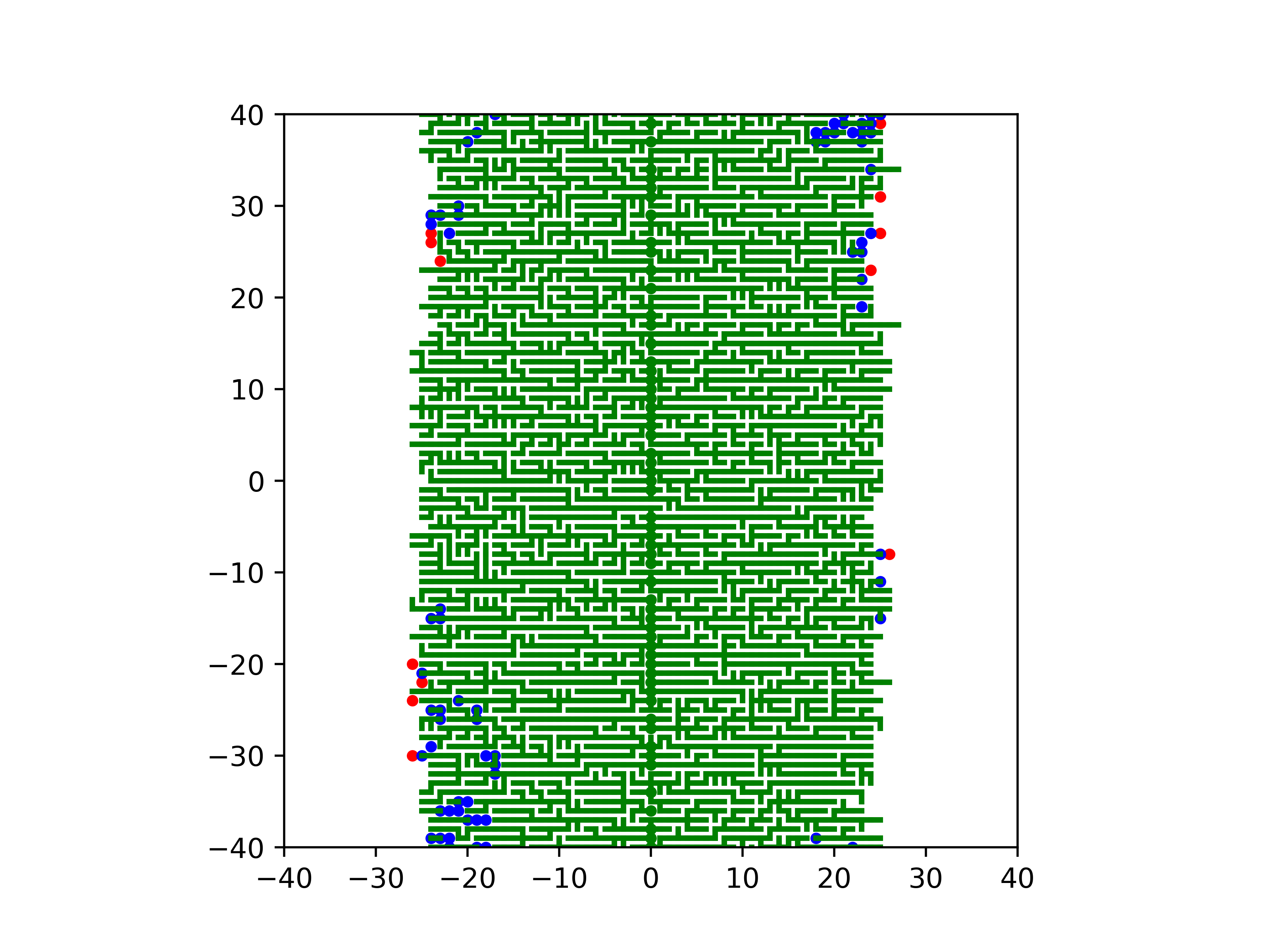} \\[1ex]
	\includegraphics[width=0.6\textwidth]{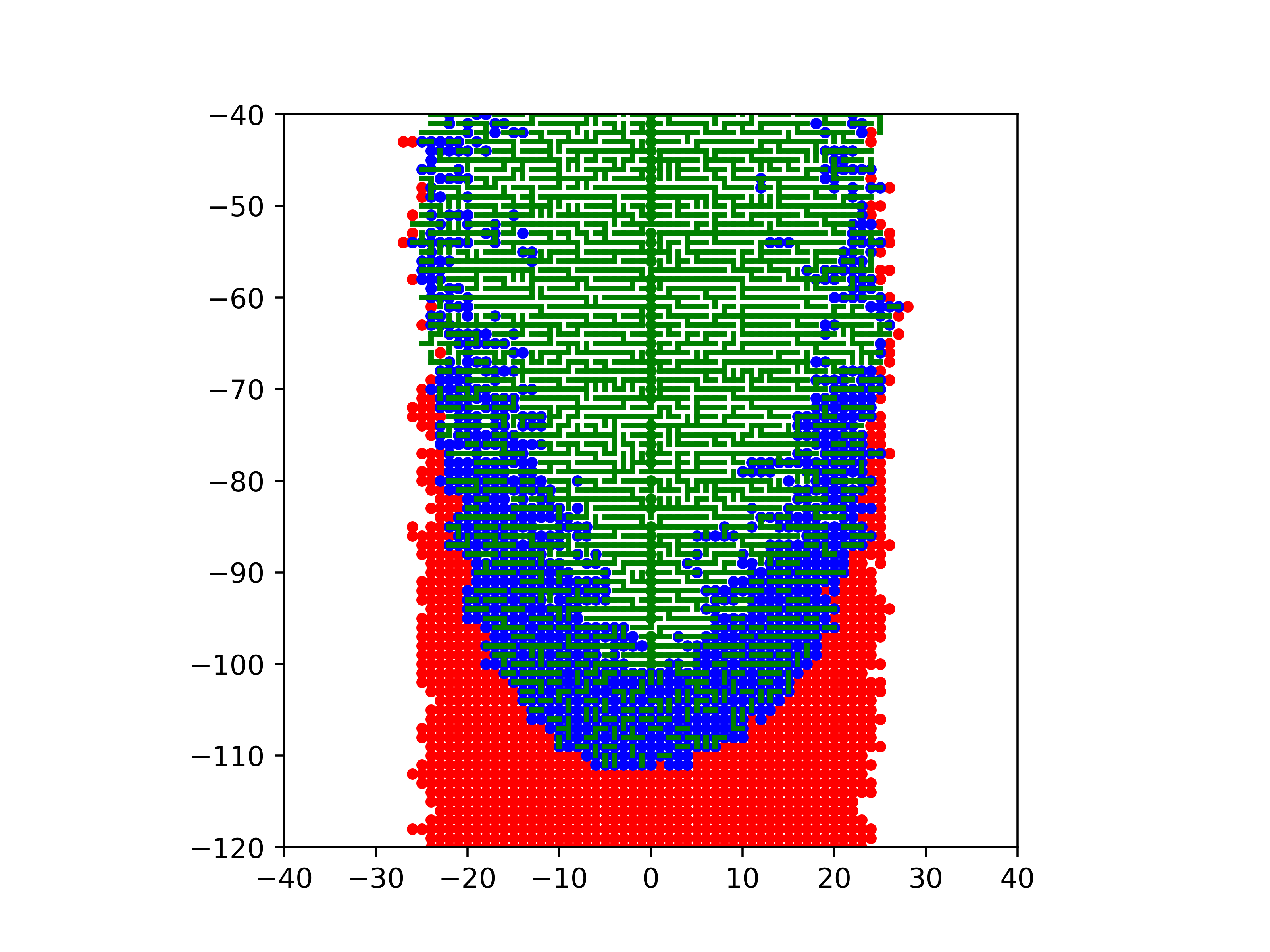}
\end{minipage}
\caption{A realization of $\fdag{50}[100]$ and $\fdag{50}[150]$ using the natural coupling, i.e.\@\xspace their vertex sets are such that $A^{\dagger}_{50}[100] \subset A^{\dagger}_{50}[150]$. The green edges are common to both forests. The blue points are vertices common to both aggregates, but reached by different particles, and may lead to discrepancies between $\fdag{50}[100]$ and $\fdag{50}[150]$. The red points are vertices of $A^{\dagger}_{50}[150] \setminus A^{\dagger}_{50}[100]$. Both pictures on the right are zooms of the one on the left. One can see the presence of many blue points, especially on both extremities of $A^{\dagger}_{50}[100]$. Remark also that some of them appear in the vicinity of the $e_1$-axis ($e_1$ denoting the first vector of the canonical basis): see the top right picture. These possible discrepancies are produced by chains of changes.}
\label{fig: forest simu}
\end{figure}

Let $M' \geq M \geq 0$ and $n \geq 0$. To explain what a chain of changes is, we need to slightly describe the natural coupling we use. The basic idea of that coupling consists in using the same random clocks $(\mathcal{N}_z)$ and the same random walks $(S_{z,j})$ for both aggregates $A^{\dagger}_n[M]$ and $A^{\dagger}_n[M']$. Hence a particle starting from a source in $\mathcal{H}_M$ will work for both aggregates, i.e.\@\xspace it will add a new site to both aggregates (but not necessarily the same site). However, a particle starting from a source in $\mathcal{H}_{M'} \setminus \mathcal{H}_M$ will only work for the larger aggregate $A^{\dagger}_n[M']$. Now, consider such a particle starting at time $t_1 \in (0,n)$ from a source in $\mathcal{H}_{M'} \setminus \mathcal{H}_M$, it adds a site $x_1$ to the larger aggregate. Precisely, if we write $A_{t_1^-}^{\dagger}[M']$ the current aggregate produced right before sending that particle, then we get $A_{t_1}^{\dagger}[M'] = A_{t_1^-}^{\dagger}[M'] \cup \{x_1\}$, while $A_{t_1}^{\dagger}[M] = A_{t_1^-}^{\dagger}[M]$ remains unchanged. At this time, $x_1$ belongs to the larger aggregate $A_{t_1}^{\dagger}[M']$ but not to the smaller one $A_{t_1}^{\dagger}[M]$. If no other future particles starting from $\mathcal{H}_{M}$ visit the site $x_1$, then $x_1$ will remain a discrepancy until time $n$ between both aggregates. Conversely, assume $x_1$ is visited at time $t_2 \in (t_1,n)$ (and for the first time) by a particle coming from $\mathcal{H}_{M}$, then both aggregates are updated as follows:
\begin{itemize}
\item[$\bullet$] The site $x_1$ is added to the smaller aggregate: $A_{t_2}^{\dagger}[M] = A_{t_2^-}^{\dagger}[M] \cup \{x_1\}$.
\item[$\bullet$] Since $x_1$ already belongs to $A_{t_2^-}^{\dagger}[M']$, the particle continues its trajectory until exiting the larger aggregate, say through a new site $x_2$, that is then added: $A_{t_2}^{\dagger}[M'] = A_{t_2^-}^{\dagger}[M'] \cup \{x_2\}$.
\end{itemize}
At time $t_2$, the site $x_1$ is no longer a discrepancy between both aggregates but it has been reached by two different particles: $x_1$ actually is a blue point using the color code of Figure~\ref{fig: forest simu}, i.e.\@\xspace both edges leading to $x_1$ in $\fdag{n}[M]$ and $\fdag{n}[M']$ could be different. Moreover, the site $x_2$ has become a discrepancy between both aggregates $A_{t_2}^{\dagger}[M]$ and $A_{t_2}^{\dagger}[M']$. In other words, the discrepancy has been \textit{relayed} from $x_1$ to $x_2$ by the particle emitted at time $t_2$.

From then on, one can imagine a scenario where, at a random time $t_3 \in (t_2,n)$, the discrepancy at $x_2$ is relayed (by a third particle from $\mathcal{H}_M$) to a new site $x_3$ and possibly becomes another difference between both forests (i.e.\@\xspace a blue point), and so on. Such a phenomenon is referred to as a \textit{chain of changes}. We point out that, even if it is initiated by a particle from $\mathcal{H}_{M'} \setminus \mathcal{H}_M$--i.e.\@\xspace quite far away from the $e_1$-axis when $M$ is large--a chain of changes can spread up to the $e_1$-axis thanks to several relays. See the blue points in the top right picture of Figure~\ref{fig: forest simu}.

\subsection*{Main results}

Our main task consists in controlling the chains of changes described in the previous section and proving that they cannot spread up too much inside the aggregate. This leads to the following stabilization result for the sequence of \textsc{idla} forests $(\fdag{n}[M])_{M \geq 0}$. For that purpose, we need to define the strip $\mathbb{Z}_K := \mathbb{Z} \times \llbracket -K, K \rrbracket^{d-1}$ for any integer $K\geq 0$.

\begin{theorem}[Forest stabilization result]
\label{thm: forest stabilization}
Let $d \geq 2$. For all $n \geq 1$ and $K \geq 1$, the following holds with probability one: 
\begin{equation}
	\label{Stabilization}
	\exists N_0 = N_0(n,K) \geq 0 , \, \forall N \geq N_0 , \, \fdag{n}[N] \cap \mathbb{Z}_K = \fdag{n}[N_0] \cap \mathbb{Z}_K ~,
\end{equation}
where the above identity means that all vertices and edges of $\fdag{n}[N]$ and $\fdag{n}[N_0]$ inside the strip $\mathbb{Z}_K$ coincide.
\end{theorem}

From then on, Theorem~\ref{thm: forest stabilization} allows us to take the limit $M \to \infty$ in the sequence $(\fdag{n}[M])_{M \geq 0}$ in order to obtain an infinite-volume forest $\mathcal{F}_n$. First remark that the sequence $(N_0(n,K))_{n,K}$ in (\ref{thm: forest stabilization}) can be chosen increasing in $K$, which implies for any $K' \geq K$,
\begin{equation}
\label{IncreasingFn}
\mathcal{F}_n[N_0(n,K)] \cap \mathbb{Z}_K = \mathcal{F}_n[N_0(n,K')] \cap \mathbb{Z}_K \subset \mathcal{F}_n[N_0(n,K')] \cap \mathbb{Z}_{K'} ~.
\end{equation}
Inclusion (\ref{IncreasingFn}) compensates for the lack of monotonicity of the sequence $(\fdag{n}[M])_{M \geq 0}$ and allows us to define the \textit{infinite-volume \textsc{idla} forest up to time $n$} denoted by $\mathcal{F}_n$ as
\begin{equation}
\label{defi:Fn}
\mathcal{F}_n := \bigcup_{K\geq 1} \uparrow \mathcal{F}_n[ N_0(n,K)] \cap \mathbb{Z}_K ~,
\end{equation}
a.s.\@\xspace and for any $n \geq 1$. A realization of $\mathcal{F}_{100}[200]$ is given in Figure~\ref{fig: adagg simu2} seen through the strip $\mathbb{Z}_{40}$.

\begin{figure}[!ht]
\centering
\includegraphics[width=8cm,height=5cm]{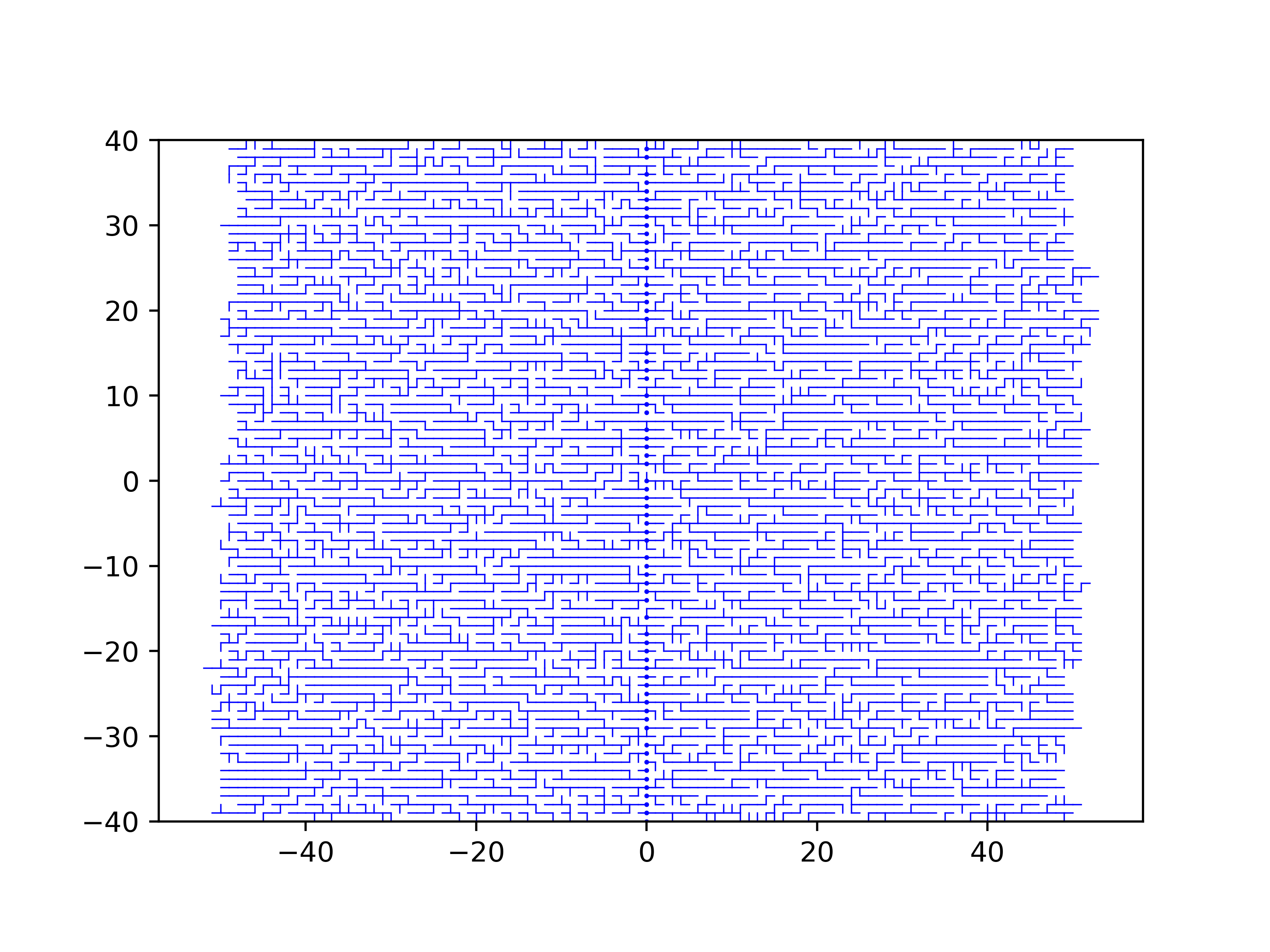}
\caption{Here is a realization of $\mathcal{F}_{100}[200]$ intersected with the strip $\mathbb{Z}_{40}$. Taking $M = 200$ large enough, one can expect the infinite-volume forest $\mathcal{F}_{100}$ and the finite volume forest $\mathcal{F}_{100}[200]$ to coincide on $\mathbb{Z}_{40}$ thanks to Theorem~\ref{thm: forest stabilization}.}
\label{fig: adagg simu2}
\end{figure}

After taking the limit $M \to \infty$ in space, let us take the limit $n \to \infty$ in time. The sequence $(N_0(n,K))_{n,K}$ can also be chosen increasing in $n$. So using once again the stabilization result of Theorem~\ref{thm: forest stabilization}, we can write:
\begin{samepage}
\begin{eqnarray*}
	\mathcal{F}_n \cap \mathbb{Z}_K = \mathcal{F}_n[N_0(n,K)] \cap \mathbb{Z}_K & = & \mathcal{F}_n[N_0(n+1,K)] \cap \mathbb{Z}_K \\
	& \subset & \mathcal{F}_{n+1}[N_0(n+1,K)] \cap \mathbb{Z}_K = \mathcal{F}_{n+1} \cap \mathbb{Z}_K,
\end{eqnarray*}
\end{samepage}
where the inclusion $\mathcal{F}_n[M] \subset \mathcal{F}_{n+1}[M]$ is merely due to extra particles emitted during the time interval $(n,n+1]$ (from $\mathcal{H}_M$). Hence, the sequence of random graphs $(\mathcal{F}_n)_{n\geq 1}$ is increasing in the sense that a.s.\@\xspace for any $n \geq 1$, $V(\fdag{n}) \subset V(\fdag{n+1})$ and $E(\fdag{n}) \subset E(\fdag{n+1})$. We then define the \textit{infinite-volume \textsc{idla} forest} $\mathcal{F}_\infty$ by 
\begin{equation}
\label{defi:Finfty}
\fdag{\infty} := \bigcup_{n\geq 1} \uparrow \fdag{n}\quad \mathrm{a.s.}
\end{equation}

The infinite-volume \textsc{idla} forests $\fdag{\infty}$ and $(\fdag{n})_{n\geq 0}$ previously defined are built from an infinite set of sources, namely the hyperplane $\mathcal{H} = \{0\} \times \mathbb{Z}^{d-1}$. Theorem~\ref{thm: forest properties} below asserts that they are invariant in distribution w.r.t. translations of $\mathcal{H}$, meaning that all sources in $\mathcal{H}$ play the same role in their constructions. This gives a mathematical sense to what was announced at the beginning about $\fdag{\infty}$ and $(\fdag{n})_{n\geq 0}$; at each time, the next source to emit a particle is chosen `uniformly' among $\mathcal{H}$.

For $k \in \mathcal{H}$, let us denote by $T_k$ the translation operator on $\mathbb{Z}^d$ defined by $\forall x \in \mathbb{Z}^d$, $T_k(x) = x+k$.

\begin{theorem}
\label{thm: forest properties}
Let $d \geq 2$. The infinite-volume \textsc{idla} forests $\fdag{\infty}$ and $(\fdag{n})_{n\geq 0}$ satisfy the following properties: 
\begin{enumerate}
	\item Almost surely, the set of vertices of $\fdag{\infty}$ satisfies $V(\fdag{\infty}) = \mathbb{Z}^d$.
	\item The distributions of $\left(\fdag{n}\right)_{n \geq 0}$ and $\fdag{\infty}$ are invariant w.r.t. translations $T_k,\ k \in \mathcal{H}$.
	\item The distributions of $\left(\fdag{n}\right)_{n \geq 0}$ and $\fdag{\infty}$ are $\alpha$-mixing, that is for all $n\geq 0$, for all events $\mathcal{A},\ \mathcal{B}$, 
	\[
	\lim_{|k| \to \infty}\mathbb{P}\left( \fdag{n} \in \mathcal{A},\ T_k\fdag{n} \in \mathcal{B}\right) = \mathbb{P}\left( \fdag{n} \in \mathcal{A} \right) \mathbb{P}\left( \fdag{n} \in \mathcal{B} \right)
	\] 
	and
	\[
	\lim_{|k| \to \infty}\mathbb{P}\left( \fdag{\infty} \in \mathcal{A},\ T_k\fdag{\infty} \in \mathcal{B}\right) = \mathbb{P}\left( \fdag{\infty} \in \mathcal{A} \right) \mathbb{P}\left( \fdag{\infty} \in \mathcal{B} \right).
	\] 
	Consequently, the distributions of $\left(\fdag{n}\right)_{n \geq 0}$ and $\fdag{\infty}$ are ergodic, w.r.t. translations $T_k,\ k \in \mathcal{H}$.
\end{enumerate}
\end{theorem}
\color{black}

\subsection*{Strategy for proving Theorem~\ref{thm: forest stabilization}}

Let us call \textit{level} $M$ the set of sources in $\mathcal{H}$ at distance $M$ from the origin (w.r.t. the infinite norm $\|(z_1,\ldots,z_d)\| := \max_i |z_i|$). Our main task consists in proving the stabilization result Theorem~\ref{thm: forest stabilization}, i.e.\@\xspace given $n \geq 1$ and $K \geq 1$, that (\ref{Stabilization}) recalled below occurs with probability $1$:
\[
\exists N_0 = N_0(n,K) \geq 0 , \, \forall N \geq N_0 , \, \fdag{n}[N] \cap \mathbb{Z}_K = \fdag{n}[N_0] \cap \mathbb{Z}_K ~.
\]
Since the event $\{\fdag{n}[N] \cap \mathbb{Z}_K \neq \fdag{n}[N_0] \cap \mathbb{Z}_K\}$ implies the existence of a chain of changes between $A^{\dagger}_n[N]$ and $A^{\dagger}_n[N_0]$, a natural approach would be, using the Borel-Cantelli Lemma combined with a union bound, to show that 
\begin{equation*}
	\sum_{N_0} \sum_{N \geq N_0} \mathbb{P} \big( \fdag{n}[N] \cap \mathbb{Z}_K \neq \fdag{n}[N_0] \cap \mathbb{Z}_K \big) < \infty ~.
\end{equation*}
However, it is difficult to obtain any upper bounds which decrease with respect to $N$ \textit{and} make the corresponding series summable. Indeed, this event provides no control on the level from which that chain of changes is initiated: it could be initiated anywhere from level $N_0+1$, independently of $N$. A different strategy is therefore required.


Our original approach is to interpret the chain of changes phenomenon with a percolation point of view. Consider $N \geq N_0 \geq K$ and a chain of changes between the forests $\fdag{n}[N_0]$ and $\fdag{n}[N]$ creating (at least) a discrepancy inside the strip $\mathbb{Z}_K$. Then, the sequence of successive relays will be interpreted as a sequence of successive overlapping balls, centered at the sources emitting the relay particles, whose cluster goes from level $N_0$ to the strip $\mathbb{Z}_K$. In particular, when $N_0 \gg K$, a very large cluster will correspond to this chain of changes.

Given $n,K \geq 1$, we proceed by contradiction and assume that, with positive probability,
\begin{equation}
\label{Absurd-Hypo}
\forall N_0 , \, \exists N \geq N_0 , \, \fdag{n}[N] \cap \mathbb{Z}_K \neq \fdag{n}[N_0] \cap \mathbb{Z}_K ~.
\end{equation}
Throughout the whole paper, we will refer to (\ref{Absurd-Hypo}) as the \textit{Absurd hypothesis}. According to our percolation viewpoint, (\ref{Absurd-Hypo}) leads to the existence of a Boolean model, say $\Sigma$, which percolates (i.e.\@\xspace admits an unbounded cluster) with positive probability. To get a contradiction, we will also state that $\Sigma$ is actually subcritical with probability $1$, concluding the proof of Theorem~\ref{thm: forest stabilization}. To do it, we will act on three characteristics of the Boolean model $\Sigma$: its intensity (the density of its centers), its radii distribution and its long-range correlations.

First, the relay particles involved in a chain of changes are emitted by space-time points $(z,t)$, i.e.\@\xspace from a source $z$ and at time $t$. Denoting by $(z_i,t_i)$'s the sequence of emitting space-time points of a given chain of changes, the time sequence $(t_i)$ is increasing by construction. Taking advantage of this monotonicity property, we prove that the intensity of the Boolean model $\Sigma$ can be chosen as small as we want.

Now, recall that the (infinite) aggregate $A^{\dagger}_n[\infty]$ defined in (\ref{Adag(n,infty)}) is also the vertex set of the \textsc{idla} forest $\fdag{n}$. So any information about it will help us to analyze the chain of changes phenomenon. In particular, in the Boolean model $\Sigma$, the radius of a ball associated to a relay particle is given by the maximal fluctuations performed by that particle from its source to exiting $A^{\dagger}_n[\infty]$. So stating a global upper bound (see Proposition \ref{thm: global upper bound}) for this infinite aggregate--within a kind of cone--will allow us to control these fluctuations and prove that the radii distribution of $\Sigma$ satisfies good integrability conditions.

Finally, in order to show that $\Sigma$ is subcritical, we will apply a multiscale argument in the manner of \cite{gouere2009subcritical}. A crucial ingredient making this strategy successful is to quantify how much $\Sigma$, when restricted to a finite window, depends on what happens far away. An important step towards such a local property satisfied by the Boolean model $\Sigma$ lies in the stabilization result below. Theorem~\ref{thm: stab agg dagg}, interesting in itself, asserts that with high probability, the infinite aggregate $A^{\dagger}_n[\infty]$, when restricted to the strip $\mathbb{Z}_M$, does not depend on particles launched from levels larger than $2M$.

\begin{theorem}[Aggregate stabilization result]
\label{thm: stab agg dagg}
Let $d \geq 2$ and $n\geq 1$. For any $L \geq 1$, there exists a positive constant $C = C(n,d,L)$ such that for any $M \geq 1$, 
\[
\mathbb{P} \big( A^{\dagger}_n[\infty] \cap \mathbb{Z}_M = A^{\dagger}_n[2M] \cap \mathbb{Z}_M \big) \geq 1 - \frac{C}{M^L} ~.
\]
\end{theorem}

Proof of Theorem~\ref{thm: stab agg dagg} is based on the global upper bound for $A^{\dagger}_n[\infty]$ mentioned above, on a donut argument already used in \cite{chenavier2023bi} and on a variant of the natural coupling between two aggregates, called the \textit{special coupling}.

\subsection*{Why does the proof of \cite{chenavier2023bi} collapse in higher dimensions?}
\label{subsec: proof fails}

In \cite{chenavier2023bi}, the authors proved the existence of \textsc{idla} forests $(\mathcal{F}_n)_{n \geq 1}$ and $\fdag{\infty}$--see (\ref{defi:Fn}) and (\ref{defi:Finfty})--in the case of dimension $d=2$, i.e.\@\xspace with the set of sources $\mathcal{H} = \{0\} \times \mathbb{Z}$. Their proof only works for dimension $d=2$ and cannot be generalized to higher dimensions: let us explain why.

For any integer $n$, let us define the \textit{vacant set} $\mathcal{V}_n \subset \mathcal{H}$ by
\[
\mathcal{V}_n := \{ z \in \mathcal{H} : \, L(z) \cap A^{\dagger}_n[\infty] = \emptyset \}
\]
where $L(z) := z + \{(k,0) : k \in \mathbb{Z}\}$ is the horizontal line passing by $z$. It is proved in Corollary 5.2 of \cite{chenavier2023bi} that the random set $\mathcal{V}_n$ contains a.s. infinitely many sources. Due to the dimension $d=2$, this implies that the aggregate $A^{\dagger}_n[\infty]$ is made up of (infinitely many) disjoint, finite connected components. The aggregate $A^{\dagger}_n[\infty]$ being the vertex set of the \textsc{idla} forest $\mathcal{F}_n$, it is then impossible for a chain of changes initiated from a very far away source to spread and create a discrepancy inside a given strip $\mathbb{Z}_K$ (the relay particles cannot jump from a connected component of $A^{\dagger}_n[\infty]$ to another one). Corollary 5.2 of \cite{chenavier2023bi} is certainly still true in dimension $d \geq 3$ (in the sense that $A^{\dagger}_n[\infty]$ contains an infinite number of empty lines $\mathbb{Z}\times \{y\}$, for $y \in \mathbb{Z}^{d-1}$) but its consequence about $A^{\dagger}_n[\infty]$, due to planarity, definitely collapses in higher dimensions.

However, one could ask whether the aggregate $A^{\dagger}_n[\infty]$ would still be a union of disjoint, finite connected components in dimension $d \geq 3$. Actually that is wrong whenever $n$ is large enough. Indeed, let us call \textit{rooted} a source $z$ having emitted at least one particle during the time interval $[0,n]$. Hence, $z$ is rooted if and only if the corresponding \textsc{ppp} $\mathcal{N}_z$ has generated at least one top in $[0,n]$, which occurs with probability $1-e^{-n}$. The events $\{\text{$z$ is rooted}\}$, $z \in \mathcal{H}$, being independent from each other and their (common) probability tending to $1$ as $n \to \infty$, we get that the set of rooted sources percolates in $\mathcal{H}$ for $n$ large enough. Since a rooted source belongs to $A^{\dagger}_n[\infty]$, we conclude that this aggregate contains a (unique) infinite connected component.

\subsection*{Comparison with the ergodic External DLA model}

A very similar work has been done through a series of papers \cite{PYZ20,PYZ21,PZ}, but for External DLA instead of Internal DLA and only in dimension $d=2$, in which \cite{PYZ20} presents analog results to Theorems \ref{thm: forest stabilization} and \ref{thm: forest properties}. Let us describe these results and their proof techniques, and compare them to the present study.

In \cite{PYZ20}, the authors construct a DLA model, denoted by $\{A_s^\infty\}_{s\leq t}$ for any $t > 0$, on the upper half planar lattice $\mathbb{H}$ (we use in this section notations from \cite{PYZ20}) and growing from the infinite horizontal line $L_0 = \Z \times \{0\}$, where the growth dynamic is governed by a stationary harmonic measure $\mathcal{H}$ (Theorem 1). Results on the stationary harmonic measure actually come from \cite{PZ}. Roughly speaking, the endpoint $y \in \mathbb{H} \setminus A_s^\infty$ of the outgoing edge $\vec{e} = (x,y)$, with $x \in A_s^\infty$, is added to the current aggregate with rate $\mathcal{H}(\vec{e})$ defined as the probability for a particle coming from infinity to come in $A_s^\infty$ through the edge $\vec{e}$. Hence such DLA models are called External because particles contributing to their growth come from outside whereas for Internal DLA ones, particles come from inside. The authors also state that, for every $t > 0$, $A_t^\infty$ is ergodic w.r.t. horizontal translations (Theorem 2 of \cite{PYZ20}).

To prove Theorem 1, the authors first use an interacting particle system valued in $\{0,1\}^\mathbb{H}$, called the interface process and introduced in \cite{PZ}, as a dominating process for their External DLA model. Then, a thinning procedure of this interface process allows one to construct on a common probability space all the aggregates $A_s^n$, $s \in [0,t]$ and $ n \geq 1$, where $A_s^n$ denotes the External DLA growing from the finite horizontal line $V_0^n = [-n,n] \times \{0\}$ up to time $s$. Thanks to this coupling construction, it then remains to prove that for any compact set $K$, the probabilities
\begin{equation}
\label{SummableProcaccia}
\mathbb{P} \big( \exists s \in [0,t] \mbox{ s.t. } A_s^n \cap K \not= A_s^{n+1} \cap K \big)
\end{equation}
are summable (Theorem 4 of \cite{PYZ20}). In other words, it is unlikely that discrepancies between $(A_s^n)_{s\leq t}$ and $(A_s^{n+1})_{s\leq t}$, that are reduced to the extremities $\pm (n+1,0)$ at time $0$, propagate over time within the compact set $K$. So far, our approach to get a stabilization result for the IDLA forests is the same as in \cite{PYZ20}.

However, the strategies to control and restrain the discrepancy propagation in their model and in ours will be completely different, mainly because these models are of different natures (External vs Internal). In \cite{PYZ20}, particles that propagate a lot from the current discrepancy set and toward the set $K$ are said \textit{bad} or even \textit{devastating}. For instance, a devastating particle first visits the current set of discrepancies (localized in a deterministic set $\textrm{Box}_0$ in the vicinity of the extremities $\pm (n+1,0)$) and then travels outside the current aggregates until adding a new site, i.e. a new discrepancy (localized in a deterministic set $\textrm{Box}$ in the vicinity of the origin, containing $K$ and far from $\textrm{Box}_0$). Using the fact that a particle coming from infinity first hits the current random aggregates before hitting the deterministic line $V_0^n$, the authors then bound the probability for a particle to be devastating by
\begin{equation}
\label{SupProcaccia}
\sup_{x \in \textrm{Box}_0} \mathbb{P}_x \big( \tau_{\textrm{Box}} < \tau_{L_0} \big)
\end{equation}
where $\tau_A$ denotes the hitting time of the set $A$. The supremum (\ref{SupProcaccia}) is controlled well enough in Lemma 6.2 of \cite{PYZ20} (using tools from \cite{PYZ21}) to make the probabilities (\ref{SummableProcaccia}) summable. Because for our Internal DLA model, particles come from inside and not from outside, we cannot apply the same strategy. For this reason, we establish a global upper bound (Proposition \ref{thm: global upper bound}) providing a deterministic cone including the current aggregate w.h.p. and playing the role of the line $V_0^n$.

Moreover, one of the first steps of the strategy used in \cite{PYZ20} (and crucial for the sequel) consists in bounding by $n^{\alpha}$, for any (small) $\alpha > 0$ and w.h.p., the number of discrepancies between the two coupled processes $\{A_s^n\}_{s\leq 1}$ and $\{A_s^{n+1}\}_{s\leq 1}$. In our $d$-dimensional setting, the (internal) aggregate $A^\dagger_{s}[M]$ is generated by particles emitted from sources of $\mathcal{H}_M = \{0\} \times \llbracket -M, M \rrbracket^{d-1}$ (here the spatial parameter $M$ plays the same role as $n$ in \cite{PYZ20}). Coupling both aggregates $A^\dagger_{s}[M]$ and $A^\dagger_{s}[M+1]$, each particle coming from a source of the set $\mathcal{H}_{M+1} \setminus \mathcal{H}_{M}$-- whose cardinal is of order $M^{d-2}$ --starts a chain of changes that is relayed by particles emitted from $\mathcal{H}_{M}$ (and working for both aggregates). So, in our context, the number of discrepancies cannot be bounded as in \cite{PYZ20} and this is why we introduce a Boolean model encoding the propagation of discrepancies, that we prove to be subcritical.

\subsection*{Organization of the paper}

Our paper is organized as follows. In Section~\ref{sec: rappels} are gathered several properties about the aggregate $A^{\dagger}_n[\infty]$ that will be useful in the sequel. The natural and special couplings are detailed and the \textit{Aggregate stabilization} result Theorem \ref{thm: stab agg dagg} is proved. In Section~\ref{sec: chains to perco}, we explain how chains of changes can be interpreted in terms of percolation. In particular, we define in Section~\ref{subsec: instant perco} a discrete Boolean model $\hat{\Sigma}_\varepsilon$, with intensity $p_\varepsilon \to 0$ as $\varepsilon \to 0$, which is proved to be supercritical (i.e. percolating) for any $\varepsilon > 0$ under the \textit{Absurd hypothesis} (\ref{Absurd-Hypo}). Conversely, in Section \ref{sec: JB}, we adapt to our context a multiscale argument due to \cite{gouere2009subcritical} which allows us to prove that $\hat{\Sigma}_\varepsilon$ does not percolate provided the intensity $p_\varepsilon$ is small enough, leading to a contradiction with the conclusion of Section~\ref{sec: chains to perco}. Finally, Section \ref{sec: proof forest properties} is devoted to the proof of Theorem \ref{thm: forest properties}. A detailed proof of Proposition \ref{thm: global upper bound} is given in the Appendix \ref{sec: appendix}.


\section{The infinite aggregate $A^{\dagger}_n[\infty]$}
\label{sec: rappels}

Let $n\geq 1$. Recall that the infinite aggregate $A^{\dagger}_n[\infty]$ has been defined in (\ref{Adag(n,infty)}) as the limit of the increasing sequence of (finite) aggregates $(A^{\dagger}_n[M])_{M \geq 1}$. In this section, are gathered several results about $A^{\dagger}_n[\infty]$ which will be helpful for the proof of our main result Theorem \ref{thm: forest stabilization}. Indeed, the infinite aggregate $A^{\dagger}_n[\infty]$ is intended to be the vertex set of the infinite-volume \textsc{idla} forest $\fdag{n}$. Results about $A^{\dagger}_n[\infty]$ are stated in Section \ref{sect:resultsAdag}. Both natural and special couplings are given in Section \ref{subsec: couplings}. Finally, we prove the stabilization result for $A^{\dagger}_n[\infty]$ (Theorem \ref{thm: stab agg dagg}) in Section \ref{subsec: proof of stab agg dagg}.

\subsection{Results}
\label{sect:resultsAdag}

Let us start with an invariance property in distribution for the infinite aggregate $A^{\dagger}_n[\infty]$. This result is based on the fact that the random ingredients generating $A^{\dagger}_n[\infty]$, respectively the collection of \textsc{ppp} $\{\mathcal{N}_z : z \in \mathcal{H}\}$ and the random walks $\{ S_{z,j} : z \in \mathcal{H}, j \geq 1\}$, are respectively i.i.d and independent. See Proposition 2.2 of \cite{chenavier2023bi} for the same result but in dimension $d=2$. The same proof actually works for any $d\geq 2$.

\begin{proposition}
\label{prop: invariance}
The distribution of $A^{\dagger}_n[\infty]$ is invariant w.r.t. translations of the source set $\mathcal{H}$:
\[
T_k A^{\dagger}_n[\infty] \overset{\mathrm{law}}{=} A^{\dagger}_n[\infty]
\]
where $T_k$, for $k \in \mathcal{H}$, is defined on $\mathbb{Z}^d$ by $T_k(x) = x+k$.
\end{proposition}
The following result, given by Proposition \ref{thm: global upper bound}, provides a global control of the shape of $A^{\dagger}_n[\infty]$. It is referred to as a \textit{Global Upper Bound} and is necessary in the proof of Theorem \ref{thm: stab agg dagg} in Section \ref{subsec: proof of stab agg dagg}. Let us begin by introducing some notation. For $0 < \alpha < 1$ and $\varepsilon > 0$, let us consider the cone $\mathscr{C}_{\varepsilon}^{\alpha}$ as 
\[
\mathscr{C}_{\varepsilon}^{\alpha} = \bigcup_{\ell \geq 0} \left\{ z \in \mathbb{Z}^d,\ \|p_{\mathcal{H}}(z)\| = \ell,\ |z_1| \leq \varepsilon \ell^{\alpha}  \right\},
\]
where $p_{\mathcal{H}}$ denotes the orthogonal projection onto $\mathcal{H}$. Then, for any integer $M\geq 0$, we define the event  
\begin{equation}
\label{eqn: over*}
\Over^{\dagger}_{\alpha}(M,n,\varepsilon) := \{A^{\dagger}_n[\infty] \cap \mathbb{Z}_M^c \nsubset \mathscr{C}_{\varepsilon}^{\alpha}\} ~,
\end{equation}
meaning that the aggregate $A^{\dagger}_n[\infty]$ exceeds the cone $\mathscr{C}_{\varepsilon}^{\alpha}$ outside the strip $\mathbb{Z}_M$. Proposition \ref{thm: global upper bound} states an upper bound for the probability of $\Over^{\dagger}_{\alpha}(M,n,\varepsilon)$ which implies that above a certain level, the aggregate is almost surely contained inside the cone $\mathscr{C}_{\varepsilon}^{\alpha}$.

\begin{proposition}{(Global upper bound)}
\label{thm: global upper bound}
For any $L\geq 1$, for any $\varepsilon > 0$ and $\alpha \in (1-1/d,1)$, there exists a positive constant $C = C(d,n,\varepsilon,\alpha,L)$ such that for all integers $M > 1$,
\[
\mathbb{P}\left(\Over^{\dagger}_{\alpha}(M, n, \varepsilon)\right) \leq \frac{C}{M^L} ~.
\]
In particular, with probability $1$, there exists a (random) level from which $A^{\dagger}_n[\infty] \cap \mathbb{Z}_M^c$ is included in the cone $\mathscr{C}_{\varepsilon}^{\alpha}$.
\end{proposition}

Proposition \ref{thm: global upper bound} actually is a refined version of Theorem 4.1 of \cite{chenavier2024idla} which states the same result but for $\alpha = 1$ (i.e. for a wider cone). So only a short proof of Proposition \ref{thm: global upper bound} is given in the Appendix, focusing on the differences due to the use of the thinner cone $\mathscr{C}_{\varepsilon}^{\alpha}$, with $\alpha \in (1-1/d,1)$. However this refinement (using $\mathscr{C}_{\varepsilon}^{\alpha}$ instead of $\mathscr{C}_{\varepsilon}^1$) is required here to get Theorem \ref{thm: stab agg dagg}--and then our main result Theorem \ref{thm: forest stabilization}--as explained at the end of Section \ref{subsec: proof of stab agg dagg}.

\subsection{Two couplings}
\label{subsec: couplings}

Let $n \geq 1$ and $M' > M$. In this section, we detail two different couplings allowing to construct both aggregates $A^{\dagger}_n[M]$ and $A^{\dagger}_n[M']$ on the same probability space in such a way that
\begin{equation}
\label{InclusionAgg}
A^{\dagger}_n[M] \subset A^{\dagger}_n[M']  \quad \mathrm{a.s.}
\end{equation}
The first one, called the \textit{natural coupling}, will be used intensively in Section \ref{sec: chains to perco} to describe the chains of changes. It has been introduced in \cite{chenavier2023bi}. In this paper, we will require a variant of the natural coupling, called the \textit{special coupling}, ensuring a special property $(\star)$ (see below) in addition to \eqref{InclusionAgg}. The special coupling will be used in Section \ref{subsec: proof of stab agg dagg} to get Theorem \ref{thm: stab agg dagg}. Hence, we first recall the natural coupling in details and then its variant.

\medskip

Let us begin by describing the natural coupling. Let $\kappa := \sum_{z \in \mathcal{H}_{M'}} \# \mathcal{N}_z([0,n])$ be the total number of particles sent from $\mathcal{H}_{M'}$ during the time interval $[0,n]$. Let us build two sequences of aggregates $(A_i)_{0\leq i \leq \kappa}$ and $(B_i)_{0\leq i \leq \kappa}$ such that
\[
\text{for all $0 \leq i \leq \kappa$, } \; A_i \subset B_i \; \text{ and } \; A_{\kappa} = A^{\dagger}_n[M] , \; B_{\kappa} = A^{\dagger}_n[M'] ~.
\]
We proceed by induction on $i \in \llbracket 0, \kappa \rrbracket$ by sorting the $\kappa$ particles according to their starting times (from time $0$ to $n$). When $i = 0$ (no particles have been emitted), we have that $A_0 = B_0 =\emptyset$. Now, suppose $i \geq 0$ and $A_i \subset B_i$ and let us say that the $(i+1)$-th particle is sent from a source $z\in \mathcal{H}_{M'}$.
\begin{itemize}
\item[$\bullet$] If $z \in \mathcal{H}_{M'}\!\setminus\!\mathcal{H}_M$ then the $(i+1)$-th particle only contributes to $B_i$. It adds a random site $x$ to $B_i$ while $A_i$ remains unchanged:
\[
A_{i+1} := A_i \subset B_i \subset B_i \cup \{x\} =: B_{i+1}.
\]
\item[$\bullet$] If $z \in \mathcal{H}_M$, the $(i+1)$-th particle contributes to both aggregates. Since $A_i \subset B_i$, it exits $A_i$ before $B_i$, and adds a random site $x$ to $A_i$. Now, we must consider two cases. 
\begin{itemize}
	\item If $x\notin B_i$, then $x$ is added to $B_i$. Hence,
	\[
	A_{i+1} := A_i \cup \{x\} \subset B_i \cup \{x\} =: B_{i+1} ~.
	\]
	\item If $x \in B_i$, then the $(i+1)$-th particle does not exit $B_i$ in $x$, and continues its trajectory until exiting $B_i$ on some site $x' \not= x$. In this case,
	\[
	A_{i+1} := A_i\cup\{x\} \subset B_i \subset B_i \cup \{x'\} =: B_{i+1} ~.
	\]
\end{itemize}
\end{itemize}

Let us now detail the special coupling. The total number of particles sent from $\mathcal{H}_{M'}$ during $[0,n]$ is still denoted by $\kappa$ and, as before, we build two sequences of aggregates $(\tilde{A}_{i})_{0\leq i \leq \kappa}$ and $(\tilde{B}_{i})_{0\leq i \leq \kappa}$ by induction on $i \in \llbracket 0, \kappa \rrbracket$. Our construction will ensure that
\[
\text{for all $0 \leq i \leq \kappa$, } \; \tilde{A}_i \subset \tilde{B}_i \; \text{ and } \; \tilde{A}_{\kappa} = A^{\dagger}_n[M] , \; \tilde{B}_{\kappa} \overset{\mathrm{law}}{=} A^{\dagger}_n[M']~,
\]
while also ensuring that the following condition holds, for any $0 \leq i \leq \kappa$:
\begin{center}
$(\star)$ \; Any element $x \in \tilde{B}_i\!\setminus\!\tilde{A}_i$ is produced by a particle emitted from a source in $\mathcal{H}_{M'}\!\setminus\!\mathcal{H}_M$.
\end{center}
Let us build the $\tilde{A}_{i}$'s and $\tilde{B}_{i}$'s. When $i = 0$, we have that $\tilde{A}_0 = \tilde{B}_0 =\emptyset$. Let us assume for some $i \geq 0$ that $\tilde{A}_i \subset \tilde{B}_i$, and that they both satisfy condition $(\star)$. Let us say that the $(i+1)$-th particle is sent from a source $z\in \mathcal{H}_{M'}$ (and moves according to a random walk $S_z$).
\begin{itemize}
\item[$\bullet$] If $z \in \mathcal{H}_{M'}\!\setminus\!\mathcal{H}_M$ then the $(i+1)$-th particle only contributes to $B_i$. As for the natural coupling, it adds a random site $x$ to $B_i$ while $A_i$ remains unchanged:
\[
\tilde{A}_{i+1} := \tilde{A}_i \subset \tilde{B}_i \subset \tilde{B}_i \cup \{x\} =: \tilde{B}_{i+1}
\]
and the couple $(\tilde{A}_{i+1},\tilde{B}_{i+1})$ still satisfies $(\star)$.
\item[$\bullet$] If $z \in \mathcal{H}_M$, the $(i+1)$-th particle contributes to both aggregates. Since $\tilde{A}_i \subset \tilde{B}_i$, it exits $\tilde{A}_i$ before $\tilde{B}_i$, and adds a random site $x$ to $\tilde{A}_i$. Once again, we consider two cases.
\begin{itemize}
	\item If $x \notin \tilde{B}_i$ then we proceed as for the natural coupling: $x$ is also added to $\tilde{B}_i$ which again implies
	\[
	\tilde{A}_{i+1} := \tilde{A}_i \cup \{x\} \subset \tilde{B}_i \cup \{x\} =: \tilde{B}_{i+1} ~.
	\]
	The couple $(\tilde{A}_{i+1},\tilde{B}_{i+1})$ still satisfies $(\star)$ since the just added site $x$ belongs to $\tilde{A}_{i+1}$ and $\tilde{B}_{i+1}$.
	\item If $x \in \tilde{B}_i\!\setminus\!\tilde{A}_i$, then thanks to condition ($\star$), the site $x$ was reached by a particle (say with index $i' < i$) originating from some source $z' \in \mathcal{H}_{M'}\!\setminus\!\mathcal{H}_{M}$ (and moving according to a random walk $S_{z'}$). Then, the $(i+1)$-th particle settles at $x$, wakes up the $i'$-th particle which continues its trajectory (according to $S_{z'}$) until exiting $\tilde{B}_i$ on some site $y$. In this case,
	\[
	\tilde{A}_{i+1} := \tilde{A}_i \cup \{x\} \subset \tilde{B}_i \subset \tilde{B}_i \cup \{y\} =: \tilde{B}_{i+1} ~.
	\]
	Note that the couple $(\tilde{A}_{i+1},\tilde{B}_{i+1})$ satisfies $(\star)$ since $y$, which is a discrepancy between $\tilde{A}_{i+1}$ and $\tilde{B}_{i+1}$, has been produced by a particle sent from $\mathcal{H}_{M'}\!\setminus\!\mathcal{H}_{M}$.
\end{itemize}
\end{itemize}

To conclude, let us remark that in both couplings, the aggregates $A_i$ and $\tilde{A}_i$ are built in the same way. They are a.s. equal and then $\tilde{A}_\kappa = A_\kappa = A^{\dagger}_n[M]$. This is not the case for the $\tilde{B}_i$'s: even if $\tilde{B}_i = B_i$, the site $y$ added following the random walk $S_{z'}$ could be different from the site $x'$ added following $S_{z}$. However, in distribution, they are equal:
\[
\tilde{B}_{\kappa} \overset{\mathrm{law}}{=} B_{\kappa} = A^{\dagger}_n[M']
\]
since the trajectory used to add the site $y$ to $\tilde{B}_{i}$ is the concatenation of $S_z$ from the source $z$ to $x$, thus $S_{z'}$ after $x$. This trajectory is a random walk.

\begin{figure}[!ht]
\centering
\begin{tikzpicture}[scale=0.8, use Hobby shortcut,closed=true]
	
	\fill [color=gray!20](-8, 0) .. (-5,2) .. (-4, 2.4) .. (-2.7, 2.1) .. (2, 2.4) .. (5, 2) .. (6, 1.5) .. (8, 0) .. (5,-2) .. (3, -2.1) .. (0, -2.4) .. (-4, -2.2);
	
	\fill [color=gray!60]
	(-6, 0) .. controls (-5.8, 1) and (-4.8, 1.6) .. (-2, 1.8)
	.. controls (-1, 1.7) and (1, 1.4) .. (2, 1.3)
	.. controls (3, 1.3) and (4.2, 1.5) .. (5, 1.2)
	.. controls (5.3, 1.1) and (6, 0.6) .. (6, 0)
	.. controls (5.8, -1.6) and (4.7, -1.3) .. (2, -1.8)
	.. controls (1, -1.9) and (-0.5, -2.1) .. (-2, -1.9)
	.. controls (-3, -1.8) and (-4.2, -1.6) .. (-5, -1.3)
	.. controls (-5.5, -1.2) and (-6, -0.5) .. (-6, 0);
	\draw [->,very thick] (-9,0) -- (9,0);
	\draw (-7.2, 0) node[below] {$-M'$};
	\draw (7, 0) node[below] {$M'$};
	\draw (1.05, 0.05) node[above] {$z$};
	\draw (6.65, 0.05) node[above] {$z'$};
	\draw (0, 0) node[below] {$0$};
	
	\draw (-5.2, 0) node[below] {$-M$};
	\draw (5, 0) node[below] {$M$};
	
	\draw (0, 0) node {$+$};
	\draw (1, 0) node {$+$};
	\draw (6.5, 0) node {$+$};
	\draw (7, 0) node {$+$};
	\draw (-7, 0) node {$+$};
	\draw (-5, 0) node {$+$};
	\draw (5, 0) node {$+$};
	\draw (3.2, 1.4) node[right] {$x$};
	\draw (3.68, 2.38) node[right] {$y$};
	\draw (2.5, 2.5) node[right] {$x'$};
	
	\draw (4, 4) node[red, right] {$S_{z'}$};
	\draw (2.5, 4.2) node[green!40!gray, left] {$S_{z}$};
	
	\draw [blue, opacity=0.5, ultra thick] (3.17, 1.35) .. (3.68, 2.38);
	\draw [red] (6.5, 0)to[curve through={(5.5, 0.5) .. (5,0).. (4.5, -0.5) .. (3, 0) .. (3.5, 2)}](4, 5); 
	\draw[blue, opacity=0.5, ultra thick] (1, 0) .. controls (2.5, -0.5) and (4.5, 0.5) .. (3.17, 1.35);
	\draw[green!40!gray] (1, 0) .. controls (2.5, -0.5) and (4.5, 0.5) .. (3.17, 1.35); 
	\draw[green!40!gray] (3.17, 1.35) .. controls (2.5, 2) and (2, 4) .. (1, 5);
	\filldraw [black] (3.17, 1.35) circle (1pt);
	\filldraw [black] (3.68, 2.38) circle (1pt);
	\filldraw [black] (2.5, 2.42) circle (1pt);
	\draw (-8, -2) node{$A^{\dagger}_n[M']$};
	\draw (-2, 2) node[above] {$A^{\dagger}_n[M]$};
	\draw (-7, -2) -- (-6.5, -1.7);
	\draw (-2, 2.1) -- (-2, 1.5);
\end{tikzpicture}
\caption{The aggregates $A^{\dagger}_n[M]$ and $A^{\dagger}_n[M']$ are represented in dark and light gray. The trajectories of random walks $S_z$ and $S_{z'}$ are respectively depicted in green and red. In the natural coupling, the site $x'$ is added to $\tilde{B}_i$ using solely $S_z$ whereas in the special coupling, the site $y$ is added to $\tilde{B}_i$ using first $S_z$ until exiting $\tilde{A}_i = A_i$ and then $S_{z'}$. We highlight in blue the actual path that is realized when doing so.}
\label{fig: coupling}
\end{figure}
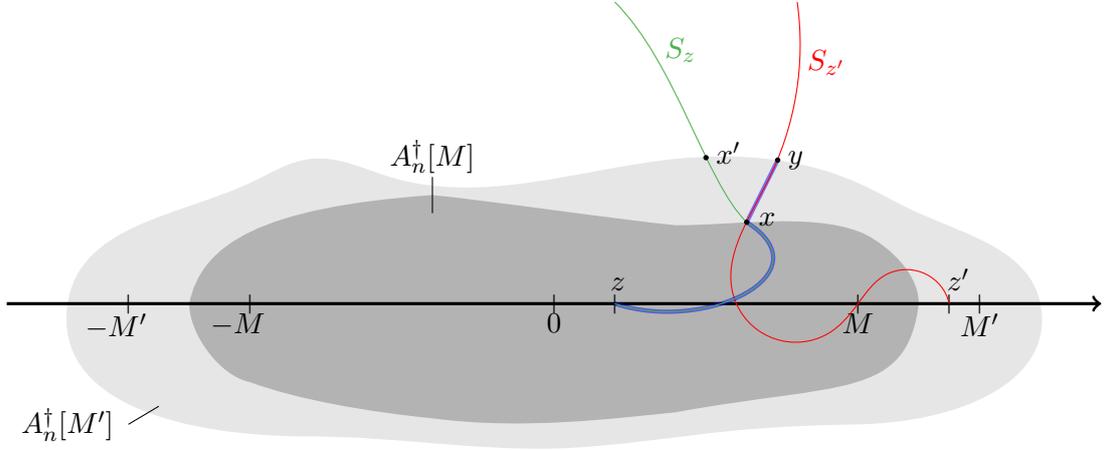

\subsection{Proof of Theorem \ref{thm: stab agg dagg}}
\label{subsec: proof of stab agg dagg}

Let $n \geq 1$. Our goal is to prove that for any $L \geq 1$, there exists a positive constant $C = C(n,d,L)$ such that for any $M\geq 1$, 
\[
\mathbb{P} \big( A^{\dagger}_n[\infty] \cap \mathbb{Z}_M = A^{\dagger}_n[2M] \cap \mathbb{Z}_M \big) \geq 1 - \frac{C}{M^L} ~.
\]

To do it, let us introduce the event 
\[
D_M := \left\{\begin{array}{c}
\text{The trajectory of any random walk associated with a particle of $A^{\dagger}_n[\infty]$ and starting} \\
\text{from a level greater than $2M$ does not visit the strip $\mathbb{Z}_M$ before exiting $\mathscr{C}_{\varepsilon}^{\alpha}$} \\
\end{array}
\right\} ~. 
\]
The event $D_M$ has a probability tending to $1$ as $M \to \infty$:

\begin{lemma}
\label{prop: particle stab}
For any $L\geq 1$, for any $\varepsilon > 0$ and $\alpha \in (0, 1)$, there exists a positive constant $C = C(d,n,\varepsilon,\alpha,L)$ such that for all $M\geq 1$, 
\[
\mathbb{P}\left( D_M \right) \geq 1 - \frac{C}{M^L}.
\]
\end{lemma}

The proof of Theorem \ref{thm: stab agg dagg} works in two steps. We first explain how to conclude from Lemma \ref{prop: particle stab} and then we prove this auxiliary result.

Let us pick $\varepsilon > 0$ and $\alpha \in (1-1/d, 1)$. Let $M,L \geq 1$. Let us consider the event $\mathcal{G}_M$ defined by
\[
\mathcal{G}_M := \Big\{ A^{\dagger}_n[\infty] \cap \mathbb{Z}_{M}^c \subset \mathscr{C}_{\varepsilon}^{\alpha} \Big\} \cap D_M ~.
\]
Thanks to Proposition \ref{thm: global upper bound} and Lemma \ref{prop: particle stab}, there exists a positive constant $C = C(d,n,\varepsilon,\alpha,L)$ such that $\mathbb{P}(\mathcal{G}_M) \geq 1-C M^{-L}$. Given $M' > 2M$, we consider the aggregates $A^{\dagger}_n[2M]$ and $A^{\dagger}_n[M']$ under the special coupling. Hence, a.s. $A^{\dagger}_n[2M]$ is included in $A^{\dagger}_n[M']$ and any element $x$ in $A^{\dagger}_n[M']\!\setminus\!A^{\dagger}_n[2M]$ is produced by a particle emitted from a source in $\mathcal{H}_{M'}\!\setminus\!\mathcal{H}_{2M}$ thanks to the condition $(\star)$. On the event $\mathcal{G}_M$, the random walk associated to this particle necessarily exited $\mathscr{C}_{\varepsilon}^{\alpha}$ before reaching $\mathbb{Z}_M$, and hence necessarily exited $A^{\dagger}_n[M']$ before reaching $\mathbb{Z}_M$. This means that, on the event $\mathcal{G}_M$, both aggregates $A^{\dagger}_n[2M]$ and $A^{\dagger}_n[M']$ coincide on $\mathbb{Z}_M$ and leads to
\begin{equation}
\label{eqn: stab agg dagg}
\mathbb{P} \left( A^{\dagger}_n[M'] \cap \mathbb{Z}_M = A^{\dagger}_n[2M] \cap \mathbb{Z}_M\right) \geq 1- \frac{C}{M^L} , \; \forall M' > 2M ~.
\end{equation}
The key argument here is the special coupling which ensures that any discrepancy between both aggregates is due to a single particle and not to a chain of changes.

We can now conclude. On the one hand, thanks to (\ref{eqn: stab agg dagg}),
\[
\liminf_{M'\to\infty} \mathbb{P} \left( A^{\dagger}_n[M'] \cap \mathbb{Z}_M \subset A^{\dagger}_n[2M] \cap \mathbb{Z}_M\right) \geq 1- \frac{C}{M^L} ~.
\]
On the other hand, the infinite aggregate $A^{\dagger}_n[\infty]$ being the limit of the increasing sequence $(A^{\dagger}_n[M'])_{M'}$ (thanks to the natural coupling), we have
\[
\lim_{M'\to\infty} \mathbb{P} \left( A^{\dagger}_n[M'] \cap \mathbb{Z}_M \subset A^{\dagger}_n[2M] \cap \mathbb{Z}_M\right) = \mathbb{P} \left( A^{\dagger}_n[\infty] \cap \mathbb{Z}_M \subset A^{\dagger}_n[2M] \cap \mathbb{Z}_M\right) ~.
\]
As a consequence, we can write
\[
\mathbb{P} \left( A^{\dagger}_n[\infty] \cap \mathbb{Z}_M \subset A^{\dagger}_n[2M] \cap \mathbb{Z}_M\right) \geq 1- \frac{C}{M^L}
\]
and the same holds for $\mathbb{P}(A^{\dagger}_n[\infty]\cap\mathbb{Z}_M = A^{\dagger}_n[2M]\cap\mathbb{Z}_M)$ since $A^{\dagger}_n[2M] \subset A^{\dagger}_n[\infty]$ with probability $1$. This concludes the proof of Theorem \ref{thm: stab agg dagg}.

\begin{proof}[Proof of Lemma \ref{prop: particle stab}]
Let $\varepsilon > 0$, $\alpha \in (0,1)$ and $M,j \geq 1$. Let us set
\[
E_{M,j} := \left\lbrace\begin{array}{c}
	\mbox{At least one random walk starting from $\mathrm{Ann}(M,j)$} \\
	\mbox{visits the strip $\mathbb{Z}_{M}$ before exiting $\mathscr{C}_{\varepsilon}^{\alpha}$} \\
\end{array}
\right\rbrace
\]
where $\mathrm{Ann}(M,j) := \mathcal{H}_{(j+2)M}\!\setminus\!\mathcal{H}_{(j+1)M}$. Hence, the event $D_M^c$ is equal to $\cup_{j \geq 1} E_{M,j}$ and we focus on bounding each term $\mathbb{P}(E_{M,j})$.

As in the proof of Theorem 1.2 of \cite{chenavier2024idla}, our strategy consists in building donuts from level $(j+1)M$ down to level $M$, symmetric w.r.t. the hyperplane $\mathcal{H}$ and containing the cone $\mathscr{C}_{\varepsilon}^{\alpha}$ (restricted to $\mathbb{Z}_{(j+1)M}\!\setminus\!\mathbb{Z}_M$). The largest donut is the one built at level $(j+1)M$. Its width is equal to $2\varepsilon((j+1)M)^{\alpha}$ and all the following donuts have smaller sizes. Therefore, the number of donuts $k^o = k^o(j,M,\varepsilon,\alpha)$ we can build from level $(j+1)M$ to level $M$ verifies 
\begin{equation}
	\label{eqn: value of k chap2}
	k^o \geq \frac{(j+1)M - M}{2\varepsilon(j+1)^{\alpha}M^{\alpha}} \geq \frac{(jM)^{1-\alpha}}{4\varepsilon},
\end{equation}
where the last inequality is due to $(j+1)^\alpha \leq 2j$.

We know that if a random walk sent from a level greater than $(j+1)M$ reaches the strip $\mathbb{Z}_{M}$ while staying inside the cone $\mathscr{C}_{\varepsilon}^{\alpha}$, it necessarily crossed over the $k^o$ donuts previously built. The probabilistic cost to cross any given donut while staying inside the cone is at most $1-c$ with $c := (2d)^{-2}$. So the probability for such random walk to cross the $k^o$ donuts before exiting $\mathscr{C}_{\varepsilon}^{\alpha}$ is at most $(1-c)^{k^o}$. See Proposition 3.1 of \cite{chenavier2024idla} for details.

Besides, let us denote by $N_{tot}^{(j)} = N_{tot}^{(j)}(d,n,M,j)$ the total number of particles sent from $\mathrm{Ann}(M,j)$ during the time interval $[0,n]$:
\[
N_{tot}^{(j)} := \sum_{z \in \mathrm{Ann}(M,j)} \mathcal{N}_z([0,n]) ~.
\]
The next result allows us to bound from above $N_{tot}^{(j)}$ with high probability.

\begin{lemma}
	\label{lemma: particle control ann}
	Let $M, j \geq 1$ and set $C_{M,j} := \# \mathrm{Ann}(M,j)$. Then,
	\[
	\mathbb{P} \big( N_{tot}^{(j)} > 2n C_{M,j} \big) \leq \exp\left(-c_0 j^{d-2}M^{d-1}\right),
	\]
	where $c_0 = c_0(d,n)$ denotes a positive constant.
\end{lemma}

\begin{proof}[Proof of Lemma \ref{lemma: particle control ann}]
	The searched inequality is a direct consequence of the concentration inequality for Poisson variables (\ref{eqn: concentration inequality poisson}) stated below, applied to $N_{tot}^{(j)}$ which is distributed as a Poisson random variable with parameter $n C_{M,j}$ and to the fact that $C_{M,j}$ is of order $j^{d-2} M^{d-1}$.
	
	If $X$ is a Poisson random variable of parameter $\lambda >0$, then for any $t\geq 0$, the following holds:
	\begin{equation}
		\label{eqn: concentration inequality poisson}
		\mathbb{P}\left(X-\mathbb{E}\left[X\right] \geq t \right) \leq \exp\left(-\frac{t^2}{2\left(\lambda + \frac{t}{3}\right)}\right).
	\end{equation}
\end{proof}

This result uses concentration inequality results from \cite{BoucheronS.2013CiAn}. If $X\sim \mathcal{P}(\lambda)$, then $X$ is sub-Poisson with variance factor $\lambda$ and scale factor $1$, and is hence sub-Gamma with variance factor $\lambda$ and scale factor $1/3$.
Finally, we use a result from that if $Z$ is sub-Gamma with variance factor $v>0$ and scale factor $c>0$, then for any $t>0$,
\[
\mathbb{P}\left(Z - \mathbb{E}[Z] \geq t \right) \leq \exp\left(- \frac{t^2}{2(v + ct)}\right).
\]

All the ingredients are gathered, we can compute:
\begin{eqnarray*}
	\mathbb{P}(E_{M,j}) & \leq & \mathbb{P} \big( E_{M,j} \cap \big\{ N_{tot}^{(j)} \leq 2n C_{M,j} \big\} \big) + \mathbb{P} \big( N_{tot}^{(j)} > 2n C_{M,j} \big) \\
	& \leq & \sum_{i=1}^{2nC_{M,j}} \mathbb{P} \big( \big\{ \text{walk $i$ visits $\mathbb{Z}_M$ before exiting $\mathscr{C}_{\varepsilon}^{\alpha}$} \big\} \big) + \exp \big( - c_0 j^{d-2} M^{d-1} \big) \\
	& \leq & \sum_{i=1}^{2nC_{M,j}} \mathbb{P} \big( \big\{ \text{walk $i$ crosses $k^o$ donuts before exiting $\mathscr{C}_{\varepsilon}^{\alpha}$} \big\} \big) + \exp \big( - c_0 j^{d-2} M^{d-1} \big) \\
	& \leq & 2n C_{M,j} (1-c)^{k^o} + \exp \big( - c_0 j^{d-2} M^{d-1} \big) ~.
\end{eqnarray*}
Thus, \eqref{eqn: value of k chap2} leads to
\[
2n C_{M,j} (1-c)^{k^o} \leq C_1 M^{d-1} (j+1)^{d-2} \exp\big( - C_2 (jM)^{1-\alpha} \big)
\]
where $C_1,C_2$ are positive constants depending only on $d,n,\varepsilon$. Summing over $j \geq 1$, we get
\begin{eqnarray*}
	\mathbb{P} \Big( \bigcup_{j\geq 1} E_{M,j} \Big) \, \leq \, \sum_{j\geq 1} \mathbb{P} (E_{M,j}) & \leq & \sum_{j\geq 1} C_1 M^{d-1}(j+1)^{d -2} \exp \big( - C_2 (jM)^{1-\alpha} \big) \\
	& & + \sum_{j\geq 1} \exp \big( -c_0 j^{d-2} M^{d-1} \big) ~.
\end{eqnarray*}
Since $1-\alpha >0$, both terms of the upper bound above are summable and decrease faster than any power of $M^{-1}$, which concludes the proof.

Remark that the previous conclusion holds only if $d\geq 3$. When $d=2$, we have to proceed slightly differently. Each $\mathrm{Ann}(M,j)$ has the same cardinality, equal to $M$. So, in the previous computation, it suffices to intersect the event $E_{M,j}$ with $\{N_{tot}^{(j)} \leq 2n M j^\beta\}$ (for some $\beta > 0$) since \eqref{eqn: concentration inequality poisson} allows to bound the probability of $\{N_{tot}^{(j)} > 2n M j^\beta\}$ by $\exp(-c'M j^\beta)$--which is summable w.r.t. $j$ and $M$.
\end{proof}

Theorem 4.1 of \cite{chenavier2024idla} is a weaker version of Proposition \ref{thm: global upper bound}: it gives the same result but in the particular case of a linear cone $\mathscr{C}_{\varepsilon}^1$, i.e. with $\alpha = 1$. Taking $\alpha = 1$ in the computation \eqref{eqn: value of k chap2} leads to a constant number of donuts $k^o$ (no longer depending on $j,M$) which prevents us to conclude as above. Hence, in \cite{chenavier2024idla}, in order to get sufficiently many donuts to make unlikely the crossing to $\mathbb{Z}_M$ for particles coming far away, we required more space. This argument led to a stabilization result (Theorem 1.2 of \cite{chenavier2024idla}) weaker than our Theorem \ref{thm: stab agg dagg} since it claimed that we need to go above levels $M^{\gamma}$, $\gamma > 1$, to stabilize $A^{\dagger}_n[\infty] \cap \mathbb{Z}_M$.

However, to apply with success the multiscale argument of Section \ref{sec: JB}, we need a stabilization result for $A^{\dagger}_n[\infty] \cap \mathbb{Z}_M$ requiring a linear number of levels of $M$ (rather than $M^\gamma$). This is why we had to improve the stabilization result Theorem 1.2 of \cite{chenavier2024idla} into Theorem \ref{thm: stab agg dagg}. This justifies the use of the thinner cone $\mathscr{C}_{\varepsilon}^\alpha$ leading to the refined global upper bound Proposition \ref{thm: global upper bound}.


\section{From chains of changes to percolation models}
\label{sec: chains to perco}

To get the stabilization result Theorem \ref{thm: forest stabilization}, we proceed by contradiction by assuming the \textit{Absurd hypothesis} (\ref{Absurd-Hypo}) that we recall now: there exist positive integers $n_0,K_0$ (fixed for the whole section) such that
\[
\forall N_0 , \, \exists N \geq N_0 , \, \fdag{n_0}[N] \cap \mathbb{Z}_{K_0} \neq \fdag{n_0}[N_0] \cap \mathbb{Z}_{K_0}
\]
occurs with positive probability.

In section \ref{sect:FirstPercoModel}, we use a space-time representation of a chain of changes between the forests $\fdag{n_0}[N_0]$ and $\fdag{n_0}[N]$ to describe the propagation of discrepancies between the corresponding aggregates $A^{\dagger}_{n_0}[N_0]$ and $A^{\dagger}_{n_0}[N]$. In Section \ref{sect:IDC}, a percolation model $\Sigma$ with good properties (Lemmas \ref{lem: tail distribution} and \ref{lem: finite intersections}) is introduced. Under the \textit{Absurd hypothesis}, we prove that $\Sigma$ percolates in the sense that it contains an \textit{infinite descending chain} with positive probability. Finally, in Section \ref{subsec: instant perco}, we take advantage of the monotonicity in time of such infinite descending chain in order to state that it actually appears instantaneously. The final result of Section \ref{sec: chains to perco} is summarized in Proposition \ref{prop:ConcSection3} saying that a discrete Boolean model $\hat{\Sigma}_{\varepsilon}$ percolates even if its intensity tends to $0$ with $\varepsilon \to 0$.

\medskip

For $z \in \mathcal{H}$ and $r \in \mathbb{N}$, let us set $\mathbb{B}(z,r) := z + \mathcal{H}_r$ the $(d-1)$-dimensional ball, included in $\mathcal{H}$, centered at $z$ and with radius $r$ (for the infinity norm). We also denote by $p_{\mathcal{H}} : \mathbb{Z}^d \to \mathcal{H}$ the orthogonal projection onto the hyperplane $\mathcal{H}$.

\subsection{A space-time representation of chains of changes}
\label{sect:FirstPercoModel}

Let $(N_0,N)$, with $N \geq N_0$, a couple of integers such that the forests $\fdag{n_0}[N]$ and $\fdag{n_0}[N_0]$ do not coincide on the strip $\mathbb{Z}_{K_0}$. From now on, we consider their vertex sets, namely the aggregates $A^{\dagger}_{n_0}[N]$ and $A^{\dagger}_{n_0}[N_0]$, under the natural coupling defined in Section \ref{subsec: couplings}, i.e. satisfying a.s. the inclusion $A^{\dagger}_{n_0}[N_0] \subset A^{\dagger}_{n_0}[N]$. A (random) space-time couple $(z,t)$ where $z$ is a source with $\| z \| \leq N$ and $t \in [0,n_0]$ is a `top' of the \textsc{ppp} $\mathcal{N}_z$, is called a \textit{starting point}. It means that, at time $t$, a particle using the random walk $S_{z,t} = (S_{z,t}(k))_{k\geq 0}$ and launched from the source $z$, contributes to the construction of $A^{\dagger}_{n_0}[N]$--and also of $A^{\dagger}_{n_0}[N_0]$ if $\| z \| \leq N_0$. As explained in the introduction, the difference between both forests $\fdag{n_0}[N]$ and $\fdag{n_0}[N_0]$ inside $\mathbb{Z}_{K_0}$ is due to a chain of changes, i.e. a sequence of $\kappa \geq 1$ particles coming from starting points $(z_i, t_i)_{1\leq i \leq \kappa}$ satisfying the following conditions:
\begin{equation}
\label{ChainChanges}
\left\{
\begin{array}{l}
	\text{$N_0 < \|z_1\| \leq N$ and $\|z_i\| \leq N_0$, for $2 \leq i \leq \kappa$} \\
	\text{$0 < t_1 < t_2 < \dots < t_{\kappa} < n_0$} \\
	\text{for $1 \leq i \leq \kappa-1$, the $i$-th particle is relayed by the $(i+1)$-th one} \\
	\text{the $\kappa$-th particle exits $A^{\dagger}_{t_{\kappa}^-}[N]$ through $\mathbb{Z}_{K_0}$}
\end{array}
\right. 
\end{equation}
Recall that `the $i$-th particle is relayed by the $(i+1)$-th one' means that the discrepancy $x_i \in A^{\dagger}_{t_i}[N] \setminus A^{\dagger}_{t_i}[N_0]$ created by the $i$-th particle at time $t_i$, is visited by the $(i+1)$-th particle at time $t_{i+1}$, which  contributes to both aggregates. So, at time $t_{i+1}$, $x_i$ is no longer a discrepancy and is replaced with a new one $x_{i+1} \in A^{\dagger}_{t_{i+1}}[N] \setminus A^{\dagger}_{t_{i+1}}[N_0]$ which actually is the site at which the $(i+1)$-th particle settles when it exits the current aggregate $A^{\dagger}_{t_{i+1}^-}[N]$.

Associated with a given starting point $(z,t)$ and with the corresponding particle, we define the radius $R_{N}(z,t)$ as follows:
\[
R_{N}(z,t) := \min \big\{ r \in \mathbb{N} : \text{$\mathbb{B}(z,r)$ contains $p_{\mathcal{H}}(S_{z,t}(0)),p_{\mathcal{H}}(S_{z,t}(1)), \ldots ,p_{\mathcal{H}}(S_{z,t}(\tau))$} \big\}
\]
where
\[
\tau = \tau(z,t,N) := \min\{ k : S_{z,t}(k) \notin A^{\dagger}_{t^-}[N]\}
\]
denotes the time at which the particle moving according to $S_{z,t}$ exits the current aggregate $A^{\dagger}_{t^-}[N]$. In other words, the ball $\mathbb{B}(z,R_{N}(z,t))$ contains the part of the projected trajectory $p_{\mathcal{H}}(S_{z,t})$ until $S_{z,t}$ exits $A^{\dagger}_{t^-}[N]$. It is worth pointing out here that $R_{N}(z,t)$ only depends on the random walk $S_{z,t}$ and the current aggregate $A^{\dagger}_{t^-}[N]$.

Now, let us come back to the sequence of $\kappa \geq 1$ particles satisfying (\ref{ChainChanges}). The fact that the $i$-th particle is relayed by the $(i+1)$-th one means that
\[
\mathbb{B}(z_i , R_{N}(z_i,t_i)) \cap \mathbb{B}(z_{i+1} , R_{N}(z_{i+1},t_{i+1})) \not= \emptyset ~.
\]

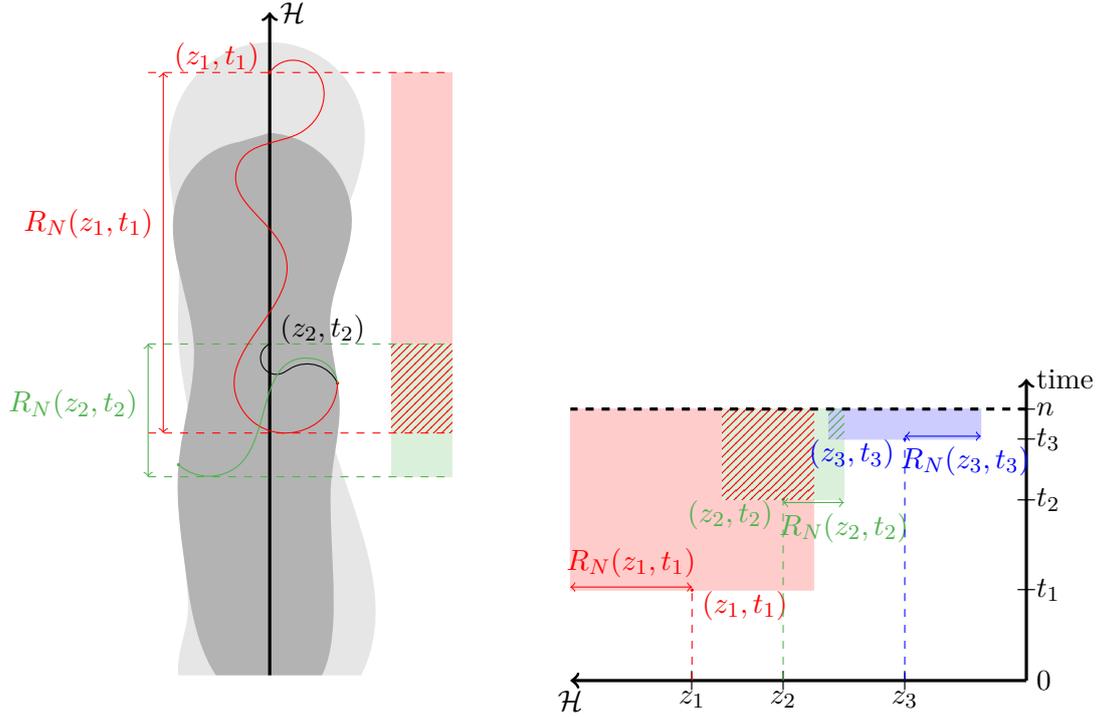
\begin{figure}[!ht]
\centering
\begin{tikzpicture}[scale=0.4, use Hobby shortcut,closed=true]
	\fill [color=gray!20](2, -1) .. (3,-9) .. (-3, -9) .. (-2.7, -7) .. (-2.8, -5) .. (-2.7, -4) .. (-2.5, -2) .. (-3,2) .. (-3, 5) .. (-3, 10) .. (0, 12) ..  (3, 9.8) .. (2, 5);
	
	\fill [color=gray!60] (0,9) .. (2,8) .. (2.7, 6) ..  (2,3) .. (2.3, 0) .. (2, -3) .. (0, -10) .. (-2.7, -6) .. (-3, -2) .. (-2.5, 2) .. (-3, 7) .. (-2, 8.3) .. (-1, 8.7) ;
	
	\fill [color=red!20](4,11) rectangle (6,-0.95);
	\fill [color=green!40!gray!20](4,2) rectangle (6,-2.4);
	\fill[pattern = north east lines, pattern color= red] (4,2) rectangle (6, -0.95);
	
	\draw (0,13) node[right] {$\mathcal{H}$};
	\draw [->,very thick] (0,-10) -- (0,13);
	
	\draw [red, dashed]  (-4,11) -- (6,11);	
	\draw [<->, red] (-3.5, -0.95) -- (-3.5, 11);
	\draw [red, dashed]  (-4,-0.95) -- (6,-0.95);
	\draw [red] (-3.5,6) node[left] {$R_{N}(z_1,t_1)$};
	
	\draw [green!40!gray, dashed]  (-4,2) -- (6,2);	
	\draw [<->, green!40!gray] (-4, 2) -- (-4, -2.4);
	\draw [color = green!40!gray, dashed]  (-4,-2.4) -- (6,-2.4);
	\draw [green!40!gray] (-4, 0) node[left] {$R_{N}(z_2,t_2)$};
	
	\draw [red] (0,11.5) node[left] {$(z_1,t_1)$};
	\draw [black] (0,2.5) node[right] {$(z_2,t_2)$};
	
	\filldraw [red] (0, 11) circle (1pt);
	\filldraw [red] (2.23, 0.7) circle (1pt);
	\filldraw [green!40!gray] (-3, -2) circle (1pt);
	\filldraw [black] (0, 2) circle (1pt);
	
	\draw [red] (0,11)to[curve through={(1.2,11.3) .. (1,9).. (-1, 8) .. (0.5, 5) .. (-1, 0)}](2.23, 0.7); 
	\draw [black] (0,2)to[curve through={(0.3,1) .. (0.7, 1.2) }](2.23, 0.7); 
	
	\draw [green!40!gray] (2.23, 0.7)to[curve through={(1.5, 1.5) .. (-1, -2) }](-3, -2); 
	
	\filldraw [white] (-3, -9) rectangle (3, -12);
\end{tikzpicture}
\hspace*{1cm}
\begin{tikzpicture}[scale= 0.4]
	\filldraw [white] (0, -3) rectangle (15,0);
	\draw [->,very thick] (15, 0) -- (15, 10);
	\draw [<-,very thick] (0, 0) -- (15, 0);
	\draw (0,0) node[below] {$\mathcal{H}$};
	\draw (15,10) node[right] {time};
	\draw (15,0) node[right] {$0$};
	\draw (15,9) node {$-$};
	\draw (15,8) node {$-$};
	\draw (15,6) node {$-$};
	\draw (15,3) node {$-$};
	\draw (15,9) node[right] {$n$};
	\draw (15,8) node[right] {$t_3$};
	\draw (15,6) node[right] {$t_2$};
	\draw (15,3) node[right] {$t_1$};
	\draw (4, 0) node {$+$};
	\draw (4,0) node[below] {$z_1$};
	\filldraw [color = red!20] (0, 3) rectangle (8, 9);
	\filldraw [red] (4, 3) circle (1pt);
	\draw [red] (4,2.5) node[right] {$(z_1,t_1)$};
	\draw [red, dashed] (4, 0) -- (4,3) ;
	\draw [red, <->] (0, 3.1) -- (4,3.1) ;
	\draw [red] (2, 3.1) node[above] {$R_{N}(z_1,t_1)$};

	\draw (7, 0) node {$+$};
	\draw (7,0) node[below] {$z_2$};
	\filldraw [color = green!40!gray!20] (5,6) rectangle (9, 9);
	\filldraw [green!40!gray] (7, 6) circle (1pt);
	\draw [green!40!gray] (7,5.5) node[left] {$(z_2,t_2)$};
	\draw [green!40!gray, dashed] (7, 0) -- (7,6) ;
	\draw [green!40!gray, <->] (7, 5.9) -- (9,5.9) ;
	\draw [green!40!gray] (9, 5.9) node[below] {$R_{N}(z_2,t_2)$};
	
	\fill[pattern = north east lines, pattern color= red] (5,6) rectangle (8, 9);
	
	\draw (11, 0) node {$+$};
	\draw (11,0) node[below] {$z_3$};
	\filldraw [color = blue!20] (8.5, 8) rectangle (13.5 , 9);
	\filldraw [blue] (11, 8) circle (1pt);
	\draw [blue] (11,7.5) node[left] {$(z_3,t_3)$};
	\draw [blue, dashed] (11, 0) -- (11,8) ;
	\draw [blue, <->] (11, 8.1) -- (13.5, 8.1) ;
	\draw [blue] (13, 8.1) node[below] {$R_{N}(z_3,t_3)$};
	
	\fill[pattern = north east lines, pattern color= green!40!gray] (8.5, 8) rectangle (9, 9);
	
	\draw [dashed, very thick] (15, 9) -- (0, 9);
	
\end{tikzpicture}
\vspace{-0.5cm}
\caption{Representations of a chain of changes}
\label{fig:ChainChanges}
\end{figure}

Figure \ref{fig:ChainChanges} properly illustrates our argument. On the left hand side, aggregates $A^{\dagger}_n[N_0]$ and $A^{\dagger}_n[N]$ are respectively in dark and light gray--the first one being included in the second one according to the natural coupling. The trajectory of the first particle, starting at $(z_1, t_1)$, is depicted in red. Since $\| z_1 \| > N_0$, it works only for the larger aggregate and creates a first discrepancy $x_1$. The trajectory of the second particle, starting at $(z_2, t_2)$, is depicted in black and green. This particle works for both aggregates until it visits $x_1$ for the first time (the black path): $x_1$ is then added to the smaller aggregate and is no longer a discrepancy. Thus, the second particle continues its trajectory (the green path) but now working only for the larger aggregate: it creates a new discrepancy $x_2$ between both aggregates. Note that the radii $R_{N}(z_1,t_1)$ and $R_{N}(z_2,t_2)$ are represented on the left hand side of Figure \ref{fig:ChainChanges}. The hatched area emphasizes the fact that the balls $\mathbb{B}(z_1,R_{N}(z_1,t_1))$ and $\mathbb{B}(z_2,R_{N}(z_2,t_2))$ overlap.

This phenomenon is perhaps better visualized in the right hand side of Figure \ref{fig:ChainChanges} where the extra time parameter is taken into account. Remark that not only the colored rectangles have to intersect, but they need to do so with respect to increasing times. This additional constraint will turn out to be crucial in our proof.\\

\textbf{Conclusion.} Given $(N_0,N)$ such that the forests $\fdag{n_0}[N]$ and $\fdag{n_0}[N_0]$ do not coincide on $\mathbb{Z}_{K_0}$, we can assert that there exists a sequence of $\kappa \geq 1$ particles coming from starting points $(z_i, t_i)_{1\leq i \leq \kappa}$ and satisfying the following conditions:
\begin{equation}
\label{ChainChanges2}
\left\{
\begin{array}{l}
	\text{$N_0 < \|z_1\| \leq N$ and $\|z_i\| \leq N_0$, for $2 \leq i \leq \kappa$} \\
	\text{$0 < t_1 < t_2 < \dots < t_{\kappa} < n_0$} \\
	\text{for $1 \leq i \leq \kappa-1$, $\mathbb{B}(z_i,R_{N}(z_i,t_i)) \cap \mathbb{B}(z_{i+1},R_{N}(z_{i+1},t_{i+1})) \not= \emptyset$} \\
	\text{$\mathbb{B}(z_\kappa,R_{N}(z_\kappa,t_\kappa)) \cap \mathbb{Z}_{K_0} \not= \emptyset$}
\end{array}
\right. 
\end{equation}

At this stage, the radii $R_{N}(z_i,t_i)$'s we consider are far from easy to handle, as the law of $R_{N}(z_i,t_i)$ strongly depends on the shape of $A^{\dagger}_{t_i}[N]$ as well as the starting point $(z_i,t_i)$, which is random. Consequently, the radii $R_{N}(z_i,t_i)$'s are neither independent, nor identically distributed. To overcome this latter obstacle, we will replace in the next section the aggregate $A^{\dagger}_{t_{i}^-}[N]$ involved in the definition of $R_{N}(z_i,t_i)$ with the larger aggregate $A^{\dagger}_T[\infty]$ (with $T \geq n_0$) whose distribution is translation invariant.

\subsection{Existence of infinite descending chains under the \textit{Absurd hypothesis}}
\label{sect:IDC}

Let $T \geq n_0$. Associated to the starting point $(z,t)$, with $z \in \mathcal{H}$ and $t \in [0,n_0]$, let us introduce the random radius $R((z,t),T)$ defined by
\[
R((z,t),T) := \min \big\{ r \in \mathbb{N} : \text{$\mathbb{B}(z,r)$ contains $p_{\mathcal{H}}(S_{z,t}(0)),p_{\mathcal{H}}(S_{z,t}(1)), \ldots ,p_{\mathcal{H}}(S_{z,t}(\tau'))$} \big\}
\]
where
\[
\tau' = \tau'(z,t,N) := \min\{ k : S_{z,t}(k) \notin A^{\dagger}_{T}[\infty]\}
\]
denotes the time at which the particle moving according to $S_{z,t}$ exits the infinite aggregate $A^{\dagger}_{T}[\infty]$. The new radius $R((z,t),T)$ is defined similarly as $R_{N}(z,t)$, but we are now considering the trajectory of $S_{z,t}$ until it exits $A^{\dagger}_T[\infty]$ rather than $A^{\dagger}_{t^-}[N]$. Growing the aggregate both in time and space, from $A^{\dagger}_{t^-}[N]$ to $A^{\dagger}_T[\infty]$, we then a.s have
\[
R_{N}(z,t) \leq R((z,t),T) ~.
\]

Hence, given $(N_0,N)$ such that the forests $\fdag{n_0}[N]$ and $\fdag{n_0}[N_0]$ do not coincide on $\mathbb{Z}_{K_0}$, we get the existence of a sequence of $\kappa \geq 1$ particles coming from starting points $(z_i, t_i)_{1\leq i \leq \kappa}$ and satisfying the following conditions:
\begin{equation}
\label{ChainChanges3}
\left\{
\begin{array}{l}
	\text{$N_0 < \|z_1\|$} \\
	\text{$0 < t_1 < t_2 < \dots < t_{\kappa} < n_0$} \\
	\text{for $1 \leq i \leq \kappa-1$, $\mathbb{B}(z_i,R((z_i,t_i),T)) \cap \mathbb{B}(z_{i+1},R((z_{i+1},t_{i+1}),T)) \not= \emptyset$} \\
	\text{$\mathbb{B}(z_\kappa,R((z_\kappa,t_\kappa),T)) \cap \mathbb{Z}_{K_0} \not= \emptyset$}
\end{array}
\right. 
\end{equation}

So, we now consider the space-time percolation model $\Sigma = \Sigma(n_0,T)$ defined as the collection of balls $\mathbb{B}(z,R((z,t),T))$, for any starting point $(z,t)$ with $z \in \mathcal{H}$ and $t \in [0,n_0]$. Note also that considering larger radii, i.e. replacing $R_{N}(z,t)$ with $R((z,t),T)$, all dependency on the parameter $N$ disappears in (\ref{ChainChanges3}) and this allows us to take advantage of the stationarity property of $A^{\dagger}_T[\infty]$, which will greatly facilitate the study of $\Sigma$.

In the rest of this section, we state two properties about the percolation model $\Sigma$; a finite moment property for its radii (Lemma \ref{lem: tail distribution}) and a finite degree property (Lemma \ref{lem: finite intersections}).

\medskip

Let us start by introducing the random radius $R_T(z)$, for $z \in \mathcal{H}$ and $T \geq n_0$, defined as
\begin{equation}
\label{eqn: R_T def}
R_T(z) := \max_{t \in \mathcal{N}_z([0,T])} R((z,t),T)
\end{equation}
where $\mathcal{N}_z$ denotes the \textsc{ppp} with intensity 1 associated to the source $z$. Also, $t \in \mathcal{N}_z([0,T])$ means that $t$ is a `top' of the \textsc{ppp} $\mathcal{N}_z$ occurring in $[0,T]$. Since the aggregate $A^{\dagger}_T[\infty]$ is translation invariant in distribution (w.r.t. translations of $\mathcal{H}$, see Proposition \ref{prop: invariance}) then all the radii $R_T(z)$, $z \in \mathcal{H}$, have the same distribution. Setting $R_T := R_T(0)$, the next result states that the $R_T(z)$'s admit finite moments.
\begin{lemma}
\label{lem: tail distribution}
Let $T \geq n_0$ and $L > 0$ be real numbers. There exists a constant $C > 0$ such that for any integer $M$,
\[
\mathbb{P} \big( R_T \geq M \big) \leq C M^{-L} ~.
\]
In particular $\mathbb{E}[(R_T)^k]$ is finite for any integer $k$.
\end{lemma}

\begin{proof}
Let $M \geq T$ be a positive integer, fix $\alpha \in (1-1/d,\ 1)$ and $\varepsilon >0$. Let us first write:
\begin{align}
	\label{eqn: R moments 2}
	\mathbb{P} \big( R_T \geq 2M \big) & \leq \mathbb{P} \left( R_T \geq 2M,\ \#\mathcal{N}_0([0,T]) \leq M ,\ \Over^{\dagger}_{\alpha}(M,T,\varepsilon)^c \right) \nonumber \\
	&+ \mathbb{P} \left( \#\mathcal{N}_0([0,T]) > M \right) + \mathbb{P} \left( \Over^{\dagger}_{\alpha}(M,T,\varepsilon) \right) ~,
\end{align}
where $\Over^{\dagger}_{\alpha}(M,T,\varepsilon) := \{A^{\dagger}_T[\infty] \cap \mathbb{Z}_M^c \nsubset \mathscr{C}_{\varepsilon}^{\alpha}\}$.
We know from Proposition \ref{thm: global upper bound} that the probability of $\Over^{\dagger}_{\alpha}(M,T,\varepsilon)$ is smaller than any power of $M^{-1}$ provided $M$ is sufficiently large. Additionally, we also know that $\mathbb{P}\left(\#\mathcal{N}_0([0,T]) > M\right)$ decreases faster than any power of $M^{-1}$, given $M$ is large enough. It then remains to focus on the first term of \eqref{eqn: R moments 2}. Hence,
\begin{multline*}
	\mathbb{P}\left( R_T \geq 2M,\ \#\mathcal{N}_0([0,T]) \leq M,\ \Over^{\dagger}_{\alpha}(M,T,\varepsilon)^c \right) \\
	\begin{split}
		&=\mathbb{P}\left(\exists t \in \mathcal{N}_0([0,T]),\ R((0,t), T) \geq 2M,\ \#\mathcal{N}_0([0,T]) \leq M,\ \Over^{\dagger}_{\alpha}(M,T,\varepsilon)^c \right) \\
		&\leq \mathbb{E}\left[ \sum_{t \in \mathcal{N}_0([0,T])} \mathbf{1}_{R((0,t), T) > 2M} \mathbf{1}_{\#\mathcal{N}_0([0,T]) \leq M}\mathbf{1}_{\Over^{\dagger}_{\alpha}(M,T,\varepsilon)^c} \right] \\
		&= \mathbb{E}\left[\mathbf{1}_{\#\mathcal{N}_0([0,T]) \leq M}\sum_{t \in \mathcal{N}_0([0,T])} \mathbb{P}\left({R((0,t), T) > 2M},\ \Over^{\dagger}_{\alpha}(M,T,\varepsilon)^c \; | \; \mathcal{N}_0([0,T]) \right) \right].
	\end{split}
\end{multline*}
Now, the event $\{ R((0,t), T) > 2M \}$ implies that the random walk traveled a distance at least $2M$, and in particular, that it traveled from levels $M$ to $2M$, while staying contained inside $A^{\dagger}_T[\infty]$. Now, working on $\Over^{\dagger}_{\alpha}(M,T,\varepsilon)^c$ implies that $A_T[\infty]$ is contained inside $\mathscr{C}_{\varepsilon}^{\alpha}$ above levels $M$. Hence working on both events implies that the random walk traveled from level $M$ to level $2M$, all while staying contained inside $\mathscr{C}_{\varepsilon}^{\alpha}$. By taking $j=1$ in \eqref{eqn: value of k chap2}, we know that the number of boxes between levels $M$ and $2M$, denoted by $k^o$, is such that 
\[
k^o \geq \frac{M^{1-\alpha}}{4\varepsilon}
\]
Hence, using a box argument similar to the one used in the proof of Theorem \ref{thm: stab agg dagg}, we can show that: 
\begin{align*}
	\sum_{t \in \mathcal{N}_0([0,T])} \mathbb{P}\left({R((0,t), T) > 2M},\ \Over^{\dagger}_{\alpha}(M,T,\varepsilon)^c \; | \; \mathcal{N}_0([0,T]) \right) &\leq \sum_{t \in \mathcal{N}_0([0,T])} \left(1-\frac{1}{4d^2}\right)^{k^o} \\
	&=\#\mathcal{N}_0([0,T])\exp\left(-c_0 M^{1-\alpha}\right),
\end{align*}
where $c_0 =-\frac{1}{4}\ln\left(1-\frac{1}{4d^2}\right)> 0$.
Hence, 
\begin{align*}
	\mathbb{P}\left( R_T \geq 2M,\ \#\mathcal{N}_0([0,T]) \leq M,\ \Over^{\dagger}_{\alpha}(M,T,\varepsilon)^c \right) &\leq \mathbb{E}\left[ \#\mathcal{N}_0([0,T])e^{-c_0M^{1- \alpha}} \mathbf{1}_{\#\mathcal{N}_0([0,T]) \leq M} \right]  \\
	& \leq M\exp(-c_0M^{1-\alpha}) ~,
\end{align*}
which decreases faster than any power of $M^{-1}$. Finally, we have proved that $\mathbb{P}(R_T \geq 2M)$ decreases faster than any power of $M^{-1}$ from which one easily concludes.
\end{proof}
Lemma \ref{lem: tail distribution} paves the way to a finite degree property for the percolation model $\Sigma = \Sigma(n_0,T)$. Precisely, any given ball $\mathbb{B}(z,R((z,t),T))$ of $\Sigma$ overlaps only a finite number of other balls $\mathbb{B}(z',R((z',t'),T))$ of $\Sigma$, with probability $1$.

\begin{lemma}
\label{lem: finite intersections}
Let $T \geq n_0$ and $z \in \mathcal{H}$. Then, with probability one, for any starting point $(z,t)$ with $t \leq n_0$, the number of starting points $(z',t')$, with $z' \in \mathcal{H}$ and $t' \leq n_0$, such that $\mathbb{B}(z,R((z,t),T)) \cap \mathbb{B}(z',R((z',t'),T)) \not= \emptyset$ is finite.
\end{lemma}

\begin{proof}
Since the radius $R_T(z)$ is by definition larger than any $R((z,t),T)$--see (\ref{eqn: R_T def})--it is enough to prove that a.s.
\[
\# \big\{ z' \in \mathcal{H} : \, \mathbb{B}(z,R_T(z)) \cap \mathbb{B}(z',R_T(z')) \not= \emptyset \big\} < \infty ~.
\]
Thus using the translation invariance property of the $R_T(z)$'s, it is also enough to state that the expectation
\[
\mathbb{E} \Big[ \# \big\{ z \in \mathcal{H} : \, \mathbb{B}(0,R_T(0)) \cap \mathbb{B}(z, R_T(z)) \not= \emptyset \big\} \Big]
\]
is finite. This follows from the next computation:
\begin{eqnarray*}
	\mathbb{E} \Big[ \# \big\{ z \in \mathcal{H} : \, \mathbb{B}(0,R_T(0)) \cap \mathbb{B}(z,R_T(z)) \not= \emptyset \big\} \Big] & = & \mathbb{E} \Big[  \sum_{z\in\mathcal{H}} \mathbf{1}_{R_T(0) + R_T(z) \geq \|z\|} \Big] \\
	& = & \sum_{z\in\mathcal{H}} \mathbb{P} \big( R_T(0) + R_T(z) \geq \|z\| \big) \\
	& \leq & \sum_{z\in\mathcal{H}} \mathbb{P} \big( \{ R_T(0) \geq \|z\|/2 \} \cup \{ R_T(z) \geq \|z\|/2 \} \big) \\
	& \leq & 2 \sum_{\ell = 0}^{\infty} \sum_{\|z\| = \ell} \mathbb{P} \big( R_T(0) \geq \|z\|/2 \big) \\
	& \leq & C_{d} \sum_{\ell = 0}^{\infty} \ell^{d-2} \, \mathbb{P} \big( R_T(0) \geq \ell \big) ~.
\end{eqnarray*}
This latter sum is finite thanks to Lemma \ref{lem: tail distribution}.
\end{proof}

\textbf{Conclusion.} Let us interpret the percolation model $\Sigma$ as the (undirected) graph $\mathcal{G}$ whose vertex set is given by the starting points $(z,t)$, with $z \in \mathcal{H}$ and $t \leq n_0$ and whose edge set is made of pairs $\{(z,t),(z',t')\}$ such that the corresponding balls $\mathbb{B}(z,R((z,t),T))$ and $\mathbb{B}(z',R((z',t'),T))$ overlap. Lemma \ref{lem: finite intersections} asserts that each vertex of the graph $\mathcal{G}$ almost surely has a finite degree.

Now, combining this finite degree property with the \textit{Absurd hypothesis}, we get the existence of an \textit{infinite descending chain}. Let us explain why. Recall that $T \geq n_0$ and $K_0$ are fixed parameters. The \textit{Absurd hypothesis} says that, with positive probability, for any integer $N_0$, there exists a sequence of $\kappa = \kappa(N_0) \geq 1$ particles satisfying (\ref{ChainChanges3}). With each of these sequences can be associated a cluster in the percolation model $\Sigma = \Sigma(n_0,T)$ joining the outside of $\mathbb{Z}_{N_0}$ to $\mathbb{Z}_{K_0}$ whose centers have increasing times. Roughly speaking, these sequences connect the strip $\mathbb{Z}_{K_0}$ to regions as far as desired through the percolation model $\Sigma$. Then, using the finite degree property and proceeding step by step, we can extract an infinite sequence of starting points $((z_i, t_i))_{i \geq 1}$ such that
\begin{equation}
\label{ChainChanges4}
\left\{
\begin{array}{l}
	\text{$\|z_i\| \to \infty$} \\
	\text{for any index $i$, $0 < t_{i+1} < t_i < n_0$} \\
	\text{for any index $i$, $\mathbb{B}(z_i,R((z_i,t_i),T)) \cap \mathbb{B}(z_{i+1},R((z_{i+1},t_{i+1}),T)) \not= \emptyset$} \\
	\text{$\mathbb{B}(z_1,R((z_1,t_1),T)) \cap \mathbb{Z}_{K_0} \not= \emptyset$}
\end{array}
\right. 
\end{equation}
Note that, w.r.t. (\ref{ChainChanges3}), we have reversed the indices of the time sequence $(t_i)_{i\geq 1}$. On the one hand, this sequence of starting points $((z_i, t_i))_{i \geq 1}$ is \textit{infinite} and means that $\Sigma$ percolates since it contains an unbounded cluster. On the other hand, it is also \textit{descending} since the sequence of starting times $(t_i)_{i \geq 1}$ is decreasing.

Let $\mathcal{K}_0 := \mathcal{H}_{K_0}$. The previous analysis allows us to say that, under the \textit{Absurd hypothesis}, the following holds
\begin{equation}
\label{eqn: contra bis}
\mathbb{P} \big( \mathrm{\textbf{Perco}}(n_0,\mathcal{K}_0,n_0) \big) > 0 ~,
\end{equation} 
where, for any $0 \leq t \leq T$ and any compact set $\mathcal{K} \subset \mathcal{H}$,
\[
\mathrm{\textbf{Perco}}(t,\mathcal{K},T) :=\left\lbrace\begin{array}{c}
\mbox{$\exists$ a sequence $((z_i,t_i))_{i\geq 1}$ of starting points s.t. $\|z_i\| \to \infty$,} \\
\mbox{$\mathbb{B}(z_1,R((z_1,t_1),T)) \cap \mathcal{K} \neq \emptyset$, and for any $i \geq 1$, $t_{i+1} < t_i < t$} \\
\mbox{and $\mathbb{B}(z_i,R((z_i,t_i),T)) \cap \mathbb{B}(z_{i+1},R((z_{i+1},t_{i+1}),T)) \neq \emptyset$}
\end{array}
\right\rbrace ~.
\]
In the next section, we will take advantage of the monotonicity of the time sequence $(t_i)_{i \geq 1}$ to establish that the infinite descending chain mentioned above appears instantaneously.

\subsection{Instantaneous percolation}
\label{subsec: instant perco}

For $T \geq t \geq 0$, let us set:
\begin{equation}
\label{eqn: perco def}
\Perco(t,T):= \bigcup_{\mathcal{K}} \Perco(t,\mathcal{K},T) ~,
\end{equation}
where the union above is taken over all compact subsets of $\mathcal{H}$. The event $\Perco(t,T)$ ensures the existence of an infinite descending chain made up with balls coming from the percolation model $\Sigma$, anchored at some (random) $\mathcal{K} \subset \mathcal{H}$. The event $\Perco(t,T)$  is monotone w.r.t. parameters $t$ and $T$:

\begin{lemma}
\label{lemma: perco}
Let $0 \leq t \leq t' \leq T \leq T'$. Then,
\[
\Perco(t,T) \subset \Perco(t,T') \; \mbox{ and } \; \Perco(t,T) \subset \Perco(t',T) ~.
\]
\end{lemma}

\begin{proof}
Let $0 \leq t \leq T \leq T'$. Recall that the infinite aggregate $A^{\dagger}_T[\infty]$ corresponds to particles launched from the whole set of sources $\mathcal{H}$ and during the time interval $[0,T]$. Hence, $A^{\dagger}_T[\infty]$ and $A^{\dagger}_{T'}[\infty]$ can be naturally coupled so that a.s. $A^{\dagger}_T[\infty] \subset A^{\dagger}_{T'}[\infty]$. This leads to $R((z_i,t_i),T) \leq R((z_i,t_i),T')$ whatever the starting point $(z_i,t_i)$ with $t_i \leq t$. So, $\Perco(t,\mathcal{K},T)$ is included in $\Perco(t,\mathcal{K},T')$, for any compact set $\mathcal{K}$, and the same holds for $\Perco(t,T)$ and $\Perco(t,T')$.

The second inclusion is easy to prove. Indeed, replacing $t$ with $t' \geq t$ amounts to relax the upper bound on the decreasing sequence $(t_i)_{i\geq 1}$ appearing in the event $\Perco(t,\mathcal{K},T)$. So $\Perco(t,\mathcal{K},T)$ is included in $\Perco(t',\mathcal{K},T)$, for any $\mathcal{K}$, and the same holds for $\Perco(t,T)$ and $\Perco(t',T)$.
\end{proof}

From now on, we fix $T := n_0+1$. (We pick this value of $T$ to ensure that $T>n_0$. This condition will be crucial for our arguments detailed below.) Let us introduce the critical percolation time as
\[
t_c = t_c(T) := \inf \big\{ 0 \leq t \leq T: \mathbb{P} \big( \Perco(t, T) \big) > 0 \big\} \in [0, T] ~.
\]
Combining to the monotone property given by Lemma \ref{lemma: perco}, we can then write:
\begin{equation}
\label{eqn: tc}
\left\{
\begin{array}{ll}
	\mathbb{P}\left(\Perco(t, T)\right) = 0$ for any $0 \leq t < t_c, \\
	\mathbb{P}\left(\Perco(t, T)\right) > 0$ for any $t_c < t \leq T.
\end{array}
\right. 
\end{equation}
Both statements of (\ref{eqn: tc}) become meaningful provided the critical percolation time $t_c$ is non-trivial, i.e. different from $0$ and $T$. This is where the \textit{Absurd hypothesis} steps in since (\ref{eqn: contra bis}) and Lemma \ref{lemma: perco} imply that
\[
\mathbb{P} \big( \Perco(n_0,T) \big) \geq \mathbb{P} \big( \Perco(n_0,n_0) \big) \geq \mathbb{P} \big( \Perco(n_0,\mathcal{K}_0,n_0) \big) > 0
\]
that is to say
\begin{equation}
\label{t_c bound}
t_c \leq n_0 < T ~.
\end{equation}
Actually, the condition $t_c < T$--ensured by the \textit{Absurd hypothesis}--leads to a phenomenon of \textit{instantaneous percolation} for the percolation model $\Sigma$ that we describe below.

Let us first assume that the critical percolation time $t_c$ is positive. The case $t_c = 0$ is similar and will be treated after. Hence for any $\varepsilon > 0$ small enough (i.e. such that $t_c - \varepsilon \geq 0$ and $t_c + \varepsilon \leq T$), we have $\mathbb{P}(\Perco(t_c + \varepsilon,T)) > 0$ while $\mathbb{P}(\Perco(t_c - \varepsilon,T)) = 0$ by (\ref{eqn: tc}), meaning that
\[
\mathbb{P} \big( \Perco(t_c + \varepsilon, T) \! \setminus \! \Perco(t_c - \varepsilon, T) \big) > 0 ~.
\]
Let us analyze what happens on the event $\Perco(t_c + \varepsilon, T)\!\setminus\!\Perco(t_c - \varepsilon, T)$. First of all, the event $\Perco(t_c + \varepsilon, T)$ asserts the existence of an infinite descending chain in $\Sigma$ associated to a sequence of starting points $((z_i,t_i))_{i \geq 1}$ with, for any $i \geq 1$, $t_{i+1} < t_i < t_c+\varepsilon$. Besides, the event $\Perco(t_c - \varepsilon, T)^c$ forces all the $t_i$'s to be larger than $t_c - \varepsilon$. Otherwise there would exist some index $i_0$ such that $t_{i_0} < t_c - \varepsilon$. In that case, one would have an infinite descending chain associated to the sequence of starting points  $((z_i,t_i))_{i \geq i_0}$ anchored at $\mathbb{B}(z_{i_0},R((z_{i_0},t_{i_0}),T))$ and such that for any $i$, $t_{i+1} < t_i \leq t_{i_0} < t_c - \varepsilon$, meaning that $\Perco(t_c - \varepsilon, T)$ actually occurs. In conclusion, the sequence of starting points $((z_i,t_i))_{i \geq 1}$ satisfies $t_c-\varepsilon < t_{i+1} < t_i < t_c+\varepsilon$ for any index $i$. This means that this infinite descending chain appears entirely during the time interval $[t_c-\varepsilon,t_c+\varepsilon]$, with length $2 \varepsilon$, for $\varepsilon > 0$ arbitrarily small. This is why we talk about instantaneous percolation. Let us emphasize that the previous argument works because the radii $R((z_{i},t_{i}),T)$'s remain the same in both events $\Perco(t_c - \varepsilon,T)$ and $\Perco(t_c+ \varepsilon, T)$ since they have the same second parameter $T$.

In the particular case $t_c = 0$, we know that there is no percolation at the critical time $t_c$ since no particles have been launched yet. So we apply the previous analysis with $t_c + \varepsilon = \varepsilon$ and $t_c - \varepsilon$ replaced with $0$. As before, an infinite descending chain appears entirely during the time interval $[0,\varepsilon]$, for any small $\varepsilon > 0$.

In order to simplify the argumentation and lighten notation, let us reduce the first case $t_c > 0$ to the second one $t_c = 0$. Indeed, because the radii $R((z_{i},t_{i}),T)$'s remain unchanged in the events $\Perco(\cdot,T)$, the time interval during which the infinite descending chain entirely occurs matters in distribution only through its length (the \textsc{ppp} $\mathcal{N}_z$'s have stationary increments). Henceforth,
\[
\Big( \forall \varepsilon > 0 , \; \mathbb{P} \big( \Perco(t_c + \varepsilon, T) \! \setminus \! \Perco(t_c - \varepsilon, T) \big) > 0 \Big) \; \Longrightarrow \Big( \forall \varepsilon > 0 , \; \mathbb{P} \big( \Perco(\varepsilon, T) \big) > 0 \Big) ~.
\]
\medskip

\textbf{Conclusion.} We have taken advantage of the monotonicity in time of the sequence of starting points $((z_i,t_i))_{i \geq 1}$ to state that the corresponding infinite descending chain in the percolation model $\Sigma$ appears instantaneously. Precisely, we have proven under the \textit{Absurd hypothesis} that
\begin{equation}
\label{Perco-eps}
\forall \varepsilon > 0 , \; \mathbb{P} \big( \Perco(\varepsilon, T) \big) > 0 ~.
\end{equation}
We can now forget the time dimension of our model: (\ref{Perco-eps}) allows us to exhibit a \textit{supercritical} discrete Boolean model, denoted by $\hat{\Sigma}_{\varepsilon}$, whose intensity tends to $0$ as $\varepsilon \to 0$.

Let us set, for any source $z \in \mathcal{H}$,
\[
Y_z := \mathbf{1}_{\# \mathcal{N}_z([0,\varepsilon]) > 0} ~.
\]
So $Y_z = 1$ means that at least one particle has been emitted during the time interval $[0,\varepsilon]$. The $Y_z$'s are Bernoulli random variables with common parameter
\begin{equation}
\label{eqn: p epsilon value}
p_{\varepsilon} := \mathbb{P}\left( Y_z = 1 \right) = 1 - e^{-\varepsilon}
\end{equation}
which tends to $0$ as $\varepsilon \to 0$. By hypothesis on the \textsc{ppp} $\mathcal{N}_z$'s, the random variables $Y_z$, $z \in \mathcal{H}$, are also independent from each other. Let us denote by
\[
\chi_{\varepsilon} := \big\{ z \in \mathcal{H} : \, Y_z = 1 \big\}
\]
the random set of emitting sources during the time interval $[0,\varepsilon]$. The set $\chi_{\varepsilon}$ can be interpreted as a discrete \textsc{ppp} on $\mathcal{H}$ with intensity $p_\varepsilon$: it will play the role of the center set for the Boolean model $\hat{\Sigma}_{\varepsilon}$. In all that follows, let us keep in mind that when $\varepsilon \to 0$ there are very few centers, and so very few balls, in the Boolean model $\hat{\Sigma}_{\varepsilon}$.

Thus, for $z \in \mathcal{H}$, let us define in the same spirit of \eqref{eqn: R_T def} the radius $R_T(z ; \varepsilon)$ as
\[
R_T(z ; \varepsilon) := \max_{t \in \mathcal{N}_z([0,\varepsilon])} R((z,t), T) ~.
\]
The only difference between radii $R_T(z;\varepsilon)$ and $R_T(z)$ defined in \eqref{eqn: R_T def} is that $R_T(z;\varepsilon)$  involves particles launched during $[0,\varepsilon]$ while $R_T(z)$ involves particles of $[0,T]$, so that $R_T(z;\varepsilon) \leq R_T(z)$ (and then $R_T(z;\varepsilon)$ also satisfies the finite moment property of Lemma \ref{lem: tail distribution}). They both refer to radii $R(\cdot,T)$ defined from the aggregate $A^{\dagger}_{T}[\infty]$ (with $T = n_0+1$).

We are now ready to define the (discrete) Boolean model $\hat{\Sigma}_{\varepsilon}$ by
\[
\hat{\Sigma}_{\varepsilon} := \bigcup_{z \in \chi_{\varepsilon}} \mathbb{B} \big( z,R_T(z;\varepsilon) \big) ~.
\]
Statement (\ref{Perco-eps}) says that, for any $\varepsilon > 0$ and with positive probability, there exists a sequence of starting points $((z_i,t_i))_{i\geq 1}$ such that $\|z_i\| \to \infty$ and, for any $i \geq 1$, $t_{i+1} < t_i < \varepsilon$ and the balls $\mathbb{B}(z_i,R((z_i,t_i),T))$ and $\mathbb{B}(z_{i+1},R((z_{i+1},t_{i+1}),T))$ overlap. So, the $z_i$'s all belong to $\chi_\varepsilon$ and the larger balls $\mathbb{B}(z_i,R_T(z_i;\varepsilon))$ and $\mathbb{B}(z_{i+1},R_T(z_{i+1};\varepsilon))$ overlap too. In other words, the Boolean model $\hat{\Sigma}_{\varepsilon}$ contains an unbounded cluster.

Finally, the argumentation of the whole Section \ref{sec: chains to perco} provides the following result:

\begin{proposition}
\label{prop:ConcSection3}
Under the \textit{Absurd hypothesis} (\ref{Absurd-Hypo}), the following holds:
\begin{equation}
	\label{Perco-Sigma-Hat}
	\forall \varepsilon > 0 , \; \mathbb{P} \Big( \text{$\hat{\Sigma}_{\varepsilon}$ percolates} \Big) > 0 ~.
\end{equation}
\end{proposition}


\section{A multiscale argument}
\label{sec: JB}

In this section, we study the $(d-1)$-dimensional, discrete Boolean model
\[
\hat{\Sigma}_{\varepsilon} = \bigcup_{z \in \chi_{\varepsilon}} \mathbb{B} \big( z , R_T(z;\varepsilon) \big) ~,
\]
defined in the previous section, and whose intensity $p_\varepsilon$ tends to $0$ as $\varepsilon \to 0$. We adapt the strategy of \cite{gouere2009subcritical} to our context to state :

\begin{proposition}
\label{prop:NoPerco}
For any $\varepsilon > 0$ small enough, with probability 1, the Boolean model $\hat{\Sigma}_{\varepsilon}$ does not percolate.
\end{proposition}

To implement the strategy of \cite{gouere2009subcritical}, we need to make our model more local. This is the role of $\hat{\Sigma}_{\varepsilon}^\text{\tiny{loc}}$ defined below. For that purpose, the stabilization result for the infinite aggregate $A^{\dagger}_n[\infty]$ (Theorem \ref{thm: stab agg dagg}) is an important ingredient to ensure that the localized model $\hat{\Sigma}_{\varepsilon}^\text{\tiny{loc}}$ is a good approximation of $\hat{\Sigma}_{\varepsilon}$.

\begin{proof}[Proof of Theorem~\ref{thm: forest stabilization}]
Propositions \ref{prop:ConcSection3} and \ref{prop:NoPerco} together say that the \textit{Absurd hypothesis} (\ref{Absurd-Hypo}) leads to a contradiction. We then conclude that, with probability $1$, there exists $N_0$ such that, for any $N \geq N_0$, the forests $\fdag{n_0}[N]$ and $\fdag{n_0}[N_0]$ coincide on the strip $\mathbb{Z}_{K_0}$. This concludes the proof of Theorem~\ref{thm: forest stabilization}.
\end{proof}

\subsection{The localized Boolean models $\hat{\Sigma}_{\varepsilon}^\text{\tiny{loc}}$}

Recall that $T = n_0+1$ and $\varepsilon > 0$ is thought to be small. Given a source $x \in \mathcal{H}$, we denote by $A^{\dagger}_{T}[\mathbb{B}(x,20M)]$ the aggregate $A^{\dagger}_T[\cdot]$ using only sources of $\mathbb{B}(x,20M)$. Associated to the starting point $(z,t)$, with $z \in \mathcal{H}$ and $t \in [0,\varepsilon]$, let us introduce the local radius $R^\text{\tiny{loc}}_{x,M}((z,t),T)$ defined by
\begin{equation}
\label{defi-Rloc}
R^\text{\tiny{loc}}_{x,M}((z,t),T) := \min \big\{ r \in \mathbb{N} : \text{$\mathbb{B}(z,r)$ contains $p_{\mathcal{H}}(S_{z,t}(0)),p_{\mathcal{H}}(S_{z,t}(1)), \ldots ,p_{\mathcal{H}}(S_{z,t}(\tau''))$} \big\}
\end{equation}
where
\[
\tau'' := \min\{ k : S_{z,t}(k) \notin A^{\dagger}_{T}[\mathbb{B}(x,20M)]\}
\]
denotes the time at which the particle moving according to $S_{z,t}$ exits $A^{\dagger}_{T}[\mathbb{B}(x,20M)]$. Taking the maximum over starting points $(z,t)$ with $t \in  \mathcal{N}_z([0,\varepsilon])$, we get
\[
R^\text{\tiny{loc}}_{T}(z;\varepsilon) = R^\text{\tiny{loc}}_{T,x,M}(z;\varepsilon) := \max_{t \in \mathcal{N}_z([0,\varepsilon])} R^\text{\tiny{loc}}_{x,M}((z,t),T) ~.
\]
Let us now define the \textit{localized Boolean model} $\hat{\Sigma}_{\varepsilon}^\text{\tiny{loc}}(x,M)$ which is a version of $\hat{\Sigma}_{\varepsilon}$ localized to the neighborhood of $x$:
\[
\hat{\Sigma}_{\varepsilon}^\text{\tiny{loc}} = \hat{\Sigma}_{\varepsilon}^\text{\tiny{loc}}(x,M) := \bigcup_{z \in \chi_{\varepsilon} \cap \mathbb{B}(x,10M)} \mathbb{B} \Big( z, R^\text{\tiny{loc}}_{T,x,M}(z;\varepsilon) \Big) ~.
\]
Unlike $\hat{\Sigma}_{\varepsilon}$, the localized Boolean model $\hat{\Sigma}_{\varepsilon}^\text{\tiny{loc}}$ is local in the following sense. It only uses sources of $\chi_{\varepsilon}$ contained inside $\mathbb{B}(x,10M)$. The radii that it considers depend only on $A^{\dagger}_{T}[\mathbb{B}(x,20M)]$.

As in \cite{gouere2009subcritical}, we consider the event
\[
G_{\varepsilon}(x,M) := \left\{
\begin{array}{c}
\text{The connected component of $x$ in $\hat{\Sigma}_{\varepsilon}^\text{\tiny{loc}}(x,M) \cup \mathbb{B}(x,M)$} \\
\text{is not included in $\mathbb{B}(x,8M)$} \\
\end{array}
\right\} ~. 
\]
Since our model is translation invariant, we have for any $x$:
\[
\pi_{\varepsilon}(M) := \mathbb{P} \big( G_{\varepsilon}(0,M) \big) = \mathbb{P} \big( G_{\varepsilon}(x,M) \big) ~.
\]

Let us denote by $\hat{C}_{\varepsilon}(0)$ the cluster of the source $0$ in the discrete Boolean model $\hat{\Sigma}_{\varepsilon}$ whose diameter $\text{diam}\,\hat{C}_{\varepsilon}(0)$ is defined as the minimal integer $r$ such that $\hat{C}_{\varepsilon}(0) \subset \mathbb{B}(0,r)$. The next two results allow us to conclude. Notably, Proposition \ref{prop: model link}, allows us to connect the Boolean model $\hat{\Sigma}_{\varepsilon}$ with its localized counterpart $\hat{\Sigma}_{\varepsilon}^\text{\tiny{loc}}$.

\begin{proposition}
\label{prop: model link}
For all $L \geq 1$, there exists a constant $C = C_{T,d,L} > 0$ such that for all $M \geq 1$,
\[
\mathbb{P} \big( \mathrm{diam}\,\hat{C}_{\varepsilon}(0) \geq 8M \big) \leq \pi_{\varepsilon}(M) + \frac{C}{M^L} ~.
\]
\end{proposition}

\begin{proposition}
\label{prop: pi inequality}
For any $\varepsilon > 0$ small enough we have
\[
\liminf_{M \to \infty} \pi_{\varepsilon}(M) = 0 ~.
\]
\end{proposition}
\begin{proof}[Proof of Proposition \ref{prop:NoPerco}]
Taking $\varepsilon > 0$ small enough according to Proposition \ref{prop: pi inequality}, the two previous results imply that
\begin{eqnarray*}
	\lim_{M \to \infty} \mathbb{P} \big( \text{diam}\,\hat{C}_{\varepsilon}(0) \geq 8M \big) & = & \liminf_{M \to \infty} \mathbb{P} \big( \text{diam}\,\hat{C}_{\varepsilon}(0) \geq 8M \big) \\
	& \leq & \liminf_{M \to \infty} \pi_{\varepsilon}(M) \; = \; 0 ~,
\end{eqnarray*}
meaning that the source $0$ almost surely belongs to a finite connected component in $\hat{\Sigma}_{\varepsilon}$. This discrete Boolean model being translation invariant in distribution, we conclude that it a.s. admits only finite connected components. I.e. $\hat{\Sigma}_{\varepsilon}$ does not percolate with probability $1$. 
\end{proof}

The next two sections are devoted to the proofs of Propositions \ref{prop: model link} and \ref{prop: pi inequality}.

\subsection{Proof of Proposition~\ref{prop: model link}}

Let us introduce for any $M \geq 1$ the event $H(M)$ defined by
\[
H(M) := \Big\{ \mbox{each ball of $\hat{\Sigma}_{\varepsilon}$ intersecting $\mathbb{B}(0,10M)$ has a radius smaller than $M$} \Big\} ~.
\]
The following lemma gives a control for the probability of $H(M)$.

\begin{lemma}
\label{lemma: radius control}
For any $L \geq 1$, there exists a positive constant $C = C(d,\varepsilon,L)$ such that for any $M\geq 1$, 
\[
\mathbb{P} \big( H(M)^c \big) \leq \frac{C}{M^L} ~.
\]
\end{lemma}

Let us first prove Proposition~\ref{prop: model link} from Lemma \ref{lemma: radius control} and in a second step, Lemma \ref{lemma: radius control} will be proven.

\begin{proof}[Proof of Proposition~\ref{prop: model link}]
We define the event 
\[
\mathcal{G}(M) := H(M) \cap \left\{ A^{\dagger}_T[\infty]\cap\mathbb{Z}_{10M} = A^{\dagger}_T[\mathbb{B}(0,20M)]\cap\mathbb{Z}_{10M} \right\}.
\]
Now, we use the stabilization result Theorem~\ref{thm: stab agg dagg} which, combined to Lemma~\ref{lemma: radius control}, provides the following control on the probability of $\mathcal{G}(M)$:
\[
\mathbb{P} \left(\mathcal{G}(M)\right) \geq 1 - \frac{C}{M^L},
\]
for some positive constant $C$ and for any $M,L\geq 1$. Recalling that $\pi_{\varepsilon}(M)$ denotes the probability of $G_{\varepsilon}(0,M)$, it is then sufficient to prove the inclusion: 
\[
\Big\{ \text{diam}\,\hat{C}_{\varepsilon}(0) \geq 8M \Big\} \cap \mathcal{G}(M) \subset G_{\varepsilon}(0,M) ~.
\]
Let us assume that the event $\{\text{diam}\,\hat{C}_{\varepsilon}(0) \geq 8M \} \cap \mathcal{G}(M)$ holds. This implies the existence of a cluster $\mathcal{C}$ of $\hat{\Sigma}_{\varepsilon}$ containing $0$ and going beyond $\mathbb{B}(0,8M)$. Each ball of $\mathcal{C}$ overlaps $\mathbb{B}(0,8M)$ and then has a radius smaller than $M$ thanks to $H(M)$. So they are included in $\mathbb{B}(0,10M)$ and their centers belong to $\chi_{\varepsilon} \cap \mathbb{B}(0,10M)$.

Given a vertex $z$, the radii $R^\text{\tiny{loc}}_{T}(z;\varepsilon)$ and $R_T(z;\varepsilon)$ used resp. for $\hat{\Sigma}_{\varepsilon}^\text{\tiny{loc}}$ and $\hat{\Sigma}_{\varepsilon}$ are possibly different since they resp. refer to $A^{\dagger}_T[\mathbb{B}(0,20M)]$ and $A^{\dagger}_T[\infty]$. However, we prove that on the event $\{\text{diam}\,\hat{C}_{\varepsilon}(0) \geq 8M \} \cap \mathcal{G}(M)$, these radii are equal for the balls involved in the cluster $\mathcal{C}$. Indeed, these balls are included in $\mathbb{B}(0,10M)$ and, on $\mathcal{G}(M)$, the aggregates $A^{\dagger}_T[\infty]$ and $A^{\dagger}_T[\mathbb{B}(0,20M)]$ coincide on $\mathbb{Z}_{10M}$. So the radii $R^\text{\tiny{loc}}_{T}(z;\varepsilon)$ and $R_T(z;\varepsilon)$ coincide for each center $z$ of balls involved in $\mathcal{C}$. Therefore, on $\{\text{diam}\,\hat{C}_{\varepsilon}(0) \geq 8M \} \cap \mathcal{G}(M)$, there exists a cluster of balls of $\hat{\Sigma}_{\varepsilon}$ such that:
\begin{itemize}
	\item all these balls have their centers in $\chi_{\varepsilon} \cap \mathbb{B}(0,10M)$;
	\item all these balls have their radii given by $R^\text{\tiny{loc}}_{T}(z;\varepsilon)$;
	\item the cluster contains $0$ and goes beyond $\mathbb{B}(0,8M)$. 
\end{itemize}
Then, this cluster of balls is also a cluster of $\hat{\Sigma}_{\varepsilon}^\text{\tiny{loc}}$ and the event $G_{\varepsilon}(0,M)$ occurs.
\end{proof}

\begin{proof}[Proof of Lemma~\ref{lemma: radius control}]
Fix $M,L \geq 1$. Let us introduce the following events:
\[
\Centered(M) := \left\lbrace\begin{array}{c}
	\mbox{ There exists a ball of $\hat{\Sigma}_{\varepsilon}$ centered \textbf{inside} $\mathbb{B}(0, 20M)$} \\
	\mbox{ that intersects $\mathbb{B}(0,10M)$ with a radius greater than $M$ }
\end{array}
\right\rbrace ,
\]
\[
\Uncentered(M) := \left\lbrace\begin{array}{c}
	\mbox{ There exists a ball of $\hat{\Sigma}_{\varepsilon}$ centered \textbf{outside} $\mathbb{B}(0, 20M)$} \\
	\mbox{ that intersects $\mathbb{B}(0,10M)$}
\end{array}
\right\rbrace ~.
\]
Hence, the event $H(M)^c$ can be written as the union $\Centered(M) \cup \Uncentered(M)$ and we have to work with the probability of these two events. We begin by handling the event $\Centered(M)$. Let us write:
\begin{align*}
	\mathbb{P}\left(\Centered(M)\right) &\leq \mathbb{P}\left(\bigcup_{z \in \chi_{\varepsilon} \cap \mathbb{B}(0, 20M)} \big\{ R_T(z;\varepsilon) \geq M \big\} \right) \\
	& \leq \sum_{z \in \mathbb{B}\left( 0, 20M \right)} \mathbb{P}\left( R_T(z;\varepsilon) \geq M \right) \\
	& \leq C_d  M^{d-1} \mathbb{P}\left( R_T(0;\varepsilon) \geq M \right) ~.
\end{align*}
Using Lemma \ref{lem: tail distribution}, $\mathbb{P}\left(R_T(0;\varepsilon) \geq M\right)$ decreases faster than any power of $M^{-1}$, which handles this case.

We switch our focus to the event $\Uncentered(M)$ which is trickier to handle since it deals with an infinite number of particles outside of $\mathbb{B}(0,20M)$. Let us recall the event
\[
\Over^{\dagger}_{\alpha}(10M,T,\delta) = \big\{ A^{\dagger}_T[\infty] \cap \mathbb{Z}_{10M}^c \nsubset \mathscr{C}_{\delta}^{\alpha} \big\}
\]
introduced in Section \ref{sect:resultsAdag}, where
\[
\mathscr{C}_{\delta}^{\alpha} = \bigcup_{\ell \geq 0} \left\{ z \in \mathbb{Z}^d,\ \|p_{\mathcal{H}}(z)\| = \ell,\ |z_1| \leq \delta \ell^{\alpha}  \right\} ~.
\]
For $\delta > 0$ and $\alpha \in (1-1/d,1)$, we know thanks to the global upper bound Proposition \ref{thm: global upper bound} that the probability of $\Over^{\dagger}_{\alpha}(10M,T,\delta)$ decreases faster than any power of $M^{-1}$. So we can restrict our attention to the event $\Uncentered(M) \cap \Over^{\dagger}_{\alpha}(10M,T,\delta)^c$. The event $\Uncentered(M)$ provides the existence of a ball of the Boolean model $\hat{\Sigma}_{\varepsilon}$ that intersects $\mathbb{B}(0,10M)$ and whose center is beyond level $20M$. This ball is due to a particle starting (during the time interval $[0,\varepsilon]$) from a source beyond level $20M$ and visiting the strip $\mathbb{Z}_{10M}$ before exiting the aggregate $A^{\dagger}_T[\infty]$. Thanks to $\Over^{\dagger}_{\alpha}(10M,T,\delta)^c$, we can assert that the random walk associated to that particle starts from a level greater than $20M$ and visits $\mathbb{Z}_{10M}$ before exiting the cone $\mathscr{C}_{\delta}^{\alpha}$. This implies the event $D_{10M}^c$ introduced in Section \ref{subsec: proof of stab agg dagg} whose probability is smaller than any power of $M^{-1}$ thanks to Lemma \ref{prop: particle stab}. This concludes the proof.
\end{proof}

\subsection{Proof of Proposition~\ref{prop: pi inequality}}

This section is an adaptation of \cite{gouere2009subcritical}. Lemmas \ref{lemma: pi squared} and \ref{lemma: pi bound} together imply, by induction, that $\pi_{\varepsilon}(10^n M) \to 0$ as $n \to \infty$ for some $M$ and $\varepsilon$ small enough. Lemma \ref{lemma: pi squared} is the induction step, allowing to go from scale $10^n M$ to $10^{n+1} M$, while Lemma \ref{lemma: pi bound} is the base step.

\begin{lemma}
\label{lemma: pi squared}
For any $L\geq 1$, there exist positive constants $c=c(d,L)$ and $C=C(T,d,L)$ such that for any $M\geq 1$, 
\[
\pi_{\varepsilon}(10M) \leq c \pi_{\varepsilon}(M)^2 + \frac{C}{M^L} ~.
\]
\end{lemma}

\begin{lemma}
\label{lemma: pi bound}
There exists a positive constant $C'=C'(d)$ such that for all $M \geq 1$ and all $\varepsilon >0$:
\[
\pi_{\varepsilon}(M) \leq \varepsilon C' M^{d-1} ~.
\]
\end{lemma}

Let us prove Proposition~\ref{prop: pi inequality} from Lemmas \ref{lemma: pi squared} and \ref{lemma: pi bound}.

\begin{proof}[Proof of Proposition~\ref{prop: pi inequality}]
This is an adaptation of Lemma 3.7 of \cite{gouere2009subcritical}. Setting $f(M) := c \pi_{\varepsilon}(M)$ and $g(M) := 10 c C / M$, Lemma \ref{lemma: pi squared} provides
\begin{equation}
	\label{eqn: f squared}
	f(10M) \leq f(M)^2 + g(10M) ~.
\end{equation}
Since $g(M) \to 0$, we can pick $M_0$ such that, for any $M \geq M_0$, $g(M) \leq 1/4$. Thus, using Lemma~\ref{lemma: pi bound}, there exists $\varepsilon_0 = \varepsilon_0(M_0) > 0$ small enough such that, for $\varepsilon \leq \varepsilon_0$ and $M \leq M_0$,
\[
f(M) = c \pi_{\varepsilon}(M) \leq c \varepsilon_0 C' M_0^{d-1} \leq 1/2 ~.
\]
So, the function $f$ is bounded by $1/2$ on the interval $(0,M_0]$. Let us first extend this bound on $(M_0, 10M_0]$. To do it, let us fix $\varepsilon \in (0,\varepsilon_0)$. Thanks to \eqref{eqn: f squared}, we can write for any $M \in (M_0, 10M_0]$:
\[
f(M) \leq f(M/10)^2 + g(M) \leq \Big( \frac{1}{2} \Big)^2 + \frac{1}{4} = \frac{1}{2} ~.
\]
Iterating this argument, we prove by induction that $f(M) \leq 1/2$ for any $M > 0$.

As a consequence, we deduce from \eqref{eqn: f squared} that, for any $M > 0$,
\[
f(10M) \leq \frac{1}{2} f(M) + g(10M) \leq \frac{1}{2} + g(10M) ~,
\]
from which is not difficult to get, once again by induction, that the following holds for any integer $n$,
\begin{equation}
	\label{eqn: f inequality}
	f(10^n M) \leq \frac{1}{2^{n}} + g(10^n M) + \frac{g(10^{n-1}M)}{2} + \ldots + \frac{g(10 M)}{2^{n-1}} ~.
\end{equation}
Henceforth, using (\ref{eqn: f inequality}) and $g(10^nM) \to 0$ as $n \to \infty$, we prove that $f(10^nM)$ tends to $0$ as $n \to \infty$, which is the searched result.
\end{proof}

\begin{proof}[Proof of Lemma~\ref{lemma: pi squared}]
Let us first consider the event $F(M)$ defined by
\[
F(M) := \left\{ \forall z \in \chi_{\varepsilon}\cap \mathbb{B}(0,100M) , \, R_{T}(z;\varepsilon) \leq M \right\}
\]
for which, any ball of $\hat{\Sigma}_\varepsilon$ centered at some $z \in \chi_{\varepsilon}\cap \mathbb{B}(0,100M)$ has a radius $R_{T}(z;\varepsilon)$ smaller than $M$. Since these radii are identically distributed and satisfy Lemma \ref{lem: tail distribution}), the event $F(M)$ is very likely. For any $L \geq 1$, there exists a constant $C > 0$ such that for any $M \geq 1$, 
\begin{equation}
	\label{F(M)likely}
	\mathbb{P} \big( F(M)^c \big) \leq \frac{C}{M^L} ~.
\end{equation}

Besides, for $M \geq 1$, let us consider the event
\[
\Stab(0, 100M) := \bigcap_{z \in \mathbb{B}(0,100M)} \Stab^{M}(z)
\]
where, for any $z \in \mathcal{H}$,
\[
\Stab^M(z) := \left\lbrace \begin{array}{c}
	A^{\dagger}_T[\mathbb{B}(z,20M)] \cap \left(\mathbb{Z}\times\mathbb{B}(z,10M)\right) = A^{\dagger}_T[S] \cap \left(\mathbb{Z}\times\mathbb{B}(z,10M)\right),\\
	\text{for any $S \subset \mathcal{H}$ such that $\mathbb{B}(z,20M) \subset S$.} \end{array}
\right\rbrace ~.
\]
In particular, $\Stab^M(0)$ means that $A^{\dagger}_T[\mathbb{B}(0,20M)]$ (also denoted by $A^{\dagger}_T[20M]$ for short) and $A^{\dagger}_T[S]$ coincide on $\mathbb{Z}_{10M}$, for any $\mathcal{H}_{20M} \subset S \subset \mathcal{H}$. Let us prove that for any $L \geq 1$, there exists a constant $C = C(L) > 0$ such that for any $M \geq 1$, 
\begin{equation}
	\label{Stab-likely}
	\mathbb{P} \big( \Stab(0,100M)^c \big) \leq \frac{C}{M^L} ~.
\end{equation}
Let us start by writing, using translation invariance of the aggregates $A^{\dagger}_T[\cdot]$, 
\[
\mathbb{P} \big( \Stab(0,100M)^c \big) \leq \sum_{z \in \mathbb{B}(0,100M)} \mathbb{P} \big( \Stab^M(z)^c \big) \leq c (100M)^{d-1} \mathbb{P} \big( \Stab^M(0)^c \big) ~.
\]
So it is sufficient to show that $\Stab^M(0)^c$ has a probability decreasing faster than any power of $M^{-1}$. The \textit{Aggregate stabilization} result Theorem \ref{thm: stab agg dagg} asserts that, with probability larger than $1-C M^{-L}$, the aggregates $A^{\dagger}_T[20M]$ and $A^{\dagger}_T[\infty]$ coincide on the strip $\mathbb{Z}_{10M}$. Then the same holds for $A^{\dagger}_T[20M]$ and $A^{\dagger}_T[S]$, for any $\mathcal{H}_{20M} \subset S \subset \mathcal{H}$ since $A^{\dagger}_T[20M] \subset A^{\dagger}_T[S] \subset A^{\dagger}_T[\infty]$ by the natural coupling. So $\Stab^M(0)$ occurs with probability larger than  $1-C M^{-L}$, and then (\ref{Stab-likely}) is proven.

Let $\mathbb{S}_r$ be the $(d-2)$-dimensional sphere centered at the origin, with radius $r$ and included in the source set $\mathcal{H}$: $\mathbb{S}_r := \{z \in \mathcal{H} : \, \|z\| = r\}$. We claim that the following key inclusion holds for any $M$,
\begin{equation}
	\label{eqn: pi squared}
	G_{\varepsilon}(0,10M) \cap F(M) \cap \Stab(0,100M) \subset \Big( \bigcup_{c \in \mathbb{S}_{10}} G_{\varepsilon}(Mc,M) \Big) \cap \Big( \bigcup_{c' \in \mathbb{S}_{80}} G_{\varepsilon}(Mc',M) \Big) ~.
\end{equation} 
Lemma~\ref{lemma: radius control} actually appears as a straight consequence of the inclusion (\ref{eqn: pi squared}) combined with (\ref{F(M)likely}) and (\ref{Stab-likely}). Let us first explain why and, in a second step, we will establish (\ref{eqn: pi squared}).

Inequalities (\ref{F(M)likely}) and (\ref{Stab-likely}) allow to write
\begin{eqnarray*}
	\pi_\varepsilon(10M) = \mathbb{P} \big( G_{\varepsilon}(0,10M) \big) & \leq & \mathbb{P} \big( G_{\varepsilon}(0,10M) \cap F(M) \cap \Stab(0,100M) \big) + \frac{C}{M^L} \\
	& \leq & \mathbb{P} \Big( \Big( \bigcup_{c\in\mathbb{S}_{10}}G_{\varepsilon}(Mc,M) \Big) \cap \Big( \bigcup_{c'\in\mathbb{S}_{80}}G_{\varepsilon}(Mc',M) \Big) \Big) + \frac{C}{M^L} ~.
\end{eqnarray*}
Recall that the event $G_{\varepsilon}(Mc,M)$ involves balls of $\hat{\Sigma}_{\varepsilon}^\text{\tiny{loc}}$ whose centers are included in $\mathbb{B}(Mc,10M)$ and radii are defined w.r.t. the aggregate $A^{\dagger}_T[\mathbb{B}(Mc,20M)]$. So the event $\bigcup_{c\in\mathbb{S}_{10}} G_{\varepsilon}(Mc,M)$ only concerns random inputs (i.e. Poisson clocks and random walks) associated to sources of $\mathcal{H}_{30M}$. In the same way, $\bigcup_{c'\in\mathbb{S}_{80}} G_{\varepsilon}(Mc',M)$ only concerns random inputs associated to sources outside of $\mathcal{H}_{59M}$. So they are independent from each other. It is then easy to conclude:
\begin{eqnarray*}
	\pi_{\varepsilon}(10M) & \leq & \mathbb{P} \Big( \bigcup_{c\in\mathbb{S}_{10}} G_{\varepsilon}(Mc,M) \Big) \, \mathbb{P} \Big( \bigcup_{c'\in\mathbb{S}_{80}} G_{\varepsilon}(Mc',M) \Big) + \frac{C}{M^L} \\
	& \leq & |\mathbb{S}_{10}| |\mathbb{S}_{80}| \pi_{\varepsilon}(M)^2 + \frac{C}{M^L} ~.
\end{eqnarray*}

As a consequence, it only remains to establish the inclusion (\ref{eqn: pi squared}) and, to do it, let us assume that $G_{\varepsilon}(0,10M) \cap F(M) \cap \Stab(0,100M)$ occurs. On the event $G_{\varepsilon}(0,10M)$, the localized Boolean model
\[
\hat{\Sigma}_{\varepsilon}^\text{\tiny{loc}}(0,10M) = \bigcup_{z \in \chi_{\varepsilon} \cap \mathbb{B}(0,100M)} \mathbb{B} \Big( z, R^\text{\tiny{loc}}_{T,0,10M}(z;\varepsilon) \Big)
\]
contains a cluster $\mathcal{C}$ joining $\mathbb{B}(0,10M)$ to $\mathbb{B}(0,80M)^c$. The balls of $\hat{\Sigma}_{\varepsilon}^\text{\tiny{loc}}(0,10M)$ are centered at vertices $z$ in $\mathbb{B}(0,100M)$ and their radii $R^\text{\tiny{loc}}_{T,0,10M}(z;\varepsilon)$ are by definition relative to the aggregate $A^{\dagger}_T[\mathbb{B}(0,200M)]$. Since $A^{\dagger}_T[\mathbb{B}(0,200M)] \subset A^{\dagger}_T[\infty]$, we have that, for any $z \in  \chi_{\varepsilon} \cap \mathbb{B}(0,100M)$, $R^\text{\tiny{loc}}_{T,0,10M}(z;\varepsilon) \leq R_{T}(z;\varepsilon) \leq M$ on the event $F(M)$. This means that all the balls of $\mathcal{C}$ have radii smaller than $M$. We can then extract from $\mathcal{C}$ a sub-cluster, say $\mathcal{C}'$, joining $\mathbb{B}(Mc,M)$ for some (random) $c \in \mathbb{S}_{10}$ to $\mathbb{B}(Mc,8M)^c$. 	Indeed, the sphere $\mathbb{S}_{10M}$ is covered by the union of balls $\mathbb{B}(Mc,M)$, with $c\in \mathbb{S}_{10}$. Now, we have to prove that $\mathcal{C}'$ is also a cluster of $\hat{\Sigma}_{\varepsilon}^\text{\tiny{loc}}(Mc,M)$, ensuring the occurrence of $G_{\varepsilon}(Mc,M)$. Hence, we have to prove that each ball $\mathbb{B}(z,R^\text{\tiny{loc}}_{T,0,10M}(z;\varepsilon))$ involved in $\mathcal{C}'$ satisfies the two following properties:
\begin{itemize}
	\item[$\bullet$] Its center $z$ belongs to $\chi_{\varepsilon} \cap \mathbb{B}(Mc,10M)$. This is clear since, by construction, each ball of $\mathcal{C}'$ overlaps $\mathbb{B}(Mc,8M)$ and has a radius smaller than $M$.
	\item[$\bullet$] Its radius $R^\text{\tiny{loc}}_{T,0,10M}(z;\varepsilon)$ is actually equal to $R^\text{\tiny{loc}}_{T,Mc,M}(z;\varepsilon)$. This is where the event $\Stab(0,100M)$ comes into play. The previous item implies that the ball $\mathbb{B}(z,R^\text{\tiny{loc}}_{T,0,10M}(z;\varepsilon))$ is completely included in $\mathbb{B}(Mc,10M)$. Its radius is defined w.r.t. the aggregate $A^{\dagger}_T[\mathbb{B}(0,200M)]$ (see (\ref{defi-Rloc})), but only through
	\[
	A^{\dagger}_T[\mathbb{B}(0,200M)] \cap \Big( \mathbb{Z} \times \mathbb{B}(Mc,10M) \Big)
	\]
	since $\mathbb{B}(z,R^\text{\tiny{loc}}_{T,0,10M}(z;\varepsilon)) \subset \mathbb{B}(Mc,10M)$. Thanks to $\Stab(0,100M)$, in particular $\Stab^M(cM)$ applied with $S := \mathbb{B}(0,200M) \supset \mathbb{B}(Mc,20M)$, we have
	\begin{equation}
		\label{EgalAggRef}
		A^{\dagger}_T[\mathbb{B}(0,200M)] \cap \Big( \mathbb{Z} \times \mathbb{B}(Mc,10M) \Big) = A^{\dagger}_T[\mathbb{B}(Mc,20M)] \cap \Big( \mathbb{Z} \times \mathbb{B}(Mc,10M) \Big) ~.
	\end{equation}
	Since the radius $R^\text{\tiny{loc}}_{T,Mc,M}(z;\varepsilon)$ is defined w.r.t. the aggregate $A^{\dagger}_T[\mathbb{B}(Mc,20M)]$, the identity (\ref{EgalAggRef}) implies that $R^\text{\tiny{loc}}_{T,Mc,M}(z;\varepsilon)$ and $R^\text{\tiny{loc}}_{T,0,10M}(z;\varepsilon)$ are equal.
\end{itemize}
Thus, we have proven that the event $G_{\varepsilon}(Mc,M)$ holds for some $c \in \mathbb{S}_{10}$. We can show in a similar fashion that $G_{\varepsilon}(Mc', M)$ also occurs for some $c' \in \mathbb{S}_{80}$. Inclusion (\ref{eqn: pi squared}) is established.
\end{proof}

\begin{proof}[Proof of Lemma~\ref{lemma: pi bound}]
Let $M \geq 1$. Note that the occurrence of the event $G_{\varepsilon}(0,M)$ forces the random set $\chi_{\varepsilon} \cap \mathbb{B}(0,10M)$ to be non-empty. Therefore,
\begin{eqnarray*}
	\mathbb{P} \left( G_{\varepsilon}(0,M) \right) & \leq & \mathbb{P}\left( \#\left(\mathbb{B}(0,10M) \cap \chi_{\varepsilon}\right) \geq 1 \right) \\
	& \leq & \mathbb{E}\left[\#\left(\mathbb{B}(0,10M) \cap \chi_{\varepsilon}\right)\right] \\
	& = & \sum_{z \in \mathbb{B}(0,10M)} \mathbb{E} \left[ \mathbf{1}_{z\in \chi_{\varepsilon}} \right] \\
	& = & \# \mathbb{B}(0,10M) p_{\varepsilon} ~.
\end{eqnarray*}
Using $p_{\varepsilon} = 1 - e^{-\varepsilon} \leq \varepsilon$ and $\# \mathbb{B}(0,10M) \leq C_d M^{d-1}$, we get the desired result.
\end{proof}

\section{Proof of Theorem~\ref{thm: forest properties}}
\label{sec: proof forest properties}

Before giving the proof of Theorem~\ref{thm: forest properties}, we give the following lemma:
\begin{lemma}
	\label{lemma: subset inclusion}
	For any finite subset $S \subset \Z^d$, 
	\[
	\mathbb{P}\left( S \subset A^{\dagger}_n[\infty] \right) \underset{n \to \infty}{\longrightarrow} 1.
	\]
\end{lemma}
\begin{proof}[Proof of Lemma~\ref{lemma: subset inclusion}]
	To do so, we will be using the Shape Theorem for standard \textsc{idla}, in the case where exactly $n$ particles are sent from the origin. Let us denote this aggregate by $A(n)$. We know from Theorem 1 of \cite{lawler1992internal} that 
	\begin{equation}
		\label{eqn: shapethm1 lawler bis}
		\mathbb{P}\left( S \subset A(n) \right) \underset{n \to \infty}{\longrightarrow} 1.
	\end{equation}
	Now, recall that particles of $A^{\dagger}_n[\infty]$ are given according to a family of \textsc{ppp}'s in $\R_+$, denoted by $\left(\mathcal{N}_z\right)_{z \in \mathcal{H}}$. Let $N_0 = \mathcal{N}_0([0,n])$ denote the number of particles sent from the origin. Since $N_0$ is a Poisson random variable of parameter $n$, we know from concentration inequality theory that
	\begin{equation*}
		\label{eqn: poisson tail ineq1 bis}
		\mathbb{P}\left( N_0 \leq n/2 \right) = \mathbb{P}\left( N_0 - \mathbb{E}[N_0] \leq -n/2 \right) \leq \exp\left(-\frac{n}{8}\right).
	\end{equation*}
	Hence, define $\mathscr{E}_n:= \left\{ N_0 \leq n/2 \right\}$. Let $S$ denote a finite subset of $\Z^d$. We write:
	\begin{align*}
		\mathbb{P}\left( S \subset A(n/2) \right) &\leq \mathbb{P}\left( \left\{ S \subset A(n/2) \right\}\cap {\mathscr{E}_n}^c  \right) + \mathbb{P}\left({\mathscr{E}_n}\right) \\
		& \leq \mathbb{P}\left(S \subset A^{\dagger}_n[0]\right) + \exp\left(-\frac{n}{8}\right) \\
		& \leq \mathbb{P}\left(S \subset A^{\dagger}_n[\infty]\right) + \exp\left(-\frac{n}{8}\right).
	\end{align*}
	Hence, for any finite subset $S \subset \Z^d$, we have that 
	\[
	\mathbb{P}\left(S \subset A^{\dagger}_n[\infty]\right) \geq \mathbb{P}\left( S \subset A(n/2) \right) - \exp\left(-\frac{n}{8}\right).
	\]
	Using \eqref{eqn: shapethm1 lawler bis}, we have the desired result.
\end{proof}
\color{black}

\begin{proof}[Proof of Theorem~\ref{thm: forest properties}]
We begin by showing $1.$
For any $n\geq 0,\ \fdag{n} \subset \fdag{\infty}$, so $V\left(\fdag{n}\right) \subset V\left(\fdag{\infty}\right)$. Now, consider $S$ a finite subset of $\mathbb{Z}^d$. Since the vertex set of $\fdag{n}$ is $A^{\dagger}_n[\infty]$, we have that for any $n \geq 0$,
\[
\mathbb{P}\left( S \subset V\left(\fdag{\infty}\right)\right) \geq \mathbb{P}\left(S \subset V\left(\fdag{n}\right)\right) = \mathbb{P}\left(S \subset A^{\dagger}_n[\infty]\right).
\]
This result is immediate using Lemma \ref{lemma: subset inclusion}.
\newline
We move to the proof of $2.$ Consider a compact $K$ of $\R^d$. Fix $k \in \mathcal{H}$ and $n\geq 0$. According to \cite{schneider} (Theorem 2.1.3), it is sufficient to show that $\fdag{n}$ and $T_k\fdag{n}$ (respectively $\fdag{\infty}$ and $T_k\fdag{\infty}$) have the same probability of intersecting $K$. Assume that this holds for $\fdag{n}$, that is:
\begin{equation}
	\label{eqn: mixing fn}
	\mathbb{P}\left(\fdag{n} \cap K \neq \emptyset \right) = \mathbb{P}\left(T_k\fdag{n} \cap K \neq \emptyset \right) .
\end{equation}
Note that for any subset $C$ such that $\left(C + \mathbb{B}(0,1)\right)\cap \mathbb{Z}^d \subset A^{\dagger}_n[\infty]$, then $\fdag{n}\cap C = \fdag{\infty} \cap C$.
Again using Lemma \ref{lemma: subset inclusion}, we have that 
\[
\mathbb{P}\left( \left(K + \mathbb{B}(0,1)\right)\cap \mathbb{Z}^d \subset A^{\dagger}_n[\infty] \right) \underset{n \to \infty}{\longrightarrow} 1.
\]
Thus, there exists an integer $n_0$ such that $\mathbb{P}\left( \left(K + \mathbb{B}(0,1)\right)\cap \mathbb{Z}^d \subset A^{\dagger}_{n_0}[\infty] \right) \geq 1 - \varepsilon$.
We have: 
\begin{align}
	\label{eqn: eps}
	| \mathbb{P}\left( \fdag{\infty} \cap K \neq \emptyset \right) - \mathbb{P}\left( \fdag{n_0} \cap K \neq \emptyset \right) | &\leq \nonumber \mathbb{P}\left(\fdag{\infty} \cap K \neq \fdag{n_0} \cap K \right) \\
	& \leq \mathbb{P}\left( \left(K + \mathbb{B}(0,1)\right)\cap \mathbb{Z}^d \nsubset A^{\dagger}_{n_0}[\infty] \right) \leq \varepsilon.
\end{align}
Similarly, we can show that
\begin{equation}
	\label{eqn: eps bis}
	|\mathbb{P}\left( T_k\fdag{\infty} \cap K \neq \emptyset \right) - \mathbb{P}\left( T_k\fdag{n_0} \cap K \neq \emptyset \right) | \leq \varepsilon.
\end{equation}
We can now conclude for $\fdag{\infty}$, since 
\begin{multline*}
	| \mathbb{P}\left( \fdag{\infty} \cap K \neq \emptyset \right) - \mathbb{P}\left( T_k\fdag{\infty} \cap K \neq \emptyset \right) | \leq  \\
	\begin{split}
		& |\mathbb{P}\left( \fdag{\infty} \cap K \neq \emptyset \right) - \mathbb{P}\left( \fdag{n_0} \cap K \neq \emptyset \right)| + |\mathbb{P}\left( \fdag{n_0} \cap K \neq \emptyset \right) - \mathbb{P}\left( T_k\fdag{n_0} \cap K \neq \emptyset \right)|\\
		+ & |\mathbb{P}\left( T_k\fdag{n_0} \cap K \neq \emptyset \right) - \mathbb{P}\left( T_k\fdag{\infty} \cap K \neq \emptyset \right)|.
	\end{split}
\end{multline*}
Now, from \eqref{eqn: mixing fn}, we know that the middle term is equal to 0, and from \eqref{eqn: eps} and \eqref{eqn: eps bis}, we get that the first and third term are bounded by $\varepsilon$. Thus, 
\[
| \mathbb{P}\left( \fdag{\infty} \cap K \neq \emptyset \right) - \mathbb{P}\left( T_k\fdag{\infty} \cap K \neq \emptyset \right) | \leq 2\varepsilon.
\]
It now remains to show \eqref{eqn: mixing fn}. Fix $n \geq 0$. Take $M \geq 1$ sufficiently large such that $K \cap \mathbb{Z}^d \subset \mathbb{Z}_M$. From Theorem~\ref{thm: forest stabilization}, there exists a random integer $N_0$ such that for any $N' \geq N_0$, with probability greater than $1- \varepsilon$, we have
\begin{equation}
	\label{eqn: stab of forests}
	\fdag{n}[N'] \cap \mathbb{Z}_M = \fdag{n} \cap \mathbb{Z}_M,\ \quad T_k\fdag{n}[N'] \cap \mathbb{Z}_M = T_k\fdag{n} \cap \mathbb{Z}_M.
\end{equation}
Then, we have:
\begin{multline*}
	| \mathbb{P}\left( \fdag{n}\cap K \neq \emptyset \right) - \mathbb{P}\left( T_k \fdag{n} \cap K \neq \emptyset \right) | \leq \\
	\begin{split}
		& | \mathbb{P}\left( \fdag{n}\cap K \neq \emptyset \right) - \mathbb{P}\left( \fdag{n}[N_0]\cap K \neq \emptyset \right) | + | \mathbb{P}\left(\fdag{n}[N_0] \cap K \neq \emptyset \right) - \mathbb{P}\left(T_k\fdag{n}[N_0] \cap K \neq \emptyset \right) | + \\
		& | \mathbb{P}\left( T_k\fdag{n}[N_0]\cap K \neq \emptyset \right) - \mathbb{P}\left( T_k\fdag{n}\cap K \neq \emptyset \right) |  \\
		\leq & | \mathbb{P}\left(\fdag{n}[N_0] \cap K \neq \emptyset \right) - \mathbb{P}\left(T_k\fdag{n}[N_0] \cap K \neq \emptyset \right) | + 2\varepsilon.
	\end{split}
\end{multline*}

Now, we grow the forests $\fdag{n}[N_0]$ and $T_k\fdag{n}[N_0]$ to obtain two forests $\mathfrak{F}_1$ and $\mathfrak{F}_2$. We obtain $\mathfrak{F}_1$ by sending the particles used to build $\fdag{n}[N_0]$, and add the additional particles from $\left(T_k\mathcal{H}_{N_0}\right) \cap {\mathcal{H}_{N_0}}^c$.
To build $\mathfrak{F}_2$, consider the forest induced by particles used to build $\fdag{n}[N_0]$ and additional particles from $\left(T_{-k}\mathcal{H}_{N_0}\right) \cap {\mathcal{H}_{N_0}}^c$. Let $\mathfrak{F}$ denote this forest. We define $\mathfrak{F}_2$ as the translation of vector $k$ of this forest, that is $\mathfrak{F}_2 := T_k \mathfrak{F}$. Now, from \eqref{eqn: stab of forests}, we know that $\mathfrak{F}_1$ and $\fdag{n}[N_0]$ coincide on the strip $\mathbb{Z}_M$ (and hence on $K$) with probability greater than $1-\varepsilon$. The same is true for $\mathfrak{F}_2$ and $T_k\fdag{n}[N_0]$. Therefore, we have 
\begin{align*}
	| \mathbb{P}\left(\fdag{n}\cap K \neq \emptyset \right) - \mathbb{P}\left(T_k\fdag{n}\cap K \neq \emptyset \right) | \leq | \mathbb{P}\left(\mathfrak{F}_1 \cap K \neq \emptyset \right) - \mathbb{P}\left(\mathfrak{F}_2 \cap K \neq \emptyset \right) | + 4\varepsilon.
\end{align*}
\begin{lemma}
	The set of sources used to build $\mathfrak{F}_1$ and $\mathfrak{F}_2$ are identical.
\end{lemma}

Now, the forest $\mathfrak{F}_1$ and $\mathfrak{F}_2$ are built according to the \textsc{idla} protocol, using the same set of sources, with i.i.d Poisson clocks over the time interval $[0,n]$. They therefore have same distribution, which implies that 
\[
| \mathbb{P}\left(\fdag{n}\cap K \neq \emptyset \right) - \mathbb{P}\left(T_k\fdag{n}\cap K \neq \emptyset \right) | \leq 4\varepsilon,
\]
which concludes the proof.
\newline
We now prove 3. 
We show that for any compacts $C_1,\ C_2$ of $\R^2$
\begin{equation}
	\label{eqn: mixing compact}
	\lim_{k\in\mathcal{H},\ \|k\|\to \infty}\mathbb{P}\left(\fdag{\infty}\cap(C_1\cup T_k C_2) = \emptyset\right) = \mathbb{P}\left( \fdag{\infty}\cap C_1 = \emptyset \right)\mathbb{P}\left(\fdag{\infty}\cap C_2 = \emptyset\right).
\end{equation}
Fix $\varepsilon >0$ and let $C_1,\ C_2$ be two compact sets of $\R^2$.
Let $r > 0$ be such that $C_1 \cup C_2$ is included in $\mathbb{B}(0,r-1)$. From Lemma \ref{lemma: subset inclusion}, we can pick $n$ large enough such that 
\[
\mathbb{P}\left( \mathbb{B}(0,r)\cap \mathbb{Z}^2 \subset A^{\dagger}_n[\infty] \right) \geq 1 - \varepsilon.
\]
On the event $\left\{ \mathbb{B}(0,r)\cap \mathbb{Z}^2 \subset A^{\dagger}_n[\infty] \right\}$, $\fdag{\infty}\cap C_i$ and $\fdag{n}\cap C_i$ are equal for any $i\in\{1,2\}$. Since the distribution of $A^{\dagger}_n[\infty]$ is invariant with respect to $T_k$, we have that for any $k \in \mathcal{H}$ (independent of $\varepsilon$),
\[
\mathbb{P}\left( T_k(C_1 \cup C_2) \subset A^{\dagger}_n[\infty] \right) \geq 1 -\varepsilon.
\]
This implies:
\[
\mathbb{P}\left( \fdag{\infty} \cap (C_1\cup T_k C_2) = \fdag{n} \cap (C_1\cup T_k C_2) \right) \geq 1 - 2\varepsilon.
\]
Therefore, 
\begin{multline*}
	| \mathbb{P}\left(\fdag{\infty} \cap (C_1\cup T_k C_2) =\emptyset \right) - \mathbb{P}\left( \fdag{\infty} \cap C_1 =\emptyset\right)\mathbb{P}\left( \fdag{\infty} \cap C_2 =\emptyset\right)| \\
	\leq | \mathbb{P}\left(\fdag{n} \cap (C_1\cup T_k C_2) =\emptyset\right) - \mathbb{P}\left( \fdag{n} \cap C_1 =\emptyset \right)\mathbb{P}\left( \fdag{n} \cap C_2 =\emptyset \right)|+4\varepsilon.
\end{multline*}
It remains to show that $\fdag{n}$ is mixing with respect to $T_k$.
From Theorem~\ref{thm: forest stabilization}, there exists $N_0$ such that 
\begin{equation}
	\label{eqn: forest stab eps}
	\mathbb{P}\left( \fdag{n}\cap C_1 \neq \fdag{n}[N_0] \cap C_1 \right) \leq \varepsilon \, \text{ and } \, \mathbb{P}\left( \fdag{n}\cap C_2 \neq \fdag{n}[N_0] \cap C_2 \right) \leq \varepsilon.
\end{equation}
We have, by shifting all the random clocks by a vector $-k \in \mathcal{H}$:
\begin{align*}
	\mathbb{P}\left( \fdag{n}\cap C_2 \neq \fdag{n}[N_0] \cap C_2 \right) & = \mathbb{P}\left( \fdag{n}(\omega)\cap C_2 \neq \fdag{n}[N_0](\omega) \cap C_2 \right) \\
	& = \mathbb{P}\left( \fdag{n}(\omega - k)\cap C_2 \neq \fdag{n}[N_0](\omega - k) \cap C_2 \right) \\
	& = \mathbb{P}\left(\fdag{n}(\omega)\cap T_kC_2 \neq \fdag{n}[\mathbb{B}(k, N_0)](\omega) \cap T_kC_2\right).
\end{align*}
Hence, from \eqref{eqn: forest stab eps}, for any $k\in \mathcal{H}$ (independent of $\varepsilon$), there exists $N_0$ such that
\begin{equation}
	\label{eqn: forest stab eps bis}
	\mathbb{P}\left(\fdag{n}(\omega)\cap T_kC_2 \neq \fdag{n}[\mathbb{B}(k, N_0)](\omega) \cap T_kC_2\right) \leq \varepsilon.
\end{equation}
Now, we write:
\begin{multline*}
	|\mathbb{P}\left(\fdag{n} \cap (C_1\cup T_k C_2) = \emptyset\right)-\mathbb{P}\left( \fdag{n} \cap C_1 = \emptyset \right)\mathbb{P}\left( \fdag{n} \cap C_2 = \emptyset\right)| \\
	\begin{split}
		& = |\mathbb{P}\left( \{\fdag{n} \cap C_1 = \emptyset\} \cap \{\fdag{n} \cap T_k C_2 = \emptyset\}\right)-\mathbb{P}\left( \fdag{n} \cap C_1 = \emptyset \right)\mathbb{P}\left( \fdag{n} \cap C_2 = \emptyset\right)| \\
		& \leq |\mathbb{P}\left( \{\fdag{n} \cap C_1 = \emptyset\} \cap \{\fdag{n} \cap T_k C_2 = \emptyset\}\right)- \mathbb{P}\left(\{\fdag{n}[N_0] \cap C_1 = \emptyset\} \cap \{\fdag{n}[\mathbb{B}(k,N_0)] \cap T_k C_2 = \emptyset\}\right) | \\
		& + | \mathbb{P}\left(\{\fdag{n}[N_0] \cap C_1 = \emptyset\} \cap \{\fdag{n}[\mathbb{B}(k,N_0)] \cap T_k C_2 = \emptyset\}\right) - \mathbb{P}\left( \fdag{n} \cap C_1 = \emptyset \right)\mathbb{P}\left( \fdag{n} \cap C_2 = \emptyset\right)|.\\
	\end{split}
\end{multline*}
\begin{lemma}
	\label{lemma: abso inequality}
	Let $A, A', B,$ and $B'$ denote 4 random sets. Suppose $\mathbb{P}\left(A \neq A'\right) \leq \varepsilon$ and $\mathbb{P}\left(B \neq B'\right) \leq \varepsilon$. Then 
	\[
	|\mathbb{P}\left( \left\{A = \emptyset\right\}\cap \left\{B = \emptyset\right\} \right) - \mathbb{P}\left(\left\{A' = \emptyset\right\} \cap\left\{B' = \emptyset\right\} \right) | \leq 2\varepsilon.
	\]
\end{lemma}
To alleviate notation, we define the following sets:
\[
\left\{
\begin{array}{ll}
	A = \fdag{n}\cap C_1,\ \\ 
	B = \fdag{n}\cap T_k C_2,\\ 
\end{array}
\right.
\quad \text{ and } \quad 
\left\{
\begin{array}{ll}
	A' = \fdag{n}[N_0]\cap C_1, \\ 
	B' = \fdag{n}[\mathbb{B}(k,N_0)]\cap T_k C_2. \\ 
\end{array}
\right.
\]
We can rewrite the previous inequality as:
\begin{multline*}
	|\mathbb{P}\left(\fdag{n} \cap (C_1\cup T_k C_2) = \emptyset\right)-\mathbb{P}\left( \fdag{n} \cap C_1 = \emptyset \right)\mathbb{P}\left( \fdag{n} \cap C_2 = \emptyset\right)| \\
	\begin{split}
		&\leq |\mathbb{P}\left(\left\{A = \emptyset\right\}\cap \left\{B = \emptyset\right\}\right)- \mathbb{P}\left(\left\{A' = \emptyset\right\} \cap\left\{B' = \emptyset\right\} \right) | \\
		&+ | \mathbb{P}\left(\left\{A' = \emptyset\right\} \cap\left\{B' = \emptyset\right\} \right) - \mathbb{P}\left(A = \emptyset\right)\mathbb{P}\left(B = \emptyset\right)|.
	\end{split}
\end{multline*}
Note that for the last term, we replaced $\mathbb{P}\left(\fdag{n} \cap C_2 = \emptyset\right)$ by $\mathbb{P}\left(B = \emptyset\right)$. We can do so since the distribution of $\fdag{n}$ is invariant with respect to $T_k$, so $\fdag{n}$ has same probability of intersecting $C_2$ or $T_kC_2$.
\newline
Notice that for any $k$ such that $\|k\| > 2N_0$, the events $A'$ and $B'$ are independent, since the two forests considered are built using disjoint sets of sources.
Additionally, from \eqref{eqn: eps} and \eqref{eqn: eps bis}, we have that $\mathbb{P}\left(A \neq A' \right) \leq \varepsilon$ and $\mathbb{P}\left(B \neq B' \right) \leq \varepsilon$. Therefore, using the result of Lemma~\ref{lemma: abso inequality}, we get, for any $\|k\| > 2N_0$:
\begin{multline*}
	|\mathbb{P}\left(\fdag{n} \cap (C_1\cup T_k C_2) = \emptyset\right)-\mathbb{P}\left( \fdag{n} \cap C_1 = \emptyset \right)\mathbb{P}\left( \fdag{n} \cap C_2 = \emptyset\right)| \\
	\leq | \mathbb{P}\left(A' = \emptyset\right)\mathbb{P}\left(B' = \emptyset\right) - \mathbb{P}\left(A = \emptyset\right)\mathbb{P}\left(B = \emptyset\right)| + 2\varepsilon.
\end{multline*}
Now, since $\mathbb{P}\left(A \neq A' \right) \leq \varepsilon$ and $\mathbb{P}\left(B \neq B' \right) \leq \varepsilon$, one can show (in the same spirit of Lemma~\ref{lemma: abso inequality}) that 
\[
| \mathbb{P}\left(A' = \emptyset\right)\mathbb{P}\left(B' = \emptyset\right) - \mathbb{P}\left(A = \emptyset\right)\mathbb{P}\left(B = \emptyset\right)| \leq 2\varepsilon.
\]
Therefore, 
\[
| \mathbb{P}\left(\fdag{\infty} \cap (C_1\cup T_k C_2) =\emptyset \right) - \mathbb{P}\left( \fdag{\infty} \cap C_1 =\emptyset\right)\mathbb{P}\left( \fdag{\infty} \cap C_2 =\emptyset\right)| \leq 8 \varepsilon,
\]
which concludes the proof.
\end{proof}

\appendix
\section{Appendix: Proof of Proposition~\ref{thm: global upper bound}}
\label{sec: appendix}
Before proving Proposition~\ref{thm: global upper bound}, we must first introduce a new family of random aggregates. 
Let $n\geq 1$. Just like for $A^{\dagger}_n[\infty]$, we begin by building a family of \textit{finite} random aggregates $(A^{*}_n[M])_{M \geq 0}$.
Like $A^{\dagger}_n[\infty]$, the number of particles sent from each source $z$ is random, given this time by a Poisson random variable $N_z$ of parameter $n$, but unlike $A^{\dagger}_n[\infty]$, these are sent in a predetermined order. 
Let $(N_z)_{z\in \mathcal{H}}$ denote a family of i.i.d Poisson random variables of parameter $n$. 
When $M=0$, $A^*_n[0]$ is the aggregate obtained after launching $N_0$ particles from the origin. Given a realization of $A_n[M-1]$, we throw $N_z$ particles from each source $z$ of level $M$ according to the lexicographical order. So $A^*_n[M]$ is defined as the aggregate produced by $A^*_n[M-1]$ and the new sites added by particles launched at level $M$.
By construction, $(A^{*}_n[M])_{M \geq 0}$ is increasing with respect to inclusion, so we can define $A^*_n[\infty]$ as 
\[
A^*_n[\infty] := \uparrow \bigcup_{M\geq 0} A^*_n[M] \quad \mathrm{a.s.}
\]
A consequence of the Abelian Property (see \cite{diaconis1991growth}, p. 97) is that for all $M\geq 0$, 
\begin{equation} 
\label{eqn: agg equal in law}
A^{\dagger}_n[M] \overset{\mathrm{law}}{=} A^{*}_n[M].
\end{equation}
Since both families of aggregates $(A_n^*[M])_{M \geq 0}$ and $(A^{\dagger}_n[M])_{M \geq 0}$ are increasing, we deduce from \eqref{eqn: agg equal in law} that $A^{\dagger}_n[\infty] \overset{\mathrm{law}}{=} A^{*}_n[\infty]$. 

The global upper bound provided by Proposition~\ref{thm: global upper bound} is a refined version of Theorem 4.1 of \cite{chenavier2024idla}. Although our result is finer and deals with a random number of emitted particles, the strategy of the proof is essentially the same.

We will be proving Proposition~\ref{thm: global upper bound} by induction over $n$. Since $A^{\dagger}_n[\infty] \overset{\mathrm{law}}{=} A^*_n[\infty]$, we will show the result for $A^*_n[\infty]$ instead, since this aggregate is built by sending particles in the \textit{usual order}. We define the event $\Over^*_{\alpha}$ in the same way as $\Over^{\dagger}_{\alpha}$ in \eqref{eqn: over*} but with respect to the aggregate $A_n^*[\infty]$.
We show that if for some fixed $n$, the aggregate $A^*_n[\infty]$ is contained within $\mathscr{C}_{\varepsilon}^{\alpha}$ for some $\varepsilon > 0$, and if we launch $N'_z$ additional particles from each source $z$ of $\mathcal{H}$, where $(N'_z)_{z\in\mathcal{H}}$ is an independent family of Poisson variables of parameter $1$, then the resulting aggregate is contained within a slightly larger cone $\mathscr{C}_{\varepsilon'}^{\alpha}$ ($\varepsilon' > \varepsilon$) with high probability. We keep the same notations as in the proof of Theorem 4.1 of \cite{chenavier2024idla} and define the sequences $(M_n)_{n\geq 0}$ and $(\varepsilon_n)_{n\geq 0}$ in the same way. 
Just like for Theorem 4.1 of \cite{chenavier2024idla}, it is sufficient to prove the following proposition:
\begin{proposition}
\label{prop: fine BMN short}
For all $L \geq 1$, for all $\alpha \in (1-1/d, 1)$, for all $\varepsilon>0$, for all $n\geq 1$, there exists a constant $C = C({\varepsilon,n,\alpha,d,L})>0$ such that for all $M\geq 1$, 
\[
\mathbb{P}\left(\Over^*_{\alpha}(M_n, \varepsilon_n)\right)\leq \dfrac{C}{M^L}.
\]
\end{proposition}
\begin{proof}[Proof of Proposition~\ref{prop: fine BMN short}:]
We show our result by induction over $n$. 
Take $L>1$. Our induction statement is the following:
\newline
$
\forall n\geq 0,\ \mathcal{P}(n): \forall \alpha \in (1-1/d, 1),\ \forall \varepsilon\in (0,1),\ \exists C = C({\varepsilon,n,\alpha,d,L})>0,\ \forall M\geq 1,\ 
$
\[
\mathbb{P}\left(\Over^*_{\alpha}(M_n,\varepsilon_n)\right)\leq \dfrac{C}{M^L}.
\]
When $n = 0$, this is clear since $A^*_n[\infty] \overset{\mathrm{a.s}}{=} \emptyset$, hence $A^*_n[\infty] \cap \mathbb{Z}_M^c \overset{\mathrm{a.s}}{\subset}\mathscr{C}_{\varepsilon}^{\alpha}$.
\newline
Let $n\geq 1$ and suppose $\mathcal{P}(n)$ holds. Fix $\alpha \in (1-1/d, 1)$. We write:
\[
\mathbb{P}\left(\Over^*_{\alpha}(M_{n+1},\varepsilon_{n+1})\right)\leq \mathbb{P}\left(\Over^*_{\alpha}(M_{n+1},\varepsilon_{n+1})\cap \Over^*_{\alpha}(M_{n},\varepsilon_{n})^c\right)+\mathbb{P}\left(\Over^*_{\alpha}(M_{n},\varepsilon_{n} ) \right).
\]
The last term is handled by our induction hypothesis.
We switch our focus to the first term. On the event $\Over^*_{\alpha}(M_{n+1},\varepsilon_{n+1})\cap \Over^*_{\alpha}(M_{n},\varepsilon_{n})^c$, we have $A^*_n[\infty] \cap \mathbb{Z}_{M_n}^c \subset \mathscr{C}_{\varepsilon_{n}}^{\alpha} $, but when launching $N'_z$ new particles from each source $z\in\mathcal{H}$, the new aggregate obtained spills over $\mathscr{C}_{\varepsilon_{n+1}}^{\alpha}$ on $\mathbb{Z}_{M_{n+1}}^c$.
This implies the existence of three random sites $(Z,Z^*,Z_{n+1})\in \mathbb{Z}^d$, and aggregates $A_{Z^*}$ and $A_{Z^*}^{-}$, defined just like in \cite{chenavier2024idla}, but defined with respect to $\mathscr{C}_{\varepsilon}^{\alpha}$ rather than $\mathscr{C}_{\varepsilon}$. Note now that we have : $Z_{n+1}=Z\pm \left(\varepsilon_{n+1}\|Z\|^{\alpha}\right)\cdot e_1$, where $e_1 = (1, 0,\dots, 0)$.

We must control that no unreasonable amount of particles is emitted from $\mathcal{H}$. To do so, we introduce the following event:
\[
\mathscr{E}_M(\gamma) := \bigcap_{l\geq 0}\left\{ \forall z\in\mathcal{H},\ \|z\|= l,\ N'_z \leq 1 + \max\{l,M\}^{\gamma}\right\}, 
\]
We can show using \eqref{eqn: concentration inequality poisson} that $\mathbb{P}\left(\mathscr{E}_M^c(\gamma)\right)$ decreases faster than any power of $M^{-1}$.
Fix $\gamma \in (0, 1)$ such that $\gamma < (\alpha -1)d +1 $ (such a value exists since $\alpha \in (1-1/d, 1)$). We explain this choice later. We write
\begin{multline*}
	\mathbb{P}\left(\Over^*_{\alpha}(M_{n+1},\varepsilon_{n+1})\cap \Over^*_{\alpha}(M_{n},\varepsilon_{n})^c\right) \\ 
	\begin{split}
		&\leq\mathbb{P}\left(\Over^*_{\alpha}(M_{n+1},\varepsilon_{n+1})\cap \Over^*_{\alpha}(M_{n},\varepsilon_{n})^c \cap \mathscr{E}_M(\gamma)\right) + \mathbb{P}\left(\mathscr{E}_M^c(\gamma)\right) \\
		&\leq \sum_{l\geq M_{n+1}}\sum_{\|z\|=l}\mathbb{P}\left(Z=z,\ \Over^*_{\alpha}(M_{n+1},\varepsilon_{n+1})\cap\Over^*_{\alpha}(M_{n},\varepsilon_{n})^c \cap \mathscr{E}_M(\gamma) \right) + \smallO(M^{-L}).
	\end{split}
\end{multline*} 
Fix $l \geq M_{n+1}$ and $z\in \mathcal{H}$ such that $\|z\| = l$, and let $z_{n+1} = z \pm \left( \varepsilon_{n+1} \|z\|^{\alpha}\right)\cdot e_1$. We consider the case where a ball of particles has settled around $z_{n+1}$, and the case where a thin tentacle reaches out to $z_{n+1}$.
We use an adaptation of Lemma 2 of \cite{jerison2012logarithmic} to deal with the event of tentacles.
\begin{lemma}
	\label{lemma: fine tentacle short} 
	There exist positive universal constants $b,\ K_0,\ c$ such that for all real numbers $r>0$ and all $z\in\mathcal{H}$ with $0\notin \mathbb{B}(z_{n+1},r)$,
	\begin{multline*}
		\mathbb{P}\left(Z=z,\ \Over^*_{\alpha}(M_{n+1},\varepsilon_{n+1})\cap\Over^*_{\alpha}(M_{n},\varepsilon_{n})^c,\ \#\left(A_{Z^*}\cap \mathbb{B}(z_{n+1},r)\right)\leq br^d\right)
		\leq K_0e^{-cr^2}.
	\end{multline*}
\end{lemma}
We apply this Lemma with $r = r_{n+1} = \frac{\varepsilon l^{\alpha} }{2^{n+1}}$, in the same spirit as in \cite{chenavier2024idla}. This gives:
\begin{align}
	&\mathbb{P}\left(Z=z,\ \Over^*_{\alpha}(M_{n+1},\varepsilon_{n+1}) \cap \Over^*_{\alpha}(M_{n},\varepsilon_{n})^c\cap \mathscr{E}_M(\gamma) \right)  \nonumber \\
	\leq &\mathbb{P}\left(Z=z,\ \Over^*_{\alpha}(M_{n+1},\varepsilon_{n+1}) \cap \Over^*_{\alpha}(M_{n},\varepsilon_{n})^c, \right. \nonumber \\
	&\hspace{10em} \left. \mathscr{E}_M(\gamma), \ \#\left(A_{Z^*} \cap \mathbb{B}(z_{n+1},r_{n+1})\right) > br_{n+1}^d \right) + K_0e^{-c_1l^{2\alpha}}, \label{eqn: fine ball short}
\end{align}
where $c_1=c_1(n, \varepsilon)=\frac{c\varepsilon^2}{4^{n+1}}$. The term \eqref{eqn: fine ball short} requires a little more work than in the deterministic global upper bound.
Note that $r_{n+1}^d$ is of order $l^{\alpha d}$. We are working on the event where (roughly) more than $l^{\alpha d}$ particles have settled around $z_{n+1}$, knowing that each source $\tilde{z}\in\mathcal{H}$ has emitted at most $1 + \|\tilde{z}\|^{\gamma}$ particles. We show that this implies that $\|Z-Z^*\| \geq Kl^{\eta}$, where $\eta >1$ and $K$ denotes some positive constant. 
Let us now explain why $\eta > 1$. Suppose, contrary to our claim, that $\eta  \leq 1$. The number of sources inside $\mathbb{B}\left(z, Kl^{\eta} \right)\cap \mathcal{H}$ is of the order $l^{\eta(d-1)}$ since $\mathcal{H}$ is a hyperplane of dimension $d-1$. Working on the event $\mathscr{E}_M$, the largest amount of particles emitted by a single source within this $(d-1)$-dimensional ball is $1+\left(l+Kl^{\eta}\right)^{\gamma}$, which is of the order $l^{\gamma}$ when $\eta \leq 1$. In the worst case scenario, if all of the sources within $\mathbb{B}\left(z, Kl^{\eta} \right)\cap \mathcal{H}$ emit of the order of $l^{\gamma}$ particles, the total number of particles emitted will be of the order of 
$
l^{\eta(d-1)}\times l^{\gamma} = l^{\eta(d-1) + \gamma}.
$
Now, for this to be of the order (or greater) than $l^{\alpha d}$, it is thus necessary that 
\[
\eta (d-1) + \gamma \geq \alpha d \Longleftrightarrow \eta \geq  \frac{\alpha d - \gamma}{d-1}.
\]
However, due to our choice of $\gamma$, this necessarily implies that $\eta >1$, which contradicts our assumption that $\eta \leq 1$.
Therefore, in order for \textit{more} than $l^{\alpha d}$ particles to settle inside $\mathbb{B}\left(z_{n+1}, r_{n+1}\right)$, it is necessary that one of these particles has been emitted from a source outside of $\mathbb{B}\left(z, Kl^{\eta} \right)\cap \mathcal{H}$, with $\eta >1$. It is thus necessary that $ \|Z-Z^*\| \geq Kl^{\eta}$. 
Since $Z^*$ is defined as the source from which the \textit{first} overflowing particle is emitted, this implies that the aggregate before sending from $Z^*$ is strictly contained within $\mathscr{C}_{n+1}^{\alpha}$, allowing us to use a donut argument (similar to the one used in the proof of Lemma~\ref{lemma: radius control}) to control the trajectory of the overflowing particle. 
This implies that one of the particles sent from $z'$, with $\|z'-z\| \geq Kl^{\eta}$, has crossed multiple donuts. We detail below how we compute a lower bound on the total number of donuts a particle needs to cross from $z'$ to $z$. Fix $h\geq Kl^{\eta}$ and $z'$ such that $\|z'-z\| = h$. We build donuts from level $\|z'\|$ to $\|z\| = l$. The first donut, which is the \textit{largest} donut, has dimensions at most $2\varepsilon_{n+1}(l+h)^{\alpha} \leq 2 \varepsilon_{n+1} (2h)^{\alpha} \leq 4\varepsilon_{n+1}h^{\alpha}$. All other donuts will have smaller dimensions, which implies that the number of donuts $k = k(h,l, \varepsilon_{n+1}, \alpha)$ between $z'$ and $z$ is such that
\[
k \geq \frac{h}{4\varepsilon_{n+1} h^{\alpha}} = \frac{h^{1-\alpha}}{4\varepsilon_{n+1}}.
\]
Now, the number of particles sent from $z'$ at distance $h$ of $z$ is (up to a multiplicative factor) at most $h^{d-2 + \gamma}$. Therefore, the donut argument gives:
\begin{align*}
	&\sum_{h\geq Kl^{\eta}} \sum_{{\|z'-z\|}=h} \mathbb{P}\left( \left\{ \begin{array}{l} 
		\text{a particle sent from } z' \text{ reaches level } l \\
		\text{ while staying within } \mathscr{C}_{\varepsilon_{n+1}}^{\alpha} \end{array} \right\} \cap \mathscr{E}_M(\gamma) \right) \\
	\leq & K_d\sum_{h \geq Kl^{\eta}} h^{d - 2 + \gamma }(1-c)^{k} = K_d \sum_{h \geq K_dl^{\eta}}h^{d - 2 + \gamma } \exp\left(-c_0\left(\varepsilon_{n+1} \right)h^{1-\alpha}\right),
\end{align*}
where $c_0\left(\varepsilon_{n+1} \right) = -\frac{\log(1-c)}{4\varepsilon_{n+1}}$. Throughout the rest of the proof, $K_d$ will denote a generic constant depending only on $d$, whose value may vary from line to line. 
Now, using the fact that $1-\alpha> 0$, standard computations yield:
\[
K_d \sum_{h \geq Kl^{\eta}}h^{d - 2 + \gamma } \exp\left(-c_0\left(\varepsilon_{n+1} \right)h^{1-\alpha}\right) \leq {K}_d\exp\left(-\frac{1}{2}c_0(\varepsilon_{n+1}) l^{\eta(1-\alpha) }\right).
\]
Combining this result with \eqref{eqn: fine ball short}, we get 
\begin{multline*}
	\mathbb{P}\left(\Over^*_{\alpha}(M_{n+1},\varepsilon_{n+1})\cap\Over^*_{\alpha}(M_{n},\varepsilon_{n})^c \cap \mathscr{E}_M(\gamma) \right) \\
	\begin{split}
		& \leq\sum_{l\geq M_{n+1}}\sum_{\|z\|=l} K_0 e^{-c_1l^{2\alpha}}
		+\sum_{l\geq M_{n+1}}\sum_{\|z\|=l} {K}_d\exp\left(-\frac{1}{2}c_0(\varepsilon_{n+1}) l^{\eta(1-\alpha) }\right) \nonumber \\
		&\leq K_d\sum_{l\geq M_{n+1}} l^{d-2}e^{-c_1l^{2\alpha}}
		+{K}_{d}\sum_{l\geq M_{n+1}} l^{d-2}\exp\left(-\frac{1}{2}c_0(\varepsilon_{n+1}) l^{\eta(1-\alpha) }\right) \\
		&\leq K_d\exp\left({-\frac{c_1M_{n+1}^{2\alpha}}{2}}\right)
		+K_{d}\exp\left(-\frac{1}{4}c_0(\varepsilon_{n+1}) M_{n+1}^{\eta(1-\alpha) }\right).
	\end{split}
\end{multline*}
Since $M\leq M_{n+1}$, it is clear that both of these terms can be bounded by $\frac{C'}{M^L}$ for some constant $C' = C(\varepsilon,n,\alpha,d,L)>0$. 
\end{proof}

\bibliographystyle{acm}
\bibliography{biblio}

\end{document}